\documentclass[11pt]{amsart}

\usepackage[draft]{changes}

\definechangesauthor[color=red]{pe}
\definechangesauthor[color=red]{xm}
\definechangesauthor[color=red]{hx}
\newcommand{\stkout}[1]{\ifmmode\text{\sout{\ensuremath{#1}}}\else\sout{#1}\fi}
\setdeletedmarkup{\stkout{#1}}
\colorlet{Changes@Color}{red}

\usepackage{amssymb}

\usepackage{graphics}
\usepackage{graphicx}

\usepackage{amsmath}
\usepackage{amsthm}
\usepackage{amsfonts}
\usepackage{mathtools}
\usepackage{xcolor}
\usepackage[english]{babel}
\usepackage[margin=1.0in]{geometry}
\usepackage[colorlinks=true,
        linkcolor=blue]{hyperref}
\usepackage{bbm}
\usepackage{verbatim}
\usepackage{extarrows}
\usepackage{blkarray}
\usepackage{tensor}

\usepackage[T1]{fontenc}

\numberwithin{equation}{section}

\newtheorem{prop}{Proposition}
\newtheorem{lemma}[prop]{Lemma}

\newtheorem{thm}[prop]{Theorem}
\newtheorem*{thm*}{Theorem}
\newtheorem{cor}[prop]{Corollary}

\numberwithin{prop}{section}

\newtheorem{defn}[prop]{Definition}
\newtheorem*{defn*}{Definition}
\theoremstyle{definition}

\newtheorem*{ex*}{Example}
\newtheorem{rmk}[prop]{Remark}

\definecolor{c1}{rgb}{0.2,0.4,0.5}
\definecolor{c2}{rgb}{0.1,0.3,0.5}
\definecolor{c3}{rgb}{0.2,0.7,0.5}
\usepackage{tikz}

\def \k {K\"ahler }
\def \ke {K\"ahler--Einstein }
\newcommand{\oo}[1]{\overline{#1}}

\newcommand{\nab}{\nabla}

\newcommand{\bC}{\mathbb{C}}

\newcommand{\dbar}{\oo\partial}

\def\Ol{\overline}
\def\ol{\overline}
\def\endpf{\hbox{\vrule height1.5ex width.5em}}
\def\e{\epsilon}
\DeclareFontFamily{U}{MnSymbolC}{}
\DeclareSymbolFont{MnSyC}{U}{MnSymbolC}{m}{n}
\DeclareFontShape{U}{MnSymbolC}{m}{n}{
	<-6>  MnSymbolC5
	<6-7>  MnSymbolC6
	<7-8>  MnSymbolC7
	<8-9>  MnSymbolC8
	<9-10> MnSymbolC9
	<10-12> MnSymbolC10
	<12->   MnSymbolC12}{}
\DeclareMathSymbol{\intprod}{\mathbin}{MnSyC}{'270}

\DeclareMathOperator{\tr}{tr}
\DeclareMathOperator{\Ric}{Ric}

\DeclareMathOperator{\Aut}{Aut}

\DeclareMathOperator{\grad}{grad}

\DeclareMathOperator{\Real}{Re}

\DeclareMathOperator{\PU}{PU}
\DeclareMathOperator{\Iso}{Iso}
\DeclareMathOperator{\Bl}{Bl}

\begin{document}

\title[]{K\"ahler-Einstein metrics and Obstruction flatness of circle bundles}

\begin{abstract} Obstruction flatness of a strongly pseudoconvex hypersurface $\Sigma$ in a complex manifold refers to the property that any (local) K\"ahler-Einstein metric on the pseudoconvex side of $\Sigma$, complete up to $\Sigma$, has a potential $-\log u$ such that $u$ is $C^\infty$-smooth up to $\Sigma$. In general, $u$ has only a finite degree of smoothness up to $\Sigma$.
In this paper, we study obstruction flatness of hypersurfaces $\Sigma$ that arise as unit circle bundles $S(L)$ of negative Hermitian line bundles $(L, h)$ over \k manifolds $(M, g).$ We prove that if $(M,g)$ has constant Ricci eigenvalues, then $S(L)$ is obstruction flat. If, in addition, all these eigenvalues are strictly less than one and $(M,g)$ is complete, then we show that the corresponding disk bundle admits a complete K\"ahler-Einstein metric. Finally, we give a necessary and sufficient condition for obstruction flatness of $S(L)$ when $(M, g)$ is a \k surface $(\dim M=2$) with constant scalar curvature.

\end{abstract}



\author [Ebenfelt]{Peter Ebenfelt}
\address{Department of Mathematics, University of California at San Diego, La Jolla, CA 92093, USA} \email{{pebenfelt@ucsd.edu}}

\author[Xiao]{Ming Xiao}
\address{Department of Mathematics, University of California at San Diego, La Jolla, CA 92093, USA}
\email{{m3xiao@ucsd.edu}}

\author [Xu]{Hang Xu}
\address{Department of Mathematics, University of California at San Diego, La Jolla, CA 92093, USA}
\email{{h9xu@ucsd.edu}}

\thanks{The first author was supported in part by the NSF grant DMS-1900955 and DMS-2154368. The second author was supported in part by the NSF grants DMS-1800549 and DMS-2045104.}

\maketitle
\section{Introduction}

On a smoothly bounded strongly pseudoconvex domain $\Omega\subset \mathbb{C}^n$, $n\geq 2$, the existence of a complete \ke metric on $\Omega$ is governed by the following Dirichlet problem:
\begin{equation}\label{Dirichlet problem}
	\begin{dcases}
	J(u):=(-1)^n \det \begin{pmatrix}
	u & u_{\oo{z_k}}\\
	u_{z_j} & u_{z_j\oo{z_k}} \\
	\end{pmatrix}=1 & \mbox{in } \Omega\\
	u=0 & \mbox{on } \partial\Omega
	\end{dcases}
\end{equation}
with $u>0$ in $\Omega$. The equation $J(u)=1$ is often referred to as Fefferman's complex Monge-Amp\`ere equation. If $u$ is a solution of \eqref{Dirichlet problem}, then $-\log u$ is the \k potential of a complete \ke metric on $\Omega$.
Fefferman \cite{Fe2} established the existence of an approximate solution $\rho\in C^{\infty}(\oo{\Omega})$ of \eqref{Dirichlet problem} that only satisfies $J(\rho)=1+O(\rho^{n+1})$, and showed that such a $\rho$ is unique modulo $O(\rho^{n+2})$. Such an approximate solution $\rho$ is called a {\em Fefferman defining function}. Cheng and Yau \cite{CheYau} then proved the existence and uniqueness of an exact solution $u\in C^{\infty}(\Omega)$ to \eqref{Dirichlet problem}, which is now called the Cheng--Yau solution. Lee and Melrose \cite{LeeMel82} established that the Cheng--Yau solution has the following asymptotic expansion:
\begin{equation}\label{Lee-Melrose expansion}
	u\sim \rho\sum_{k=0}^{\infty} \eta_k\bigl(\rho^{n+1}\log\rho\bigr)^k
\end{equation}
where each $\eta_k\in C^{\infty}(\oo{\Omega})$ and $\rho$ is a Fefferman defining function.

It follows from \eqref{Lee-Melrose expansion} that, in general, the Cheng--Yau solution $u$ can only be expected to possess a finite degree of boundary smoothness; namely, $u\in C^{n+2-\varepsilon}(\oo{\Omega})$ for any $\varepsilon>0$. Graham \cite{Graham1987a} discovered that in the expansion (\ref{Lee-Melrose expansion}), the restriction of $\eta_1$ to the boundary, $\eta_1|_{\partial\Omega}$, turns out to be precisely the obstruction to $C^{\infty}$ boundary regularity of the Cheng--Yau solution. More precisely, in \cite{Graham1987a} Graham proved that if $\eta_1|_{\partial\Omega}$ vanishes identically (on $\partial\Omega$), then every $\eta_k$ vanishes to infinite order on $\partial\Omega$ for all $k\geq 1$. For this reason, $\eta_1|_{\partial\Omega}$ is called the \emph{obstruction function}.
Graham also showed ({\em op.\ cit.}) that, for any $k\geq 1$, the coefficients $\eta_k$ mod $O(\rho^{n+1})$ are independent of the choice of Fefferman defining function $\rho$ and are locally uniquely determined by the local CR geometry of $\partial\Omega$.
As a consequence, the $\eta_k$ mod $O(\rho^{n+1})$, for $k\geq 1$, are local CR invariants that can be defined on any strongly pseudoconvex CR hypersurface in a complex manifold. In particular, the obstruction function $\mathcal{O}=\eta_1|_{\partial\Omega}$ is a local CR invariant and it can be defined on any strongly pseudoconvex CR hypersurface $\Sigma$ (see more details $\S$ \ref{Sec 2.1}). If $\Sigma$ is a CR hypersurface for which the obstruction function $\mathcal{O}$ vanishes identically, then $\Sigma$ is said to be \emph{obstruction flat}. The most basic examples of obstruction flat hypersurfaces are the sphere $\{z\in \mathbb{C}^n: |z|=1 \}$ and, more generally, any spherical CR hypersurface; recall that a CR hypersurface $\Sigma$ is called {\em spherical} if, at every $p\in \Sigma$, there is a neighborhood of $p$ that is CR diffeomorphic to an open piece of the sphere. On the other hand, another result of Graham ({\em op.\ cit.}) shows that there are also plenty of non-spherical CR hypersurfaces that are obstruction flat.
The construction of such examples, however, is local and amounts to solving a Cauchy problem. Two questions then arise naturally: (Q1) Is there a more constructive and natural way to obtain obstruction flat CR hypersurfaces, and (Q2) is it possible to characterize {\em compact}, obstruction flat hypersurfaces in terms of more classical notions of CR geometry and complex geometry/several complex variables?

For question (Q2), for example, ample evidence supports the conjecture that if a $3$-dimensional strongly pseudoconvex CR hypersurface $\Sigma$ is obstruction flat and compact, then $\Sigma$ is spherical. In particular, Curry and the first author confirmed this conjecture for smooth boundaries of bounded, strongly pseudoconvex domains in $\bC^2$, subject to the existence of a holomorphic vector field satisfying a mild approximate tangency condition along the boundary $\Sigma$. The same authors also confirmed the conjecture for arbitrary small deformations, not necessarily embeddable, of the standard CR structure of the unit sphere in $\bC^2$.
We refer readers to \cite{CE1, CE2} for precise statements and related results.


Much less is known about higher dimensional obstruction flat strongly pseudoconvex hypersurfaces. Recently, Takeuchi in \cite{Ta18} proved that a Sasakian $\eta$-Einstein manifold (see \cite{Ta18} for the definition) must be obstruction flat. Hirachi announced the result that if the boundary of a smoothly bounded strongly pseudoconvex domain $\Omega\subset\mathbb{C}^n$, $n\geq 3$, is obstruction flat and is sufficiently close to the unit sphere, then $\Omega$ is biholomorphic to the unit ball (in his talk at the conference on `Symmetry and Geometric Structures' at IMPAN, Warsaw, November 12-18, 2017).

In this paper, we shall consider the special case where the CR hypersurface arises as the unit circle bundle of a negative Hermitian line bundle over a \k manifold. Let $(L,h)$ be a negative line bundle over a complex manifold $M$, so that the dual bundle $(L^*, h^{-1})$ induces a K\"ahler metric $g$ on $M$. By a well-known observation of Grauert, the corresponding circle bundle $S(L)=\{v \in L: |v|_h=1 \}$ is strongly pseudoconvex; here, $|v|_h$ denotes the norm of $v$ with respect to the metric $h$. One of our goals is to characterize obstruction flatness of the circle bundle $S(L)$ in terms of the \k geometry of $(M,g)$.

To formulate our results, we first recall some standard notions in \k geometry. Let $(M, g)$ be an $n$-dimensional \k manifold  and let $\mathrm{Ric}=-i\partial \Ol{\partial} \log \det(g)$ denote the associated Ricci tensor.  The latter naturally induces an endomorphism, the {\em Ricci endomorphism}, of the holomorphic tangent space $T_p^{1,0}M$ given by $\mathrm{Ric} \cdot g^{-1}$ for $p\in M$.
The eigenvalues of this endomorphism will be referred to as
the {\em Ricci eigenvalues} of $(M,g)$ and, by design, depends on $p\in M$. 
All Ricci eigenvalues are real-valued as both  $\mathrm{Ric}$ and $g$ are Hermitian tensors.
For a fixed $p$, we label the Ricci eigenvalues such that $\lambda_1(p) \leq \cdots \leq \lambda_n (p)$.
Note that the sum of the $\lambda_i(p)$, i.e., the trace of the Ricci endomorphism, gives the scalar curvature at $p$. The product of the $\lambda_i(p)$, i.e., the determinant of the Ricci endomorphism, gives the so-called central curvature at $p$ (cf. \cite{Ma03}). The \k manifold $(M, g)$ is said to have {\em constant Ricci eigenvalues}, if each $\lambda_i(p)$, for $1 \leq i \leq n$, is a constant function on $M$; equivalently, the characteristic polynomial of the Ricci endomorphism, $\mathrm{Ric} \cdot g^{-1}: T_p^{1,0}M \rightarrow T_p^{1,0}M$ is the same at every point $p \in M$.

A \k manifold with constant Ricci eigenvalues must have constant scalar and central curvatures. Typical examples of \k manifolds with  constant Ricci eigenvalues include homogeneous spaces and \ke manifolds, and more generally products of \ke manifolds. We emphasize, however, that there are also many examples of  \k manifolds with constant Ricci eigenvalues, which cannot be expressed as a product of \ke manifolds. The following is such an example:
\begin{ex*}[Siegel-Jacobi space]
	Set $M=\mathbb{C}\times \mathbb{H}=\{(z_1, z_2)\in \mathbb{C}^2: \Real z_2>0 \}$. Let $\phi: M \rightarrow \mathbb{R}$ be
	
	\begin{equation*}
	\phi(z_1, z_2)=\frac{(z_1+\oo{z_1})^2}{2(z_2+\oo{z_2})}-\log(z_2+\oo{z_2}).
	\end{equation*}
	
	One can verify that $\phi$ is strongly plurisubharmonic and thus $\omega=i\partial\dbar\phi$ induces a \k metric $g$ on $M$. The \k manifold $(M, g)$ is called the Siegel-Jacobi space. It is a complete homogeneous \k surface (see \cite{Yang15}).
	A routine computation yields that
	the Ricci endomorphism $T^{1,0}M\rightarrow T^{1,0}M$ is given by
	\begin{equation}\label{Ricci operator Siegel-Jacobi}
	\Ric \cdot g^{-1}=\begin{pmatrix}
	0 & 0 \\
	-3\frac{z_1+\oo{z_1}}{z_2+\oo{z_2}} & -3\\
	\end{pmatrix}.
	\end{equation}
	The Ricci eigenvalues are $0$ and $-3$, which are obviously constants. However, the Siegel-Jacobi space cannot be locally holomorphically isometric to a product of \ke manifolds (Riemann surfaces in this case)  because if it were, then every entry in the matrix \eqref{Ricci operator Siegel-Jacobi} would be holomorphic, as is clearly not the case.
	
\end{ex*}

Our first result asserts that $S(L)$ is obstruction flat if $(M, g)$ has constant Ricci eigenvalues. 

\begin{thm}\label{thm 1}
Let $(L,h)$ be a negative line bundle over a complex manifold $M$, so that the dual bundle $(L^*, h^{-1})$ of $(L, h)$ induces a K\"ahler metric $g$ on $M$ (not necessarily complete). 
If $(M, g)$ has constant Ricci eigenvalues, then the circle bundle $S(L)$ is obstruction flat.
\end{thm}

\begin{rmk}\label{rmk12}
A standard result in \k geometry states that a \k metric with constant scalar curvature must be real analytic. Consequently, the circle bundle $S(L)$ in Theorem \ref{thm 1} is a real analytic CR hypersurface in the ambient space $L$. We also remark that a special case of Theorem \ref{thm 1}  occurs when the manifold $(M, g)$ is K\"ahler-Einstein. In this case the circle bundle $S(L)$ becomes a Sasakian $\eta$-Einstein manifold and Theorem \ref{thm 1} reduces to  Proposition 3.5 of \cite{Ta18}.
\end{rmk}

An interesting problem (as alluded to above) is to find obstruction flat CR hypersurfaces that are not spherical. Combining Theorem \ref{thm 1} with the work of Webster \cite{Web77b} and Bryant \cite{Bry01} (see also Wang \cite{Wang}), we obtain the following corollary, which provides a way to construct many such CR hypersurfaces as circle bundles; $\omega_c$ in the statement of the corollary denotes the \k metric with constant holomorphic sectional curvature $c$.

\begin{cor}\label{cor 1}
Let $(L, h)$ and $(M, g)$ be as in Theorem $\ref{thm 1}$ and suppose that $(M, g)$ has  constant Ricci eigenvalues. Then the circle bundle $S(L)$ of $L$ is obstruction flat. Furthermore, $S(L)$ is spherical if and only if 
$(M,g)$ is locally holomorphically isometric to one of the following:
	\begin{itemize}
		\item [(1)] $(\mathbb{B}^n, \lambda \,\omega_{-1})$ for some $\lambda\in \mathbb{R}^+$,
		\item [(2)] $(\mathbb{CP}^n, \lambda\, \omega_1)$ for some $\lambda\in \mathbb{R}^+$,
		\item [(3)] $(\mathbb{C}^n, \omega_0)$,
		\item [(4)] $(\mathbb{B}^l\times \mathbb{CP}^{n-l}, \lambda \omega_{-1}\times \lambda\omega_1)$ for some $1\leq l\leq n-1$ and some $\lambda\in\mathbb{R}^+$.
	\end{itemize}
\end{cor}

We observe that Corollary \ref{cor 1} allows us to easily construct examples of compact, obstruction flat and non-spherical CR hypersurfaces as circle bundles $S(L)$ (with $\dim S(L) \geq 5$); take as the base any compact \k manifold $(M,g)$ with constant Ricci eigenvalues, but not locally holomorphically isometric to any one in the list (1)-(4) above. Such $(M,g)$ are clearly abundant, for instance, compact homogeneous Hodge manifolds (in particular compact Hermitian symmetric spaces) other than $\mathbb{CP}^n$. This observation indicates that question (Q2) above is non-trivial even for $S(L)$ when $n=\dim M\geq 2$, in the situation considered in this paper. When $n=1$, obstruction flatness of $S(L)$ implies that it is also spherical by a result proved (in a different context) by the first author \cite{Ebenfelt18}, since $S(L)$ always has a transverse $S^1$-symmetry.

As we have noted, the notion of obstruction flatness is a local CR invariant, but is originally related to complete \ke metrics on domains. It is natural to ask whether a complete \ke metric exits on the corresponding disk bundle $D(L):=\{w \in L: |w|_h <1 \}$ of $(L, h)$ in the situations we are considering.  We prove the following result:

\begin{thm}\label{thm 2}
Let $m=n+1 \geq 2.$ Let $M$ be a complex manifold of dimension $n$ and $(L,h)$ a negative line bundle over $M$ such that the dual bundle $(L^*, h^{-1})$ of $(L, h)$ induces a complete K\"ahler  metric $g$ on $M$. If $(M, g)$ has constant Ricci  eigenvalues
and every Ricci eigenvalue is strictly less than one, then the disk bundle $D(L)=\{w \in L: |w|_h <1 \}$ admits a unique complete \ke metric $\widetilde{g}$ with Ricci curvature equals to $-(m+1)$.
Moreover, this metric is induced by the following K\"ahler form:
\begin{equation}\label{eqnmet}
\widetilde{\omega} (w, \ol{w})=-\frac{1}{m+1} \pi^*(\Ric)|_{w}+\frac{1}{m+1} \pi^*(\omega)|_{w} -i \partial \Ol{\partial} \log \phi(|w|_h),
\end{equation}
where $\omega$ and $\Ric$ are respectively the \k and the Ricci form of $(M,g)$, $\pi: L \rightarrow M$ the canonical fiber projection of the line bundle, and $\phi: (-1, 1) \rightarrow \mathbb{R}^+$ an even real analytic function that depends only on the characteristic polynomial of the Ricci endomorphism.
\end{thm}
\begin{rmk} The even real analytic function $\phi: (-1, 1) \rightarrow \mathbb{R}^+$ in Theorem \ref{thm 2} is constructed in Proposition \ref{prpnzh2} below.
\end{rmk}
\begin{rmk}
The special case of Theorem \ref{thm 2} when $(M, g)$ is \ke  was already proved by Calabi \cite{Ca79} and Tsuji \cite{Tsu88} (see also \cite{Ban90}). However, our treatment is different from these classical approaches, as the methods used there do not reflect the boundary regularity of the solution.
\end{rmk}

\begin{rmk}
If the complex manifold $M$ in Theorem \ref{thm 2} is realized as a domain in $\mathbb{C}^n$ and $L$ is the trivial line bundle, then the disk bundle $D(L)$ becomes a domain in $\mathbb{C}^{n+1}.$ In this case, the Cheng--Yau solution of $D(L)$, which satisfies (\ref{Dirichlet problem}) with $\Omega=D(L)$, extends real analytically across the circle bundle $S(L).$ This can be seen easily from the proof of Theorem \ref{thm 2}.
\end{rmk}

Theorem \ref{thm 2} is optimal in the sense that the conclusion fails, in general, if some Ricci eigenvalue is greater than or equal to $1$, as is shown by the following proposition.

\begin{prop}\label{prpn17}
Let $\pi: (L, h)\rightarrow \mathbb{CP}^n$ be a negative line bundle over $\mathbb{CP}^n$. Assume that the metric $g$ on $\mathbb{CP}^n$ induced by  the dual bundle of $L$ satisfies the \ke equation $\Ric(g)=\lambda g$ with $\lambda \in \mathbb{R}$.
Then the disk bundle $D(L)=\{ v\in L: |v|_h^2<1 \}$  admits a complete \ke metric with negative Ricci curvature if and only if $\lambda<1$.
\end{prop}
When the base manifold $M$ has a high degree of symmetry (e.g, is homogeneous), one can use the automorphisms of $M$ to reduce the complex Monge-Amp\`ere equation (\ref{Dirichlet problem}) to an ODE. This idea seems to first have been used by Bland \cite{Bla86} and subsequently by many others; see for example \cite{WaYiZhaRo} and references therein. The novelty of our proof (of Theorem \ref{thm 1} and \ref{thm 2}) consists of applying an ODE trick in a different way so that it works on a much more general class of base manifolds $M$ whose automorphism groups can be trivial. While one cannot expect to explicitly solve the resulting ODE in general, we investigate the rationality of its solution in a particular case of interest; namely, when the underlying \k manifold is a domain of holomorphy. As a consequence, we obtain the following corollary, characterizing the unit ball among egg domains in terms of the Bergman-Einstein condition. A well-known conjecture posed by Yau \cite{Yau16} asserts that if the Bergman metric of a bounded pseudoconvex domain is K\"ahler--Einstein, then the domain must be homogeneous. The following result provides affirmative support for this conjecture in a special class of  domains; we note that some (most) domains in this class are not smoothly bounded.
\begin{prop}\label{prpn1d10}
		Let $p$ be a positive real number and $E_p \subset \mathbb{C}^{n+1}$, $n \geq 1,$  the egg domain:
		$$E_p=\{W=(z, \xi) \in \mathbb{C}^n \times \mathbb{C}: |z|^2 +|\xi|^{2p} <1  \}.$$
		Then the Bergman metric of $E_p$ is K\"ahler-Einstein if and only if $p=1$.
\end{prop}

\begin{rmk}
	When $n=1$ in the above proposition, the result was proved by Cho in \cite{Cho21} (see Proposition 12 there) via an explicit (and rather lengthy) computation. When $n=1$ and $ p \in \mathbb{Z}^+$ (so that $E_p$ is smoothly bounded), the result also follows from Fu--Wong \cite{FuWo} (see Section 3 there).
\end{rmk}

When $M$ is a complex manifold and $g$ a \k metric on $M$ (which is not necessarily complete), then for any $p\in M$, the \k condition implies that there exist a local coordinate chart $(U,z)$ near $p$, a trivial line bundle $L=U\times \mathbb{C}$, and a Hermitian metric $h$ on $L$ such that $(L, h)$ is negative and $g$ is induced by the dual of $(L, h)$ on $U$ via $\omega_g=i\partial\dbar\log h$, where $\omega_g$ is the \k form of $g$. We define in the usual way the circle bundle
\begin{equation*}
	S(L)=S(L,h):=\{(z,\xi)\in U\times\mathbb{C}: |\xi|^2h(z,\bar{z})=1 \}.
\end{equation*}
Note that while the choice of bundle metric $h$ inducing the K\"ahler metric $g$ is not unique, any other choice $\widetilde{h}$ must be such that $\log \widetilde{h}$ and $\log h$ differ by a pluriharmonic function near $p$. Consequently, $h=\widetilde{h} |e^{f}|^2$ for some holomorphic function $f$ near $p$. This implies that, by shrinking $U$ if needed, the circle bundles $S(L,h)$ and $S(L,\widetilde{h})$ are CR diffeomorphic via the map $(z, \xi) \rightarrow (z, e^f \xi)$. With this observation and motivated by Theorem \ref{thm 1}, we introduce the following terminology.
\begin{defn} {\rm 
The \k manifold $(M, g)$, or simply the metric $g$, is said to be {\em obstruction flat} (respectively, {\em spherical}) at $p \in M$, if the circle bundle $S(L)$ (introduced above) is obstruction flat (respectively, spherical) over some neighborhood $U$ of $p$. Moreover, $(M, g)$, or simply $g$, is said to be {\em obstruction flat} (respectively, {\em spherical}) if it is so at every $p\in M$.}
\end{defn}

With this terminology, Theorem \ref{thm 1} can be reformulated as follows: {\em If a \k manifold $(M, g)$ has constant Ricci eigenvalues, then $g$ is obstruction flat.}

When $(M,g)$ is a Riemann surface (i.e., $n=1$) and $g$ has constant Gaussian curvature, then locally $(M,g)$ is a space form and, consequently, $g$ is spherical and therefore obstruction flat. As indicated above, the converse holds for compact Riemann surfaces by a result of the first author \cite{Ebenfelt18}: If $g$ is obstruction flat, then it is spherical and, hence, $(M,g)$ is locally a space form. It is natural to ask about the higher dimensional case with Gaussian curvature replaced by scalar curvature.
As a consequence of the work of Webster \cite{Web77b}, Bryant \cite{Bry01} and Wang \cite{Wang}, the only constant scalar curvature \k manifolds (cscK manifolds) that are spherical are those that are locally holomorphically isometric to the spaces listed in (1)-(4) in Corollary \ref{cor 1}. We summarize this characterization of spherical metrics in the following proposition:

\begin{prop}\label{prpnsp}
Let $(M, g)$ be a \k manifold of complex dimension at least two. Assume that either $M$ is compact or $(M, g)$ has constant scalar curvature. Then $(M, g)$ is spherical if  and only if $(M, g)$ is locally holomorphically isometric to one of (1)-(4) in Corollary $\ref{cor 1}$.
\end{prop}

Much less (in fact, essentially nothing) is known about necessary conditions for a cscK manifold to be obstruction flat in higher dimensions. We shall give a complete answer to this question in the \k surface case ($n=\dim M=2$). Recall that the central curvature is defined as the determinant of the Ricci endomorphism and note that when $n=2$, the scalar curvature and the central curvature determine completely both the Ricci eigenvalues.

\begin{thm}\label{thm 3}
	Let $(M,g)$ be a \k surface with constant scalar curvature and central curvature $C: M\rightarrow \mathbb{R}$. The following are equivalent:
	\begin{itemize}
		\item [(1)] The metric $g$ is obstruction flat.
		\item [(2)] $L C:=\nab^{\bar{\alpha}}\nab^{\bar{\beta}}\nab_{\bar{\beta}}\nab_{\bar{\alpha}}C=0$ on $M$, where $\nab$ is the Levi-Civita connection of $g$.
	\end{itemize}
	If, in addition, $M$ is compact, then (1) and (2) are also equivalent to:
	\begin{itemize}
		\item [(3)] $\nab_{\bar{\beta}}\nab_{\bar{\alpha}} C=0$, or equivalently, the central vector field $\grad^{1,0}C$ is a holomorphic vector field on $M$.
	\end{itemize}
\end{thm}

\begin{rmk}\label{Lichnerowicz operator rmk}
	The fourth order differential operator $L$ is the well-known Lichnerowicz operator, which plays a fundamental role in the study of extremal metrics. Indeed, on a compact manifold a \k metric is extremal if and only if $L$ annihilates the scalar curvature (see \cite{Ga14} for more details on this topic). When the scalar curvature is constant, we can also express this operator as
	\begin{equation*}
	L \Phi=\Delta^2 \Phi +R^{\alpha\bar{\beta}}\nab_{\alpha}\nab_{\bar{\beta}} \Phi, ~\Phi \in C^{\infty}(M)
	\end{equation*}
	where $\Delta=g^{\alpha\bar{\beta}}\nab_{\alpha}\nab_{\bar{\beta}}$ is the complex Laplacian on $(M, g)$ and $R^{\alpha\bar{\beta}}=g^{\alpha\bar\mu}g^{\nu\bar\beta}R_{\nu\bar\mu}$ is the Ricci curvature with indices raised by the metric in the usual way.
	
In Theorem \ref{thm 3}, if we assume that $M$ is compact and, additionally, that $M$ admits no nontrivial holomorphic vector fields with zeros, then by (3), $C$ must be constant. Combining this with the constant scalar curvature assumption, we see that $(M, g)$ has constant Ricci eigenvalues. Thus, in this situation, the converse of Theorem \ref{thm 1} holds.
\end{rmk}

Theorem \ref{thm 1}, as well as the work of Lee \cite{Lee86} and Graham-Hirachi \cite{GrHi05, GrHi08} (see also Gover-Peterson \cite{GoPe06}), will play a fundamental role in the proof of Theorem \ref{thm 3}.
As an application of Theorem \ref{thm 3}, we provide what seems to be the first collection of examples of cscK manifolds that are not obstruction flat.
Noncompact examples can be found explicitly, while compact examples seem more complicated to construct. For the following proposition, recall that the Burns-Simanca metric on $\mathbb{C}^2\setminus\{0\}$ is given by the \k form $\omega=i\partial\dbar(|z|^2+\log |z|^2 )$, where $z=(z_1, z_2)$. It extends to a complete \k metric on $\Bl_0\mathbb{C}^2$ (see subsection 8.1.2 in \cite{Ga14} for more details), where $\Bl_0(\mathbb{C}^2)$ denotes the space obtained by blowing up the origin in $\mathbb{C}^2$. The so extended metric is scalar flat (cscK with scalar curvature identically zero) on $\Bl_0(\mathbb{C}^2)$.

\begin{prop}\label{prop two examples} The following hold:
\vspace {-.2 cm}
	
	\begin{enumerate}
		\item The Burns-Simanca metric is not obstruction flat.
		
		\item Let $k\geq 14$. The complex projective space $\mathbb{CP}^2$, blown up $k$ suitably chosen points, admits a scalar flat \k metric that is not obstruction flat.
	\end{enumerate}
\end{prop}

\begin{rmk}
For the Burns-Simanca metric, with $C$ denoting the central curvature as above, we actually prove that in  $\mathbb{C}^2\setminus\{0\}$, $L C$ equals $0$ only on two concentric spheres. (The obstruction non-flatness then follows from Theorem \ref{thm 3}.)
\end{rmk}

The paper is organized as follows.  In $\S$ \ref{Sec 2}, we will prove Theorem \ref{thm 1} and Theorem \ref{thm 2}, as well as Corollary \ref{cor 1} and Propositions \ref{prpn17}, \ref{prpn1d10} and \ref{prpnsp}. In $\S$ \ref{Sec 3}, we first recall some preliminary materials on the CR invariant theory and pseudohermitian geometry which will be important for the proofs. Then Theorem \ref{thm 3} and Proposition \ref{prop two examples} will be established, except that some detailed computations needed in the proofs will be left to the Appendix.

\textbf{Acknowledgement.} The authors would like to thank Sean Curry for helpful discussions on the topics in this paper.

\section{Proof of Theorem \ref{thm 1} and \ref{thm 2}}\label{Sec 2}
We will prove Theorem \ref{thm 1}, Proposition \ref{prpnsp}   in $\S$ \ref{Sec 2.1}. Corollary \ref{cor 1} will then follow from the two results. In $\S$ \ref{Sec 2.2},  we establish Theorem \ref{thm 2} and Proposition \ref{prpn17}. An ODE (see Proposition \ref{prpnzh1} and \ref{prpnzh2}) will play an important role in the proofs. In $\S$ \ref{Sec 2.3},  we will study  the rationality of the solution to this ODE when the underlying manifold is a domain of holomorphy. We then use the rationality result to prove Proposition \ref{prpn1d10}.
\subsection{Proof of Theorem \ref{thm 1}}\label{Sec 2.1}
We first recall the following work of Graham \cite{Graham1987a}, which will be the starting point of our proof.
Let $\Sigma$ be a piece of smooth strongly pseudoconvex hypersurface in $\mathbb{C}^n, n \geq 2,$ and let $\Omega$ be a one-sided neighborhood of $\Sigma$ on its strongly pseudoconvex side, such that $\overline{\Omega}:=\Omega \cup \Sigma$ becomes a manifold with boundary. Graham \cite{Graham1987a} showed that, given any $a \in C^{\infty}(\Sigma)$ and $\rho$ a Fefferman defining function of $\Sigma,$ there exists an asymptotic formal solution $u$ in the form of (\ref{Lee-Melrose expansion}), with $\eta_k \in C^{\infty}(\overline{\Omega}),$ solving $J(u)=1$ to the infinite order along $\Sigma$ for which $\frac{\eta_0-1}{\rho^{n+1}}=a$. (Any formal solution of the form (\ref{Lee-Melrose expansion}) to $J(u)=1$ necessarily satisfies $\frac{\eta_0-1}{\rho^{n+1}} \in C^{\infty}(\Sigma)$).  Such formal solution $u$ is unique in the sense that two choices of each $\eta_k$ agree to infinite order along $\Sigma.$ Furthermore, Graham proved that for each $k \geq 1, \eta_k ~\mathrm{mod}~ O(\rho^{n+1})$ is independent of the choice of $a$ and the choice of the Fefferman defining function $\rho$. The restriction $\eta_1|_{\Sigma}$ gives precisely the obstruction function on $\Sigma.$ Graham in addition showed that if $\eta_1|_{\Sigma} \equiv 0$ for some formal solution of form (\ref{Lee-Melrose expansion}) to $J(u)=1$,  then every formal solution $u$ must satisfy $\eta_k=0$ to infinite order along $\Sigma$ for all $k \geq 1.$

Therefore, if there is a function $u_0 \in  C^{\infty}(\overline{\Omega})$ such that $u_0=0$ on $\Sigma$ and $J(u_0)=1$ in $\Omega,$ then one sees $\Sigma$ must be obstruction flat, by writing $u_0$ into the form of (\ref{Lee-Melrose expansion}). To establish Theorem \ref{thm 1}, we will construct such a function $u_0$ locally for $S(L).$

We split the proof of Theorem \ref{thm 1} into propositions \ref{prpnzh1} and \ref{prpn3d5}. Our proof, especially Lemma \ref{lemma3d6}, is inspired by the work of Bland \cite{Bla86}.

\begin{prop}\label{prpnzh1}
	Let $P(y)$ be a monic polynomial in $y \in \mathbb{R}$ of degree $n=m-1 \geq 1.$ Let $Q(y)$ be a polynomial satisfying $\frac{dQ}{dy}=(m+1)yP(y)$ (thus $Q$ is a monic polynomial of degree $m+1$ and is unique up to a constant addition). Let $\hat{P}$ and $\hat{Q}$ be polynomials satisfying
	$$\hat{P}(x)=x^{m-1}P(x^{-1}), \quad \hat{Q}(x)=x^{m+1} Q(x^{-1}).$$
	Then  $\hat{P}(0)=1, \hat{Q}(0)=1.$
	Let $I \subset (0, \infty)$ be an open interval containing $r=1$ and  $Z(r)$ a real analytic function in $I$ satisfying the following conditions:
	\begin{equation}\label{eqnzodenz1}
	rZ'\hat{P}(Z)+\hat{Q}(Z)=0~\text{on}~I, \quad Z(1)=0.
	\end{equation}
	\begin{equation}\label{eqnzivnz}
	Z'(r)<0~\text{and}~\hat{P}(Z)>0~\text{on}~I, ~\text{and}~Z(r)>0 ~\text{on}~I_0:=I \cap (0,1).
	\end{equation}
	Set $\phi(r):=2 (\frac{r}{-Z'\hat{P}(Z)})^{\frac{1}{m+1}}\, Z$ on $I$.  Then $\phi$ is real analytic on $I,~\phi(1)=0,$ and $\phi>0$ on $I_0.$ Moreover, $\phi$ satisfies that $(m+1)rZ\phi'+(m+1-2Z)\phi=0$ on $I$, and $\phi'(1)=-2.$
\end{prop}

\begin{rmk}
	It follows easily from the assumption and elementary ODE theory that (\ref{eqnzodenz1}) has a real analytic solution $Z$ in some open interval $I$ containing $1$. Since $\hat{P}(0)=1, \hat{Q}(0)=1,$ we see (\ref{eqnzodenz1}) implies $Z'(1)=-1$. Then by shrinking $I$ if necessary, we can assume (\ref{eqnzivnz}) holds.
\end{rmk}

{\em Proof of Proposition \ref{prpnzh1}.} It is clear that $\phi$ is real analytic on $I, ~\phi(1)=0,$ and $\phi>0$ on $I_0$ by the definition of $\phi$ and the assumption of $Z$. We only need to prove the last assertion in Proposition \ref{prpnzh1}. For that, we first establish the following two lemmas.

\begin{lemma}\label{lemma3d3}
	Let $Z$ be as in Proposition \ref{prpnzh1}. Then we have
	$$r(Z'Z^{-(m+1)}\hat{P}(Z))'=(m+1)Z'Z^{-(m+2)}\hat{P}(Z)-Z'Z^{-(m+1)}\hat{P}(Z)~\text{on}~I_0.$$
	Equivalently,
	$$r\frac{(Z'Z^{-(m+1)} \hat{P}(Z))'}{Z'Z^{-(m+1)}\hat{P}(Z)}=-1+(m+1)Z^{-1}~\text{on}~I_0.$$
\end{lemma}

\begin{proof}
	We first observe by the definition of $\hat{Q}$ and $Q$, we have $\frac{d (x^{-(m+1)}\hat{Q}(x))}{dx}=-(m+1)x^{-3}P(x^{-1}).$
	Dividing the ODE in (\ref{eqnzodenz1}) by $Z^{m+1}$ , we have
	$$rZ'Z^{-(m+1)}\hat{P}(Z)+ Z^{-(m+1)}\hat{Q}(Z)=0~\text{on}~I_0.$$
	We then differentiate the above equation with respect to $r$ to obtain
	\begin{equation}\label{eqn03}
	r(Z'Z^{-(m+1)}\hat{P}(Z))'+ Z'Z^{-(m+1)}\hat{P}(Z) -(m+1)Z'Z^{-3}P(Z^{-1})=0.
	\end{equation}
	By the definition of $\hat{P}$, we see $Z^{-3}P(Z^{-1})=Z^{-(m+2)} \hat{P}(Z).$
	Then (\ref{eqn03}) yields Lemma \ref{lemma3d3}.
\end{proof}

\begin{lemma}\label{lemma3d4}
	Let $\phi, Z$ be as in Proposition \ref{prpnzh1}. Then we have
	$$r\frac{(Z'Z^{-(m+1)}\hat{P}(Z))'}{Z'Z^{-(m+1)}\hat{P}(Z)}=1-(m+1)r\frac{\phi'}{\phi}~\text{on}~I_0.$$
\end{lemma}

\begin{proof}
	By the definition of $\phi$, we have
	$$\phi^{m+1}=2^{m+1}\frac{r Z^{m+1}}{-Z' \hat{P}(Z)}~\text{on}~I_0.$$
	Equivalently,
	\begin{equation}\label{eqn04}
	-Z' Z^{-(m+1)}\hat{P}(Z)=2^{m+1} r \phi^{-(m+1)}~\text{on}~I_0.
	\end{equation}
	We now take the logarithmic derivative of both sides of (\ref{eqn04}) with respect to $r$, and then multiply by $r$ to obtain the desired equation in Lemma \ref{lemma3d4}.
\end{proof}

Finally we compare (the second equation in) Lemma \ref{lemma3d3} and Lemma \ref{lemma3d4} to obtain $(m+1)rZ\phi'+(m+1-2Z)\phi=0$ on $I_0$. By analyticity, it holds on $I.$ Recall
$Z(1)=0, \hat{P}(0)=1,$ and $Z'(1)=-1$. It then follows from the definition of $\phi$ that $\phi'(1)=-2.$ This finishes the proof of Proposition \ref{prpnzh1}. \qed


Let $(M, g)$ and $(L, h)$ be as in  Theorem \ref{thm 1}. Write $n$ for the complex dimension of $M$ and write $m=n+1.$
Choose a coordinate chart $(D, z)$ of $M$ with a frame $e_L$ of $L$ over $D$. Writing $\pi: L \rightarrow M$ for the canonical projection, we have $$\pi^{-1}(D)=\{\xi e_L(z): (z, \xi) \in D \times \mathbb{C} \}.$$
Under this trivialization,  $S(L)$ can be written as follows locally over $D,$
$$\Sigma=\{(z, \xi) \in D \times \mathbb{C}: |\xi|^2 H(z, \Ol{z})=1 \}.$$
Here $H(z, \Ol{z})=h(e_L(z), \Ol{e_L(z)}).$
In the local coordinates $z=(z_1, \cdots, z_n)$, we write $g=\sum g_{i\Ol{j}} dz_i \otimes d\Ol{z_j}$ on $D$. By the definition of $g$, we have $g_{i \Ol{j}}=\frac{\partial^2 \log H}{\partial z_i \partial \Ol{z_j}}.$ Write $G(z)=G(z, \Ol{z})=\det (g_{i\Ol{j}})>0.$

Denote by $I_n$ the $n \times n$ identity matrix. Write $P(y, p)$ for the characteristic polynomial of the linear operator $\frac{2}{m+1} \mathrm{Ric} \cdot g^{-1}: T_{p}^{1,0}M \rightarrow T_{p}^{1,0}M.$ That is,  $P(y, p)=\det(yI_n-\frac{2}{m+1} \mathrm{Ric} \cdot g^{-1})$. In the local coordinates, writing Ricci tensor as $\mathrm{Ric}=(R_{i \Ol{k}})_{1 \leq i, k \leq n}=-(\frac{\partial^2 \log G}{\partial z_i \partial \Ol{z_k}})_{1 \leq i, k \leq n},$ we have $P(y, p)$ is the determinant of the $n \times n$ matrix $(y\delta_{ij}-\frac{2}{m+1}R_{i \Ol{k}} \cdot g^{j \Ol{k}}(p))$. Here the indices $i, j, k$ run over $\{1, 2, \cdots , n\}.$

By the constant Ricci eigenvalue assumption, 
$P(y,p)$ does not depend on $p$. We will therefore just write it as $P(y).$
It is clear that $P(y)$ is a monic polynomial in $y$ of degree $n.$ We apply Proposition \ref{prpnzh1} to this polynomial $P(y)$, then we get  polynomials $Q(y), \hat{P}(x), \hat{Q}(x)$, as well as real analytic functions $Z(r), \phi(r)$ in some interval $I$ containing $r=1$
as in Proposition \ref{prpnzh1}. Write $y(r):=\frac{1}{Z(r)}$ for $r \in I_0.$ By Proposition \ref{prpnzh1}, $(m+1)rZ\phi'+(m+1-2Z)\phi=0$ on $I$. It then follows that
\begin{equation}\label{eqn3d5}
y(r)=\frac{2\phi(r)-(m+1)r\phi'(r)}{(m+1)\phi(r)}=\frac{2}{m+1}-r\frac{\phi'}{\phi}.
\end{equation}

Theorem \ref{thm 1} will follow from the next proposition. 
\begin{prop}\label{prpn3d5}
	Let
	$$U=\{W=(z, \xi) \in \mathbb{C}^{m}: z \in D, |\xi| H(z)^{\frac{1}{2}} \in I \}, \quad U_0=\{W=(z, \xi) \in \mathbb{C}^{m}: z \in D, |\xi| H(z)^{\frac{1}{2}} \in I_0 \}.$$
	Set $$u(W)=(GH)^{-\frac{1}{m+1}}\phi\big(|\xi| H^{\frac{1}{2}}(z)\big)~\text{for}~W  \in U.$$
	Then $u$ is real analytic in $U$ and satisfies
	$$J(u)=1~\text{on}~ U_0, \quad u=0~\text{on}~\Sigma.$$
	Consequently, $\Sigma$ is obstruction flat.
\end{prop}

{\em Proof of Proposition \ref{prpn3d5}.} The analyticity of $u$ follows easily from the analyticity of $\phi$, as well as that of $G$ and $H$ (for the later, see the discussion in Remark \ref{rmk12}). We thus only need to prove the remaining assertions. For that,  we first prove the following lemma. Write $X=X(W):=|\xi| H^{\frac{1}{2}}(z)$ for $W \in U,$  and $Y=Y(W):=\frac{2}{m+1}-X\frac{\phi'(X)}{\phi(X)}$ for $W \in U_0.$ Then by (\ref{eqn3d5}), we have $Y=y(r)|_{r=X}=\frac{1}{Z(r)}|_{r=X}.$

In the following, we will also write the coordinates of $D \times \mathbb{C}$ as $W=(w_1, \cdots, w_{m-1}, w_m).$ That is, we identify $w_i$ with $z_i$ for $1 \leq i \le m-1$, and $w_{m}$ with $\xi.$ For a sufficiently differentiable function $\Phi$ on an open subset of $D \times \mathbb{C}$,  we write, for $1 \leq i, j \leq m,$ $\Phi_i=\frac{\partial \Phi}{\partial w_i}, \Phi_{\Ol{j}}=\frac{\partial \Phi}{\partial \Ol{w_j}},$ and $\Phi_{i\Ol{j}}=\frac{\partial^2 \Phi}{\partial w_i \partial \Ol{w_j}}$.
\begin{lemma}\label{lemma3d6}
	Let u(W) be as defined in Proposition \ref{prpn3d5}.  Write $Y'=\frac{dY}{dX},$ i.e., $Y'=y'(r)|_{r=X}.$ Then we have in $U_0,$	
	$$(-\log u)_{i\Ol{j}}=-\frac{1}{m+1} R_{i\Ol{j}}+ \frac{Y}{2}g_{i \Ol{j}}+\frac{Y'}{X}X_iX_{\Ol{j}}, \quad~\text{if}~ 1 \leq i, j \leq m-1,$$
	$$(-\log u)_{i\Ol{j}}=\frac{Y'}{X}X_iX_{\Ol{j}}, \quad~\text{if}~ i=m~\text{or}~j=m.$$
	Consequently,
	$$\det \big((-\log u)_{i\Ol{j}}\big)_{1 \leq i, j \leq m}=\frac{P(Y) Y'}{2^{m+1} X}H(z)G(z).$$
\end{lemma}

\begin{proof}
	Note for $1 \leq i, j \leq m, (\log X)_{i \Ol{j}}=\frac{1}{2}(\log H)_{i \Ol{j}}.$
	Consequently, for $1 \leq i, j \leq m,$
	$$(-\log u)_{i \Ol{j}}=\frac{1}{m+1}(\log G)_{i \Ol{j}}+ \frac{1}{m+1}(\log H)_{i \Ol{j}}-(\log \phi(X))_{i \Ol{j}}$$
	$$=\frac{1}{m+1}(\log G)_{i \Ol{j}}+ \frac{2}{m+1}(\log X)_{i \Ol{j}}-(\frac{\phi'}{\phi}X_{i \Ol{j}}+ (\frac{\phi'}{\phi})'X_i X_{\Ol{j}}).$$
	Since $X_{i \Ol{j}}=X (\log X)_{i \Ol{j}}+\frac{1}{X}X_i X_{\Ol{j}},$ the above is reduced to
	$$(-\log u)_{i \Ol{j}}=\frac{1}{m+1}(\log G)_{i \Ol{j}}+\big(\frac{2}{m+1}-X \frac{\phi'}{\phi}\big)(\log X)_{i \Ol{j}}-\big((\frac{\phi'}{\phi})'+\frac{1}{X} \frac{\phi'}{\phi} \big)X_i X_{\Ol{j}}.$$
	$$=\frac{1}{m+1}(\log G)_{i \Ol{j}}+ Y (\log X)_{i \Ol{j}}+\frac{Y'}{X}X_iX_{\Ol{j}}.$$
	Since $G(z)$ only depends on $z,$ we have $(\log G)_{i \Ol{j}}=0$ if $i=m$ or $j=m$. Furthermore, $(\log G)_{i \Ol{j}}=-R_{i\Ol{j}}$ for $1 \leq i, j \leq m-1.$ Similarly, $(\log X)_{i \Ol{j}}=\frac{1}{2}(\log H)_{i \Ol{j}}=0$ if $i=m$ or $j=m$, and $(\log X)_{i \Ol{j}}=\frac{1}{2}(\log H)_{i \Ol{j}}=\frac{1}{2}g_{i \Ol{j}}$ if $1 \leq i, j \leq m-1.$
	Then the first two equations in the lemma follow. 
To compute the determinant of the matrix $\big((-\log u)_{i\Ol{j}}\big)$, we first take out the factor $\frac{Y'}{X}X_m$ from the last row, and then use the last row to eliminate $\frac{Y'}{X}X_iX_{\Ol{j}}$ terms from other rows. This yields
\begin{align*}
	\det \big((-\log u)_{i\Ol{j}}&\big)_{1 \leq i, j \leq m}=\frac{Y'}{X}X_m X_{\Ol{m}}  \det \big(-\frac{1}{m+1} R_{i\Ol{j}}+ \frac{Y}{2}g_{i \Ol{j}}\big)_{1 \leq i, j \leq m-1}\\
	=&\frac{Y'}{X}X_m X_{\Ol{m}}  \det \big(\frac{Y}{2}\delta_{ij}-\frac{1}{m+1}R_{i \Ol{k}} \cdot g^{j \Ol{k}}(p)\big)_{1 \leq i, j \leq m-1} \det(g_{i \Ol{j}})_{1 \leq  i, j \leq m-1}\\
	=&\frac{Y'}{2^{m-1}X}X_m X_{\Ol{m}}   \det \big(Y\delta_{ij}-\frac{2}{m+1}R_{i \Ol{k}} \cdot g^{j \Ol{k}}(p)\big)_{1 \leq i, j \leq m-1} \det(g_{i \Ol{j}})_{1 \leq  i, j \leq m-1}.
\end{align*}
	Note $X_m=\frac{1}{2}|\xi|^{-1} \overline{\xi} H^{\frac{1}{2}}(z), X_{\Ol{m}}=\frac{1}{2}|\xi|^{-1} \xi H^{\frac{1}{2}}(z).$
	Therefore by the definitions of $G$ and $P(y)$, we have the above equals to $\frac{P(Y) Y'}{2^{m+1} X}H(z)G(z).$
	This proves the last equation Lemma \ref{lemma3d6}.
\end{proof}

We continue to prove Proposition \ref{prpn3d5}. 
Recall $\phi$ is defined by $\phi(r)=2  (\frac{r}{-Z'\hat{P}(Z)})^{\frac{1}{m+1}}\, Z$ on $I.$ This yields
$$-Z'\hat{P}(Z)Z^{-(m+1)}=2^{m+1}r \phi^{-(m+1)}~\text{on}~I_0.$$
Consequently, since $y=Z^{-1}$ and $-Z'=y'y^{-2},$
$$y'y^{m-1}\hat{P} (y^{-1})=2^{m+1}r \phi^{-(m+1)}~\text{on}~I_0.$$
By the definition of $\hat{P},$ we have $y^{m-1}\hat{P} (y^{-1})=P(y).$ Therefore the above is reduced to
$$y'P(y)=2^{m+1}r \phi^{-(m+1)}~\text{on}~I_0.$$
Replacing $r$ by $X=X(W)=|\xi| H^{\frac{1}{2}}(z)$ for $W \in U_0$, we have
$$Y' P(Y)=2^{m+1} (\phi(X))^{-(m+1)}\, X=2^{m+1} X G^{-1}(z) H^{-1}(z)(u(W))^{-(m+1)}.$$
The last equality follows from the definition of $u$. Combining this with Lemma \ref{lemma3d6}, we have
$$u^{m+1} \det \big((-\log u)_{i\Ol{j}}\big)_{1 \leq i, j \leq m}=1~\text{for}~W \in U_0.$$
This yields $J(u)=1$ on $U_0.$ Since $\phi(1)=0$, we have $u=0$ on $\Sigma.$ This proves the first part of Proposition \ref{prpn3d5}. The latter part of Proposition \ref{prpn3d5} follows from the first part and Graham's work \cite{Graham1987a} (see the discussion at the beginning of $\S$ 2.1). \qed

{\em Proof of Theorem  \ref{thm 1}:}  Theorem \ref{thm 1} follows from Proposition \ref{prpn3d5}. \qed

\smallskip

{\em Proof of Proposition \ref{prpnsp}:} By the work of Webster \cite{Web77b} (see  Proposition 4 in \cite{Wang}), $(M, g)$ is spherical if and only if $(M, g)$ is Bochner flat (i.e., its Bochner tensor vanishes). If $M$ is compact, then the conclusion follows from Corollary 4.17 in Bryant \cite{Bry01}; if $(M, g)$ has constant scalar curvature, then the conclusion follows from Proposition 2.5  in Bryant \cite{Bry01}. \qed

\smallskip

{\em Proof of Corollary \ref{cor 1}:} As the statement is trivial if $\dim M=1$, we will assume $\dim M\geq 2$. Note since $(M, g)$ has constant Ricci eigenvalues, it also has constant scalar curvature.  Corollary \ref{cor 1} then follows from Theorem \ref{thm 1} and Proposition \ref{prpnsp}. \qed

\subsection{Proof of Theorem \ref{thm 2}}\label{Sec 2.2}


	

Before we prove Theorem \ref{thm 2}, we first establish the following proposition, which is of particular importance for the proof of the theorem.

\begin{prop}\label{prpnzh2}
Let $n=m-1 \geq 1.$  Let $\lambda_1 \leq \cdots \leq \lambda_n <1$ be real numbers. Let $P(y)=\prod_{i=1}^n (y-\frac{2 \lambda_i}{m+1}).$ Let $Q(y)$ be a polynomial satisfying $\frac{dQ}{dy}=(m+1)yP(y)$ and $Q(\frac{2}{m+1})=0$ (thus $Q$ is a monic polynomial of degree $m+1$ and is uniquely determined). Let $\hat{P}$ and $\hat{Q}$ be polynomials satisfying
$$\hat{P}(x)=x^{m-1}P(x^{-1}), \quad \hat{Q}(x)=x^{m+1} Q(x^{-1}).$$
Note $\hat{P}(0)=1,\hat{P}(x)>0$ on $(0, \frac{m+1}{2}]$,  and $\hat{Q}(0)=1, \hat{Q}(\frac{m+1}{2})=0,~\hat{Q}'(\frac{m+1}{2})<0,~ \hat{Q}(x)>0$ on $(0, \frac{m+1}{2}).$ We have the following conclusions hold:
	
(1). There exists a unique real analytic function $Z=Z(r)$ on $[-1, 1]$ (meaning it extends real analytically to some open interval containing $[-1, 1]$) satisfying the following conditions:
\begin{equation}\label{eqnzodenz}
rZ'\hat{P}(Z)+\hat{Q}(Z)=0, \quad Z(1)=0.
\end{equation}
Moreover, $Z$ is an even function satisfying $Z(0)=\frac{m+1}{2}$ and $Z'(0)=0, Z'(r)<0$ on $(0, 1]$ and $Z'(1)=-1, Z''(0)<0.$ Consequently, $Z(r) \in (0, \frac{m+1}{2})$ on $(-1, 0) \cup (0, 1).$
	
(2). Let $\phi(r):=2(\frac{r}{-Z'\hat{P}(Z)})^{\frac{1}{m+1}}\, Z$.  Then $\phi$ is real analytic on $[-1, 1]$. Moreover, $\phi$ is an even function satisfying $\phi>0$ on $(-1, 1)$ and $\phi(1)=0.$ Moreover, $\phi$ satisfies $(m+1)rZ\phi'+(m+1-2Z)\phi=0$, and $\phi'(1)=-2.$
\end{prop}

{\em Proof of Proposition \ref{prpnzh2}.}
It follows easily from the assumption and elementary ODE theory that the ODE in (\ref{eqnzodenz}) has a real analytic solution $Z$ in some open interval $I$ containing $1$. Since $\hat{P}(0)=1, \hat{Q}(0)=1,$ we see the ODE in (\ref{eqnzodenz}) implies $Z'(1)=-1$. Set
$$t_0=\inf \{t \in [0, 1): ~\text{on}~(t,1],~\exists~\text{a real analytic solution}~Z~\text{to}~(\ref{eqnzodenz})~\text{with}~Z'<0\}.$$
By the definition, $0 \leq t_0 <1$ and  on $(t_0,1]$, there is a real analytic solution $Z$ to (\ref{eqnzodenz})  with $Z'<0$.

\begin{lemma} \label{lm1}
	We must have $t_0=0.$ Consequently, on $(0, 1]$ there exists a (unique) real analytic solution $Z$ to (\ref{eqnzodenz}) and it satisfies $Z'<0.$
\end{lemma}
{\em Proof of Lemma \ref{lm1}:} Seeking a contradiction, suppose $t_0 >0.$ Since $Z$ is decreasing on $(t_0, 1),$ we have $\mu=\lim_{r \rightarrow t_0^+} Z(r)>0$ exists. We claim $\mu \leq \frac{m+1}{2}.$ Otherwise, there exists some $t^* \in (t_0, 1)$ such that $Z(t^*)=\frac{m+1}{2}.$
This is a contradiction as the ODE in (\ref{eqnzodenz}) cannot hold at $t^*.$ We therefore proceed in two cases:

{\bf Case I:} Assume $\mu= \frac{m+1}{2}.$ Recall $\hat{P}(\mu)>0$ and $\hat{Q}(\mu)=0.$ Since $-Z'=\frac{\hat{Q}}{r\hat{P}}$, we have $0<-Z'\leq c (\mu-Z)$ on $(t_0, 1)$ for some positive constant $c$. Consequently, $\frac{d \ln (\mu -Z)}{dr} \leq c$ on $(t_0, 1).$ This contradicts the fact that $Z(r) \rightarrow \mu$ as $r \rightarrow t_0^+.$

{\bf Case II:} Assume $\mu < \frac{m+1}{2}.$ Then $\hat{P}(\mu)>0$ and $\hat{Q}(\mu)>0$. In the case, the following initial value problem has a real analytic solution $\widetilde{Z}$ on some open interval $J$ containing $t_0$:
$$rZ'\hat{P}(Z)+\hat{Q}(Z)=0, \quad Z(t_0)=\mu.$$
Shrinking $J$ if necessary, we can assume $\widetilde{Z}'<0$ on $J$. By the uniqueness of solutions in the ODE theory, we can glue the previous solution
with $\widetilde{Z}$ to obtain a real analytic solution to (\ref{eqnzodenz}), still called $Z$, on some open interval containing $[t_0, 1],$ which still satisfies $Z'<0.$ This contradicts the definition of $t_0$.

This proves Lemma \ref{lm1}. \qed

\smallskip

By Lemma \ref{lm1}, $Z$ is decreasing on $(0, 1)$ and therefore $\lambda=\lim_{r \rightarrow 0^+} Z(r)>0$ exists. By the same reasoning as in the proof of Lemma \ref{lm1}, we must have $\lambda \leq \frac{m+1}{2}.$ Suppose $\lambda < \frac{m+1}{2}.$ Note $\hat{P}(x), \hat{Q}(x)>0$ on $[0, \lambda]$. Since $-Z'=\frac{\hat{Q}}{r\hat{P}}$, we have $-Z' \geq \frac{c}{r}$ on $(0, 1)$ for some positive constant $c$. This contradicts the fact that $Z$ is bounded on $(0, 1).$  Hence we must have $\lambda = \frac{m+1}{2}$ and thus $\hat{Q}(\lambda)=0$.

We will keep the notation $\lambda := \frac{m+1}{2}$ and write $\hat{Q}(Z)=(Z-\lambda)g(Z)$ for some polynomial $g$.

\begin{lemma}\label{lm2}
	It holds that $g(\lambda)=-2\hat{P}(\lambda)<0.$ Furthermore, $Z'(0):=\lim_{r \rightarrow 0^+} \frac{Z(r)-\lambda}{r}=0.$
\end{lemma}

{\em Proof of Lemma \ref{lm2}:} The first statement follows from the direct computation as below. By the definition of $g$ and $\hat{Q},$
$$g(\lambda)=\hat{Q}'(\lambda)=(m+1)\lambda^{m}Q(\frac{1}{\lambda})-\lambda^{m-1}Q'(\frac{1}{\lambda}).$$
By the fact that $Q(\frac{1}{\lambda})=0$ and the definition of $Q$, we have
$$g(\lambda)=-\lambda^{m-1}Q'(\frac{1}{\lambda})=-\frac{(m+1)}{\lambda}\lambda^{m-1}P(\frac{1}{\lambda})=-2\lambda^{m-1}P(\frac{1}{\lambda})=-2\hat{P}(\lambda).$$
To prove the latter assertion, we write $Z(r)=\lambda+r G(r)$ for some analytic function $G$ on $(0, 1).$ It is clear that $G<0$ on $(0, 1).$ We only need to show $\lim_{r \rightarrow 0^+} G(r)=0$. For that, we substitute $Z=\lambda+rG$ and $\hat{Q}=(Z-\lambda)g(Z)=rG(r) g(Z)$ into the ODE in (\ref{eqnzodenz}) to get
$$\frac{G'}{G}=-\frac{\hat{P}+g}{r \hat{P}}~\text{on}~(0, 1).$$
By the first part of the lemma, $\lim_{r \rightarrow 0^+}\frac{\hat{P}+g}{\hat{P}}=-1.$ Consequently,
$\frac{d\ln (-G)}{dr}=\frac{G'}{G} \geq \frac{c}{r}$ on $(0, \delta)$ for some positive constants $c$ and $\delta$. This implies $\lim_{r \rightarrow 0^+} G(r)=0.$ \qed

\begin{lemma}\label{lm3}
	Write $Z(r)=\lambda + r^2 W(r)$. It holds that $W< 0$ on $(0, 1],$ and $a=\lim_{r \rightarrow 0^+} W(r)$ exists and is a negative real number.
\end{lemma}

{\em Proof of Lemma \ref{lm3}:} We only need to prove the second assertion. By assumption, $\hat{Q}(Z)=(Z-\lambda)g(Z)=r^2 W(r)g(Z)$. Substituting this and $Z=\lambda + r^2 W$ into the ODE in (\ref{eqnzodenz}), we get
\begin{equation}\label{eqnzw}
W'=-\frac{2\hat{P}(Z)+g(Z)}{r\hat{P}}W =-\frac{2\hat{P}(\lambda + rG)+g(\lambda + rG)}{r\hat{P}}W~\text{on}~(0, 1].
\end{equation}
Note $2\hat{P}(\lambda + rG)+g(\lambda + rG)$ is a polynomial in $rG$, whose constant term equals $2\hat{P}(\lambda)+g(\lambda)=0$ by Lemma \ref{lm2}. Again by Lemma \ref{lm2}, $G(r) \rightarrow 0$ as $r \rightarrow 0^+.$ Consequently, $f(r):=\frac{2\hat{P}(\lambda + rG)+g(\lambda + rG)}{r\hat{P}}$ extends to a continuous function on $[0, 1].$ 
Then by (\ref{eqnzw}) we have
$$\ln (-W(r))=\ln (-W(1)) -\int_{r}^1 f(t)dt=\ln (\lambda) -\int_{r}^1 f(t)dt $$
Consequently, $a=\lim_{r \rightarrow 0^+} W(r)$ exists and is a negative real number. \qed

\begin{lemma}\label{lm4}
	There exists a unique real analytic function $T_0(r)$ at $r=0$ satisfying the following initial value problem:
	\begin{equation}\label{eqntp}
	T'=-\frac{2\hat{P}(\lambda + r^2 T)+g(\lambda + r^2 T)}{r\hat{P}(\lambda + r^2 T)}T, \quad T(0)=a.
	\end{equation}
	Moreover, the function is even on $(-\epsilon, \epsilon)$ for some small $\epsilon>0$.
\end{lemma}

\begin{rmk}\label{rmk16}
	Note $2\hat{P}(\lambda + r^2 T)+g(\lambda + r^2 T)$ is a polynomial in $r^2T$, whose constant term equals $2\hat{P}(\lambda)+g(\lambda)=0$ by
	Lemma \ref{lm2}. Therefore $\frac{2\hat{P}(\lambda + r^2 T)+g(\lambda + r^2 T)}{r}$ is a polynomial in $r$ and $T$.
\end{rmk}

{\em Proof of Lemma \ref{lm4}:} Note by Remark \ref{rmk16}, the right hand side of the ODE in (\ref{eqntp}) is real analytic in $r, T$ in a neighborhood of $(0, a).$ Therefore the existence and uniqueness of the solution, as well as its real analyticity, follow from elementary ODE theory. Note if $T_0$ is a solution  to the initial value problem \eqref{eqntp}, then so is $T_0(-r)$. By uniqueness of the solution, $T_0$ is an even function. \qed

Let $T_0: (-\epsilon, \epsilon) \rightarrow \mathbb{R}$ be as in Lemma \ref{lm4}. Write $Z(r)=\lambda + r^2 W(r)$ as in Lemma \ref{lm3}.
Note $W$ and $T_0$ are both continuous functions on $[0, \epsilon)$ satisfying the following ODE and the Cauchy data:
\begin{equation}\label{eqntwr}
T'=-\frac{2\hat{P}(\lambda + r^2 T)+g(\lambda + r^2 T)}{r\hat{P}(\lambda + r^2 T)}T~\text{on}~(0, \epsilon), \quad T(0)=a.
\end{equation}

\begin{lemma}\label{lm17}
	We have $W=T_0$ on $[0, \epsilon)$.
\end{lemma}

{\em Proof of Lemma \ref{lm17}:} This follows from basic ODE theory. For the convenience of readers, we sketch a proof. Write $\psi(r, T)$ for the right hand side of the ODE in (\ref{eqntwr}). Recall $\psi(r, T)$ is real analytic in a neighborhood of $(r, T)=(0, a).$ Pick a small neighborhood $U$ of $(0, a)$ and a positive constant $L$ such that
$$|\psi(r, T_1)-\psi(r, T_2)|\leq L |T_1-T_2|~\text{for}~(r, T_1), (r, T_2) \in U.$$
By making $\epsilon$ smaller if necessary, we can assume $(r, W(r)), (r, T_0(r)) \in U$ whenever $r \in  [0, \epsilon)$.
Then for $x \in [0, \epsilon),$
$$|W(x)-T_0(x)| \leq \int_{0}^x |\psi(t, W(t))-\psi(t, T_0(t))|dt \leq \int_{0}^x L|W(t)-T_0(t)|dt.$$
Then by the integral form of the Gr\"onwall lemma, we have $W=T_0$ on $[0, \epsilon)$. \qed

We now continue the proof of Proposition \ref{prpnzh2}. Let $\lambda, T_0$ be as above and set $\Psi=\lambda + r^2 T_0$. Then $\Psi$ is a real analytic even function on $(-\epsilon, \epsilon).$ Moreover, $\Psi=Z$ on $[0, \epsilon)$.
Therefore we can glue $Z$ with $\Psi$, and then apply the even extension to obtain a real analytic function on $[-1, 1]$, which we still denote by $Z$. It is clear that this new function $Z$ still satisfies the ODE in (\ref{eqnzodenz}). Moreover, since $a=W(0)<0,$ we see  $Z''(0)<0$. This proves part (1) of Proposition \ref{prpnzh2}.

We now prove part (2) of Proposition \ref{prpnzh2}. First by the definition of $\phi$ and the properties of $Z$ in part (1), it is clear that $\phi$ is real analytic and even on $[-1, 1]$, and $\phi>0$ on $(-1, 1)$ and $\phi(1)=0$. The latter assertion in part (2) can be proved identically as Proposition \ref{prpnzh1}.

This finishes the proof of Proposition \ref{prpnzh2}. \qed

We are now ready to prove Theorem \ref{thm 2}.
Choose a coordinate chart $(D, z)$ of $M$ together with a frame $e_L$ of $L$ over $D$. Writing $\pi: L \rightarrow M$ for the canonical fiber projection, we have $$\pi^{-1}(D)=\{\xi e_L(z): (z, \xi) \in D \times \mathbb{C} \}.$$
Under this trivialization,  $D(L)$ and $S(L)$ can be written as follows locally over $D,$
$$D(L) \cap \pi^{-1}(D) =\{W=(z, \xi) \in D \times \mathbb{C}: |\xi|^2 H(z, \Ol{z})<1 \},$$
$$S(L) \cap \pi^{-1}(D) =\{W=(z, \xi) \in D \times \mathbb{C}: |\xi|^2 H(z, \Ol{z})=1 \}.$$
Here $H(z)=H(z, \Ol{z})=|e_L(z)|_h^2.$
In the local coordinates, we write $g=(g_{i\Ol{j}})$ on $D$. By the definition of $g$, we have $g_{i \Ol{j}}=\frac{\partial^2 \log H}{\partial z_i \partial \Ol{z_j}}.$ Write $G(z)=G(z, \Ol{z})=\det (g_{i\Ol{j}})>0.$
In the local coordinates, the K\"ahler form $\widetilde{\omega}$ in Theorem \ref{thm 2} is given by $\widetilde{\omega}=-i \partial \Ol{\partial} \log u$, where $u=(GH)^{-\frac{1}{m+1}}\phi\big(|\xi| H^{\frac{1}{2}}(z)\big)$ and $\phi$ is given by part (2) of Proposition \ref{prpnzh2}.
By the properties of $\phi$, $u$ is real analytic in a neighborhood of $\Ol{D(L)} \cap \pi^{-1}(D)$.
By the identical proof as in Proposition \ref{prpn3d5}, we can show 
$u=0~\text{on}~S(L) \cap \pi^{-1}(D)$ and
$J(u)=1~\text{on}~D(L) \cap \pi^{-1}(D).$

Since $J(u)=1$, equivalently, $\det \big((-\log u)_{i\Ol{j}}\big)_{1 \leq i, j \leq m}=u^{-(m+1)}$ in $D(L) \cap \pi^{-1}(D),$ and $u$ is a local defining function of some strongly pseudoconvex piece of boundary, we see $\widetilde{\omega}$ is positive definite in $D(L) \cap \pi^{-1}(D)$. 
Also $J(u)=1$ implies the metric $\widetilde{g}$ induced by $\widetilde{\omega}$ has constant Ricci curvature $-(m+1).$
Since the coordinate chart $D$ is arbitrarily chosen, we conclude $\widetilde{g}$ is a \ke metric in $D(L)$.

It remains to prove $\widetilde{g}$ gives a complete metric on $D(L)$. By the Hopf-Rinow theorem, it suffices to show $(D(L), \widetilde{g})$ is geodesically complete.
Let $\gamma: [0,a) \rightarrow D(L)$ be a  non-extendible geodesic in $D(L)$ of unit speed with respect to $\widetilde{g}$. We only need to show that $a=+\infty,$ equivalently that $\gamma$ has infinite length. For that, we first establish the following lemma.

\begin{lemma}\label{lm3d15}
The metric $\widetilde{g} \geq \frac{1-\lambda_n}{m+1} \pi^{*}(g)$ in $D(L).$
\end{lemma}
{\em Proof of Lemma \ref{lm3d15}:} Fix a coordinate chart $(D, z)$ of $M$ together with trivialization of $L$ over $D$.
Let $Z$ be as in Proposition \ref{prpnzh2} and let $y(r)=\frac{1}{Z(r)}.$ Under the  coordinates of $\pi^{-1}(D),$ as in $\S$ \ref{Sec 2.1}, we write
$X=X(W):=|\xi| H^{\frac{1}{2}}(z)$,  and $Y=Y(W):=\frac{2}{m+1}-X\frac{\phi'(X)}{\phi(X)}$. Similarly as in $\S$ \ref{Sec 2.1},   by (\ref{eqn3d5}), we have $Y=y(r)|_{r=X}.$ Furthermore, by Proposition \ref{prpnzh2}, $Y \in  [\frac{2}{m+1}, \infty)$ on $D(L) \cap \pi^{-1}(D).$

Since the setup here is the same as in $\S$ \ref{Sec 2.1}, an identical version of  Lemma \ref{lemma3d6} still holds. In particular, we have in $D(L) \cap \pi^{-1}(D),$ writing $Y'=\frac{dY}{dX},$
\begin{equation}\label{eqn3d11}
\widetilde{g}_{i\Ol{j}}=(-\log u)_{i\Ol{j}}=-\frac{1}{m+1} R_{i\Ol{j}}+ \frac{Y}{2}g_{i \Ol{j}}+\frac{Y'}{X}X_iX_{\Ol{j}}, \quad~\text{if}~ 1 \leq i, j \leq m-1,
\end{equation}
\begin{equation}\label{eqn3d12}
\widetilde{g}_{i\Ol{j}}=(-\log u)_{i\Ol{j}}=\frac{Y'}{X}X_iX_{\Ol{j}}, \quad~\text{if}~ i=m~\text{or}~j=m.
\end{equation}

Recall by Proposition \ref{prpnzh2},  $Z(r)$ is even, real analytic and positive on $[-1, 1]$, $Z'(r)<0$ on $(0, 1)$ and $Z'(0)=0, Z''(0)<0$.
Note $y'(r)=-Z^{-2}(r)Z'(r)$. It follows that $\frac{y'}{r}$ is a  real analytic even function on $(-1, 1).$ Furthermore, it is positive on $(-1, 1).$
Consequently, $\frac{Y'}{X}$ is a positive real analytic function in $D(L) \cap \pi^{-1}(D)$. Finally by (\ref{eqn3d11}) and (\ref{eqn3d12}), since $(X_iX_{\Ol{j}})$ is semi-positive definite and $\frac{Y}{2} \geq \frac{1}{m+1}$, we have
$$\widetilde{g}_{i\Ol{j}}=\big(-\frac{1}{m+1} R_{i\Ol{j}}+ \frac{Y}{2}g_{i \Ol{j}}\big) \oplus 0 \geq \frac{1}{m+1} \big(-R_{i\Ol{j}}+ g_{i \Ol{j}}\big) \oplus 0~~\text{on}~D(L) \cap \pi^{-1}(D).$$
By the assumption on the Ricci eigenvalues, we have $R_{i\Ol{j}} \leq \lambda_n g_{i \Ol{j}}$ with $\lambda_n <1.$ Then the lemma follows easily. \qed

Next we write $\hat{\gamma}=\pi(\gamma)$ for the projection of $\gamma$ to $M$ and proceed in two cases:

{\bf Case I:} Assume $\hat{\gamma}$ is not contained in any compact subset of $M$. In this case, since $g$ is complete on $M$, the length $L_{g}(\hat{\gamma})$ of $\hat{\gamma}$ under the metric $g$ equals $+\infty.$  Now by Lemma \ref{lm3d15} the length $L_{\widetilde{g}}(\gamma)$ of $\gamma$ under the metric $\widetilde{g}$ satisfies $L_{\widetilde{g}}(\gamma) \geq \sqrt{\frac{1-\lambda_n}{m+1}} L_{g}(\hat{\gamma}).$ Therefore it also equals $+\infty.$ Hence $a=+\infty.$

{\bf Case II:} Assume $\hat{\gamma}$ is contained in some compact subset $K$ of $M$. Note for any sequence  $\{t_i\}_{i \geq 1}$ with $t_i \rightarrow a$, we must have $|\gamma(t_i)|_h \rightarrow 1.$ Otherwise, by passing to subsequence, $\gamma(t_i)$ converges in $D(L),$ which contradicts the assumption that $\gamma$ is non-extendible.
Consequently, there exist some small $\epsilon >0$ and some $0 < t_0 <a $ such that $\gamma([t_0, a)) \subseteq W_{\e}:=\{w \in D(L): \pi(w) \in K, |w|_h > \e\}$.

Again by the assumption on the Ricci eigenvalues, we have $\Ric \leq \lambda_n g \leq g.$ Then by (\ref{eqnmet}), we have in $D(L),$
$$\omega(w, \Ol{w}) \geq \omega_1(w, \Ol{w}):= -i \partial \Ol{\partial} \log \phi(|w|_h).$$

Define a function $\psi: L \rightarrow \mathbb{R}$ by $\psi(w):=|w|_h^2-1$. Then by the well-known Grauert's observation, $\psi$ is strictly plurisubharmonic at every point $w \in L$ with $|w|_h>0.$ Moreover, $\psi$ is a  defining function of $S(L).$

Write $\varphi(w)=\phi(|w|_h).$ Recall  $\phi(1)=0, \phi'(1)<0$, and $\phi>0$ on $(-1, 1)$.  Hence $\varphi=(-\psi)\mu$ for some positive smooth function $\mu$ in a neighborhood of $\Ol{W_{\e}}$.
By page 509 of Cheng--Yau \cite{CheYau}, it proves $\Psi:=\log \frac{1}{(-\psi)}$ has bounded gradient with respect to $\omega_0=i\partial \Ol{\partial} \Psi$. More precisely, $|\nabla \Psi|_{\omega_0} \leq 1$ in $W_{\e}$.
Here the gradient is taken with respect to $\omega_0.$ 
Note in any coordinates chart, we have
$$\Psi_{j\Ol{k}}=\frac{1}{(-\psi)}\psi_{j\Ol{k}}+\frac{1}{\psi^2} \psi_j \psi_{\Ol{k}} \geq \frac{1}{(-\psi)}\psi_{j\Ol{k}}.$$
Therefore $\omega_0 \geq \frac{1}{(-\psi)} i\partial \Ol{\partial} \psi.$
Write $\kappa:=\log \frac{1}{\mu}$. By the compactness of $\Ol{W_{\e}}$ and the positivity of $i\partial \Ol{\partial} \psi,$ we see
$i\partial \Ol{\partial} \psi \geq -c i\partial \Ol{\partial} \kappa$ in $\Ol{W_{\e}}$ for some positive constant $c>0.$
Consequently, when $\frac{1}{(-\psi)}$ is large enough, we have $\omega_0 > 4 (-i\partial \Ol{\partial} \kappa)$. Hence we can find $M>0$ such that in $V_{M}:=\{z \in W_{\e}: \Psi(z)>M\},$ it holds that $\omega_0 > 4 (-i\partial \Ol{\partial} \kappa)$. This implies $\omega_1=-i\partial \Ol{\partial} \log \varphi=\omega_0 + (i\partial \Ol{\partial} \kappa) > \frac{3}{4} \omega_0$ in $V_M.$  Consequently,  $|v|_{\omega_1} \geq  \frac{\sqrt{3}}{2} |v|_{\omega_0}$ for any tangent vector $v$.

Next we choose a sequence $\{t_j\}_{j \geq 1}$ such that all $t_j \geq t_0, t_j \rightarrow a,$ and $\gamma(t_j) \in O_{j}:=\{z \in W_{\e}: \Psi(z)>2^j\}.$ Then for any $j$ large enough, it holds that $O_{j-2} \subset V_M.$ Moreover, we can find some $0 < t_j^* <t_j$ such that  $\gamma((t_j^*, t_j]) \subset  O_{j-1}$ and $\Psi (\gamma(t_j^*))=2^{j-1}.$ Then we have
$$\int_{0}^{t_j} |\gamma'(t)|_{\omega_1} dt \geq \int_{t_j^*}^{t_j} |\gamma'(t)|_{\omega_1} dt \geq  \frac{\sqrt{3}}{2} \int_{t_j^*}^{t_j} |\gamma'(t)|_{\omega_0} dt \geq \frac{\sqrt{3}}{2} \int_{t_j^*}^{t_j} |\nabla \Psi|_{\omega_0} |\gamma'(t)|_{\omega_0} dt  $$
$$\geq \frac{\sqrt{3}}{2} \int_{t_j^*}^{t_j} \langle \nabla \Psi, \gamma'(t) \rangle_{\omega_0} dt= \frac{\sqrt{3}}{2} \int_{t_j^*}^{t_j}  \frac{d}{dt} \big(\Psi \circ \gamma(t)\big) dt= \Psi \circ \gamma(t_j)-\Psi \circ \gamma (t_j^*)> 2^{j-1}.$$
This proves $\gamma$ has infinite length and thus $a=+\infty$.

Therefore in any case, $a=+\infty$. 
This proves the completeness of the metric $\widetilde{g}$. The uniqueness of a complete \ke metric of negative curvature is well-known and follows from Yau's Schwarz Lemma \cite{Yau78}. This finishes the proof of Theorem \ref{thm 2}.

\begin{rmk}
Proposition \ref{prpn17} shows that in Theorem \ref{thm 2} the assumption that all Ricci eigenvalues are strictly less than $1$ cannot be relaxed. On the other hand, it remains interesting to see whether the constancy assumption on the Ricci eigenvalues can be dropped. More precisely, let $M$ be a complex manifold and $(L,h)$ be a negative line bundle over $M$. Assume the \k metric $g$ induced by the dual bundle $(L^*, h^{-1})$ is complete on $M$. Assume at every point $p \in M,$ the Ricci eigenvalues of $(M, g)$ are all strictly less than one.
Does the disk bundle $D(L)=\{w \in L: |w|_h <1 \}$ admit a complete \ke metric with negative Ricci curvature?
\end{rmk}

\subsection{Proof of Proposition \ref{prpn17}}\label{Sec 2.3}

In this section, we prove Proposition \ref{prpn17}.
By Theorem \ref{thm 2}, it suffices to prove the ``only if" implication.
For that, we let $(L, h)$ be as in the assumption of Proposition \ref{prpn17}, such that $D(L)$ admits a complete \ke metric with negative Ricci curvature, and we will prove $\lambda<1$. Set
	\begin{equation*}
		L_0:=\{([Z], \xi Z): \xi \in \mathbb{C} \mbox{ and } Z\in \mathbb{C}^{n+1}\setminus\{0\} \}\subset \mathbb{CP}^n\times \mathbb{C}^{n+1},
	\end{equation*}
	which is the tautological line bundle over $\mathbb{CP}^n$. Let $h_0$ be the Hermitian metric on $L_0$ defined by
	\begin{equation*}
		h_0(\xi Z):=|\xi Z|^2 \quad \mbox{ for any } ([Z], \xi Z)\in L_0.
	\end{equation*}
Then $\omega_0:=-c_1(L_0,h_0)=i \partial\dbar\log h_0$
induces the Fubini-Study metric $g_0$ on $\mathbb{CP}^n$, which satisfies the \ke equation
$\Ric(\omega_0)=(n+1) \omega_0.$

	Note that the automorphism group $\Aut(\mathbb{CP}^n)$ is $\text{PGL}(n+1)$. For any $f\in \Aut(\mathbb{CP}^n)$, there exists some $A\in \text{GL}(n+1)$ such that $f([Z])=[ZA]$. Note $f$ naturally induces a bundle isomorphism $F: L_0\rightarrow L_0$, depending on the choice of $A$,
	\begin{equation*}
		F([Z], \xi Z)=([ZA], \xi ZA).
	\end{equation*}
	In particular, if $A\in \mathrm{U}(n+1)$, then $f$ is an isometry of $(\mathbb{CP}^n,g_0)$ and $F$ is an isometry of $(L_0, h_0)$. More generally, every $f \in \Aut(\mathbb{CP}^n)$ also induces an automorphism of $L_0^k$ given by
	\begin{equation}\label{induced bundle isomorphism}
		F([Z], \xi Z\otimes Z\cdots \otimes Z)=([Z], \xi ZA\otimes ZA\cdots \otimes ZA).
	\end{equation}
	In addition, if $A\in \mathrm{U}(n+1)$, then $F$ is an isometry of $(L_0^k, h_0^k)$.
	

	Since the Picard group of $\mathbb{CP}^n$ is generated by the tautological line bundle $L_0$, we can assume, up to isomorphism, $L=L_0^k$ for some positive integer $k$. Now we have two Hermitian metrics, $h$ and $h_0^k$, on $L$.  Their induced metrics on $\mathbb{CP}^n$, $g$ and $k g_0$, respectively satisfy
	\begin{equation*}
		\Ric(\omega)=\lambda\omega \qquad \mbox{and} \qquad \Ric(k\omega_0)=\frac{n+1}{k} (k\omega_0).
	\end{equation*}
Here $\omega$ denotes the \k form of $g$. Since $[\Ric(\omega)]=[\Ric(k\omega_0)]= c_1(\mathbb{CP}^n)$ and $[\omega]=[k\omega_0]=c_1(L)$, we have $\lambda=\frac{n+1}{k}$.
	
By the uniqueness of the complete \ke metric on $\mathbb{CP}^n$, there exists some $f\in \Aut(\mathbb{CP}^n)$ such that $k g_0=f^*g$.
By a standard argument in \k geometry, up to a bundle isometry,  we can just assume $(L, h)= (L_0^k, h_0^k)$, and $g=k g_0$. In this case, the isometry group of $(\mathbb{CP}^n, g)$ is $\PU(n+1)$, which acts transitively on $\mathbb{CP}^n$.
	
Suppose $D(L)$ admits a complete \ke metric  of negative Ricci curvature. Write its \k form as $\widetilde{\omega}$. By rescaling the metric, we can assume $\Ric(\widetilde{\omega})=-(n+2)\widetilde{\omega}$.
Set $U_0:=\bigl\{ [Z_0, Z_1, \cdots, Z_n] \in \mathbb{CP}^n: Z_0 \neq 0  \bigr\} \subset \mathbb{CP}^n.$	
Let $z=(z_1, \cdots z_n)$ be the local affine coordinates of $\mathbb{CP}^n$ on $U_0$, that is, $z_j=Z_j/Z_0$ for $1\leq j \leq n$. Under the local trivialization of $L$ on $U_0:$
\begin{equation*}
	([1,z], \xi(1,z)\otimes \cdots \otimes (1,z) ) \rightarrow (z, \xi) \in U_0\times \mathbb{C},
\end{equation*}
we have $(z,\xi)$ forms a local coordinate system of $L$ on $\pi^{-1}(U_0)$. For simplicity, we also denote $z_{n+1}=\xi$. On $\pi^{-1}(U_0)\cap D(L)$, we write
\begin{equation*}
	\widetilde{\omega}=\sqrt{-1}\sum_{i, j=1}^{n+1} \widetilde{g}_{i\bar{j}} dz_i\wedge d\oo{z_j}.
\end{equation*}
Set $u=\bigl(\det(\widetilde{g}_{i\bar{j}}) \bigr)^{-\frac{1}{n+2}}.$
Since $\widetilde{\omega}$ is K\"ahler-Einstein, $u$ is real analytic on $\pi^{-1}(U_0)\cap D(L)$ (cf. Theorem 6.1 in \cite{DeKa81}). We write the \ke condition in terms of $u$:
\begin{equation}\label{KE equation}
	-\partial_i\partial_{\bar{j}}\log u^{-(n+2)}=-(n+2) \widetilde{g}_{i\bar{j}}.
\end{equation}
By taking the determinant, we get
\begin{equation}\label{u complex MA equation}
	\det\bigl( (-\log u)_{i\bar{j}} \bigr)_{1 \leq i, j \leq n+1}=u^{-(n+2)}.
\end{equation}



Given any biholomorphism $F: D(L)\rightarrow D(L)$, by the uniqueness of negatively curved complete \ke metric, we have $F^*\widetilde{\omega}=\widetilde{\omega}$. This yields $F^*\widetilde{\omega}^{n+1}=\widetilde{\omega}^{n+1}$, and in local coordinates this writes into
\begin{equation}\label{u transformation formula}
	u(F(z,\xi))\, |\det JF (z,\xi)|^{-\frac{2}{n+2}}= u(z,\xi) \quad \mbox{ if } (z,\xi), F(z,\xi) \in \pi^{-1}(U_0)\cap D(L),
\end{equation}
where $JF$ is the Jacobian matrix.

We will use this transformation relation to simplify $u$. Consider the $S^1$ action on $D(L)$. That is, for any $\theta\in [0,2\pi)$, we have an associated biholomorphism $R_{\theta}: D(L)\rightarrow D(L)$ defined by
\begin{equation*}
	R_{\theta}([Z], \xi Z\otimes Z\cdots \otimes Z)=([Z],e^{i\theta}\xi Z\otimes Z \cdots \otimes Z).
\end{equation*}
In local coordinates $(z,\xi)$, $R_{\theta}(z,\xi)=(z,e^{i\theta}\xi)$, and the determinant of the Jacobian of $R_{\theta}$ equals $e^{i\theta}$. We let $F=R_{\theta}$ in (\ref{u transformation formula}) to obtain
\begin{equation}\label{u invariant S1 action}
	u(z,\xi)=u(z,|\xi|) \quad \mbox{ for any } (z,\xi)\in \pi^{-1}(U_0)\cap D(L).
\end{equation}

Let $f$ be an isometry of $(\mathbb{CP}^n, \omega_0)$. Recall that $f$ induces a bundle isometry $F$ as in \eqref{induced bundle isomorphism}. In particular, $F$ is a biholomorphism on $D(L)$. Since $\Iso(\mathbb{CP}^n, \omega_0)=\PU(n+1)$, $f$ is rational. Thus there exists a complex variety $E$, depending on $f$, such that $f|_{U_0\setminus E}$, which we still denote by $f$, gives a holomorphic rational map from $U_0\setminus E\rightarrow U_0$. Moreover, the induced bundle isometry $F$ can be written as follows for some linear function $p$ in $z:$
\begin{equation*}
F(z,\xi)=(f(z), \xi p(z)^k)~~\text{on}~\pi^{-1}(U_0\setminus E)\cap D(L).
\end{equation*}
To apply the transformation relation \eqref{u transformation formula} to the above map $F$, we need to compute $|\det Jf|$ and $|\det JF|$. Since $f\in \Iso(\mathbb{CP}^n, \omega_0)$, $f$ preserves the volume form $\frac{\omega_0^n}{n!}$. Note that on $U_0,$
\begin{equation}\label{volume form Fubini-Study metric}
	\frac{\omega_0^n}{n!}=\frac{1}{(1+|z|^2)^{n+1}} \,i^ndz_1\wedge d\oo{z_1}\wedge\cdots dz_n\wedge d\oo{z_n}.
\end{equation}
Thus,
\begin{equation}\label{Jf determinant}
	\frac{|\det Jf(z)|^2}{(1+|f(z)|^2)^{n+1}}=\frac{1}{(1+|z|^2)^{n+1}} \quad \mbox{ for any } z\in U_0\setminus E.
\end{equation}

Note that the Hermitian metric of $L$ is $h_0^k$ and
\begin{equation*}
	h_0^k(\xi(1,z)\otimes\cdots \otimes(1,z))=|\xi|^2(1+|z|^2)^k.
\end{equation*}
Since $F$ preserves this Hermitian metric, we have
\begin{equation*}
	|\xi|^2(1+|z|^2)^k=|\xi|^2|p(z)|^{2k}(1+f(z))^k.
\end{equation*}
Thus,
\begin{equation}\label{p(z)}
	|p(z)|^2=\frac{1+|z|^2}{1+|f(z)|^2} \quad \mbox{ for any } z\in U_0 \setminus E.
\end{equation}

By \eqref{Jf determinant} and \eqref{p(z)}, we obtain
\begin{equation*}
	|\det(JF)(z)|=|\det(Jf)(z)| \cdot |p(z)|^k=\bigl(\frac{1+|f(z)|^2}{1+|z|^2}\bigr)^{\frac{n+1-k}{2}}.
\end{equation*}

Now we can rewrite \eqref{u transformation formula} into
\begin{equation*}
	u(F(z,\xi))\cdot \bigl(\frac{1+|f(z)|^2}{1+|z|^2}\bigr)^{n+1-k}=u(z,\xi).
\end{equation*}
Since $\Iso(\mathbb{CP}^n, \omega_0)$ acts transitively on $\mathbb{CP}^n$, for any $z\in U_0$ we can choose some $f\in \Iso(\mathbb{CP}^n, \omega_0)$ such that $f(z)=0$. Therefore, we get
\begin{equation*}
	u(z,\xi)=u(0,\xi p(z)^k) \cdot \bigl(\frac{1}{1+|z|^2}\bigr)^{\frac{n+1-k}{n+2}}.
\end{equation*}
By \eqref{u invariant S1 action} and \eqref{p(z)}, we simplify it into
\begin{equation}\label{eqn3d21}
	u(z,\xi)=u\bigl(0, |\xi| (1+|z|^2)^{\frac{k}{2}}\bigr) \cdot \bigl(\frac{1}{1+|z|^2}\bigr)^{\frac{n+1-k}{n+2}}.
\end{equation}

We denote
\begin{align*}
	G(z):=\frac{k^n}{(1+|z|^2)^{n+1}}, \quad
	H(z):=h\bigl((1,z)\otimes\cdots\otimes(1,z)\bigr)=(1+|z|^2)^k.
\end{align*}
Note that by \eqref{volume form Fubini-Study metric}, $G(z)$ is the determinant of the metric $g=k g_0$.
By writing $X:=|\xi|H^{\frac{1}{2}}$, the equation (\ref{eqn3d21}) is reduced to
\begin{equation}
	u(z,\xi)=k^{\frac{n}{n+2}}u(0,X) \bigl(GH\bigr)^{-\frac{1}{n+2}}.
\end{equation}



Let $\phi(r)=k^{\frac{n}{n+2}}u(0,r)$ for $r\in (-1,1)$. Then $\phi$ is a positive, real analytic, even function (see (\ref{u invariant S1 action})). We rewrite $u$ into
\begin{equation*}
	u(z,\xi)=\phi(X) \bigl(GH\bigr)^{-\frac{1}{n+2}}.
\end{equation*}
We shall reduce the complex Monge-Amp\`ere equation \eqref{u complex MA equation} on $u$ into an ODE on $\phi$. Set
\begin{equation}\label{y definition}
	y(r):=\frac{2}{n+2}-\frac{r \phi'(r)}{\phi(r)}.
\end{equation}
By the same computation as in Lemma \ref{lemma3d6},  we have
\begin{align*}
	(-\log u)_{i\bar{j}}=\begin{dcases}
	-\frac{1}{n+2} R_{i\bar{j}}+\frac{Y}{2} g_{i\bar{j}}+\frac{Y'}{X} X_i X_{\bar{j}} & \mbox{ if } 1\leq i, j\leq n
	\\
	\frac{Y'}{X} X_iX_{\bar{j}} & \mbox{ if } i=n+1 \mbox{ or } j=n+1,
	\end{dcases}
\end{align*}
where $Y=y|_{r=X}$, $Y'=y'|_{r=X}$, $X_i=\frac{\partial X}{\partial z_i}$ for $1\leq i\leq n+1$ ($\xi=z_{n+1}$) and $X_{\bar{j}}$ is defined similarly as in Lemma \ref{lemma3d6}; $g= \sum g_{i \Ol{j}} dz_i \otimes d \Ol{z_j}$  and $R_{i\bar{j}}$ is the Ricci curvature of $g$. Moreover,
\begin{equation*}
	\det\bigl( (-\log u)_{i\bar{j}} \bigr)_{1 \leq i, j \leq n+1}=\frac{P(Y) Y'}{2^{n+2}X} GH,
\end{equation*}
where
\begin{equation*}
	P(Y)=\det\bigl(YI_n-\frac{2}{n+2} \Ric(g) \cdot g^{-1} \bigr)=\bigl(Y-\frac{2\lambda}{n+2}\bigr)^n.
\end{equation*}
Therefore, the complex Monge-Amp\`ere equation \eqref{u complex MA equation} now writes into
\begin{equation*}
	\frac{P(Y) Y'}{2^{n+2}X} GH=\phi(X)^{-(n+2)} GH.
\end{equation*}
After simplification, we obtain
\begin{equation*}
  	P(Y) Y'=2^{n+2}X \phi(X)^{-(n+2)} \quad \mbox{ for any } X\in (0,1).
\end{equation*}
That is,
\begin{equation*}
	P(y) y'=2^{n+2}r \phi(r)^{-(n+2)} \quad \mbox{ for any } r\in (0,1).
\end{equation*}
Note that $\phi(r)$ and $y(r)$ are both real analytic on $(-1,1)$. It follows that
\begin{equation}\label{ODE CPn}
P(y) y'=2^{n+2}r \phi(r)^{-(n+2)} \quad \mbox{ for any } r\in (-1,1).
\end{equation}


We observe the following properties of function $y(r)$.
\begin{lemma}
$y=y(r)$ satisfies (1) $y'(0)=0$; (2) $y'(r)>0$ for any $r\in (0,1)$; and (3) $\int_{\varepsilon}^1 y'(r) dr=+\infty$ for any $\varepsilon\in (0,1)$.
	In particular, $y$ is an increasing function on $(0,1)$ and $\lim_{r \rightarrow 1^-} y(r)=+\infty$.
\end{lemma}
\begin{proof}
	Since $\phi(r)$ is an even function on $(-1,1)$, by \eqref{y definition}, $y(r)$ is also an even function on $(-1,1)$. Thus, $y'(0)=0$. Next, note that $(\log (\frac{1}{u}))_{i\bar{j}}$ is positive definite by the \ke condition \eqref{KE equation}. Thus,
	\begin{equation*}
		\bigl(\log (\frac{1}{u})\bigr)_{n+1\oo{n+1}}=\frac{Y'}{X}X_{n+1}X_{\oo{n+1}}=\frac{Y'H}{4X}>0~\text{for}~X \in (0, 1).
	\end{equation*}
Therefore, $y'>0$ on $(0,1)$.  To prove (3), we consider a curve $\gamma(t)=(0,\cdots, 0, t)\in \pi^{-1}(U_0)\cap D(L)$ for $t\in (\varepsilon,1)$. Since $\widetilde{\omega}$ is complete, the length of $\gamma$ with respect to $\widetilde{\omega}$ is infinity. Note that on $\gamma$, $H=1$ and $X=t$. Thus, by H\"older's inequality
	\begin{equation*}
		\infty=\int_{\varepsilon}^1 |\gamma'(t)|_{\widetilde{\omega}} dt= \sqrt{2}\int_{\varepsilon}^1 \bigl( \frac{y'}{4t} \bigr)^{\frac{1}{2}}dt\leq \sqrt{2}\bigl(\int_{\varepsilon}^1 y'dt\bigr)^{\frac{1}{2}}\cdot \bigl(\int_{\varepsilon}^1\frac{1}{4t}dt\bigr)^{\frac{1}{2}}.
	\end{equation*}
	Therefore, $\int_{\varepsilon}^1y'=\infty$ for any $\varepsilon\in (0,1)$.
\end{proof}

We are now ready to prove $\lambda< 1$. First, we compare the vanishing order at $r=0$ of both sides in \eqref{ODE CPn}. Since $\phi(r)$ is positive on $(-1,1)$, the vanishing order of the right hand side is $1$. Therefore, so is the left hand side. By this fact, since $y'(0)=0$ and $y(0)=\frac{2}{n+2}$, we must have $y''(0)\neq 0$ and
\begin{equation*}
	0\neq P(y)|_{r=0}=\bigl(\frac{2}{n+2}-\frac{2\lambda}{n+2}\bigr)^n.
\end{equation*}
Therefore, $\lambda\neq 1$. It remains to rule out the case $\lambda>1$. Assume $\lambda>1$. Since $y(0)=\frac{2}{n+2}$ and $y$ is increasing to infinity on $(0,1)$, there exists some $r^*\in (0,1)$ such that $y(r^*)=\frac{2\lambda}{n+2}$. We let $r=r^*$ in \eqref{ODE CPn} to obtain
\begin{equation*}
	0=2^{n+2} r^* \phi(r^*)^{-(n+2)}>0.
\end{equation*}
This is a contradiction. Hence we must have $\lambda<1$, and this proves Proposition \ref{prpn17}. 

\begin{rmk}
By using the same argument as above, one can prove the following result: Let $M=\mathbb{CP}^{m_1}\times \cdots \times \mathbb{CP}^{m_k}$. Let $\pi_i: M\rightarrow \mathbb{CP}^{m_i}$ be the projection from $M$ to the $i$th component $\mathbb{CP}^{m_i}$. Let $(L_i, h_i)$ be a line bundle over $\mathbb{CP}^{m_i}$ such that the dual line bundle $L_i^{-1}$ induces a \ke metric $g_i$ on $\mathbb{CP}^{m_i}$ with $\Ric(g_i)=\lambda_ig_i$. Let $(L,h)$ be the line bundle over $M$ defined by $(L, h)=\pi_1^*(L_1, h_1)\otimes\cdots\otimes\pi_k^*(L_k, h_k)$. Then the disk bundle $D(L)$ admits a complete \ke metric with negative Ricci curvature if and only if $\lambda_i<1$ for every $1\leq i\leq k$.

\end{rmk}


\subsection{An interesting case in Theorem \ref{thm 2} and proof of Proposition \ref{prpn1d10}}
A case of particular interest in Theorem \ref{thm 2} is when the base manifold $M$ is a bounded domain of holomorphy $D$.
Recall that by a result of Mok--Yau \cite{MoYau}, given $\lambda <0,$ any bounded domain of holomorphy $D \subseteq \mathbb{C}^n$ admits a unique complete K\"ahler-Einstein metric $g$ with $\mathrm{Ric}(g)=\lambda g$.  Write $z=(z_1, \cdots, z_n)$ for the coordinates of $\mathbb{C}^n$, and $g=\sum g_{i \Ol{j}} dz_i \otimes d \Ol{z_j}$. Write $G(z)=\det(g_{i \Ol{j}})$. Let $L=D \times \mathbb{C}$ be the trivial line bundle over $D$ and $H=H(z, \overline{z})$ a metric of $L$ with $g_{i \Ol{j}}=\frac{\partial^2 \log H}{\partial z_i \partial \Ol{z_j}}$ (the existence of $H$ is guaranteed by the \ke condition). In this case, the disk bundle $D(L)$ becomes a domain in $\mathbb{C}^m=\mathbb{C}^n \times \mathbb{C}, m=n+1,$
\begin{equation}\label{eqnodl}
\Omega=D(L)=\{W=(z, \xi)\in D \times \mathbb{C}: |\xi|^2 H(z, \Ol{z}) -1<0 \}.
\end{equation}

Likewise, $S(L) \subseteq D \times \mathbb{C}$ is defined by $|\xi|^2 H(z, \Ol{z})=1.$ Note all $n$ Ricci eigenvalues of the metric $g$ equal to $\lambda.$ By Theorem \ref{thm 2}, $\Omega$ admits  a complete K\"ahler-Einstein metric $\widetilde{g}$ with Ricci curvature equals to $-(m+1)$. We further use (\ref{eqnmet}) to write the metric $\widetilde{g}$ in a more concrete way. For that, we first investigate the function $\phi$ given by Proposition \ref{prpnzh2}.  Writing $\mu=\frac{2\lambda}{m+1},$ the polynomials $P(y), \hat{P}(x), Q(y), \hat{Q}(x)$ in Proposition \ref{prpnzh2} are given by,
$$P(y)=(y-\mu)^n,~~Q(y)=\frac{m+1}{m} y (y-\mu)^m-\frac{1}{m}(y-\mu)^{m+1}+c; $$
$$ \hat{P}(x)=(1-\mu x)^n,~~ \hat{Q}(x)=\frac{m+1}{m}(1-\mu x)^m-\frac{1}{m}(1-\mu x)^{m+1}+cx^{m+1}.$$
Here $c=-(\frac{2}{m+1})^{m+1}(1-\lambda)^m (1+\frac{\lambda}{m}).$ It turns out that in this case the ODE considered in Proposition \ref{prpnzh2} can be a lot simplified. Let $\xi(r)=\frac{Z(r)}{1-\mu Z(r)}$, which defines  a real analytic and even function in a neighborhood of $[-1, 1].$ With the change of variables from $Z$ to $\xi$, one can verify that the ODE in (\ref{eqnzodenz}) is reduced to (\ref{eqnzode}), and $\phi(r)=2  (\frac{r}{-Z'\hat{P}(Z)})^{\frac{1}{m+1}} Z=2  (\frac{r}{-\xi'})^{\frac{1}{m+1}} \xi.$ In addition, since  $(m+1)rZ\phi'+(m+1-2Z)\phi=0$, we have (\ref{eqnhxi})  holds. Then Proposition \ref{prpnzh2} and Theorem \ref{thm 2} (together with Proposition \ref{prpn3d5}) yield Proposition \ref{prpn2} and Theorem \ref{T1}, respectively.

\begin{prop}\label{prpn2}
There exists a unique real analytic function $\xi(r)$ on $J:=[-1, 1]$ (meaning it extends real analytically to some open interval containing $J$) such that the following conditions hold:
\begin{equation}\label{eqnzode}
r\xi'=-c\xi^{m+1}+a\xi-1~\text{on}~J, \quad \xi(1)=0.
\end{equation}
Here $c$ is as above and $a=-\frac{2 \lambda}{m}>0.$ Furthermore, the function $\phi(r)=2 \left(\frac{r}{-\xi'}\right)^{\frac{1}{m+1}}\xi$  is the function in part (2) of Proposition \ref{prpnzh2}. In addition, the following holds:
\begin{equation}\label{eqnhxi}	
 r\xi \phi'+(1-\alpha \xi)\phi=0~\text{with}~\alpha=\frac{2(1-\lambda)}{m+1}.
\end{equation}

\end{prop}

\begin{thm}\label{T1}
Let $D\subset \mathbb{C}^n$ be a bounded domain of holomorphy and $g$ be a complete K\"ahler-Einstein metric on $D$ with some negative Ricci curvature $\lambda$.
Let $G, H$ be as above and $\Omega$ be the domain in $\mathbb{C}^m$ as defined in (\ref{eqnodl}). Let $\xi$ and
$\phi$ be as in Proposition \ref{prpn2}. Set
\begin{equation}\label{eqnuh}
u(W)=(GH)^{-\frac{1}{m+1}} \phi(|\xi|H^{\frac{1}{2}}(z)).
\end{equation}
Then $u$ is real analytic in $\Omega,$  and
$\log \frac{1}{u}$ gives a potential function for a complete K\"ahler-Einstein metric on $\Omega$ with Ricci curvature $-(m+1).$
In particular, $u$ is the (global) Cheng--Yau solution in $\Omega.$ It extends real analytically across the boundary part $S(L).$
\end{thm}

\begin{rmk}\label{rmkzh}
	The ODE (\ref{eqnzode}) can be explicitly solved when $\lambda =-m.$ In this case, $a=\alpha=2$ and $c=0.$ A direct computation yields $\xi(r)=\frac{1-r^2}{2}, Z(r)=\frac{1-r^2}{2+\mu -\mu r^2}$ and $\phi(r)=1-r^2.$
\end{rmk}

While it seems difficult to solve (\ref{eqnzode}) explicitly in general, in this section we will study the rationality of $\xi$ (equivalently, the rationality of $Z$) and find some applications. 

\begin{prop}\label{prpnrn}
The following statements are equivalent: (1) $\xi$ is rational; (2) $\phi^{m+1}$ is rational; (3) $\lambda=-m.$
\end{prop}

{\em Proof of Proposition \ref{prpnrn}:} We first prove (1) and (2) are equivalent. By Proposition \ref{prpn2},  $\phi^{m+1}=\big(\frac{r}{-\xi'} \big)(2 \xi)^{m+1}.$ Consequently, the rationality of $\xi$ implies that of $\phi^{m+1}$. On the other hand, by (\ref{eqnhxi}), $\frac{1}{\xi}=\frac{\alpha \phi-r\phi'}{\phi}=\alpha-r\frac{\phi'}{\phi}=\alpha-\frac{r}{m+1} \frac{(\phi^{m+1})'}{\phi^{m+1}}.$ Hence if $\phi^{m+1}$ is rational, then so is $\xi$. This proves the equivalence of (1) and (2). By Remark \ref{rmkzh}, we see (3) implies (1) and (2). It then suffices to show
(1) implies (3).  Suppose $\xi$ is rational and write $\xi(r)=\frac{p(r)}{q(r)}$ with $p$ and $q$ coprime. Putting this into the ODE in (\ref{eqnzode}), we obtain
\begin{equation}\label{eqnxipq}
rq^{m-1}(p'q-pq')=-cp^{m+1}+ap q^m -q^{m+1}.
\end{equation}
Since $m \geq 2,$ we have $q$ divides $p^{m+1}.$ But this is only the case when $q$ is constant. Therefore we can assume $q=1$ and thus $\xi=p.$ Since $\xi$ cannot be constant by (\ref{eqnzode}), we have the degree of $p$ is at least $1$. Then by comparing the degree of
both sides of (\ref{eqnxipq}), we must have $c=0$, and consequently $\lambda=-m.$ This finishes the proof of Proposition \ref{prpnrn}. \qed



We are now ready to prove Proposition \ref{prpn1d10}.

{\em Proof of Proposition \ref{prpn1d10}:} Let $D$ be the $n-$dimensional complex unit ball $\{z \in \mathbb{C}^n: |z|^2<1 \}.$ Let $m=n+1$, $\lambda <0$ and $H=(1-|z|^2)^{\frac{m}{\lambda}}.$
Let $g=\sum g_{i \Ol{j}} dz_i \otimes d \Ol{z_j}$ be the induced metric with $g_{i \Ol{j}}=\frac{\partial^2 \log H}{\partial z_i \partial \Ol{z_j}}$.
 It is a well-known fact that $g$ is a complete K\"ahler-Einstein metric on $D$ with Ricci curvature equals to $\lambda$. In addition, writing $G(z)=\det(g_{i \Ol{j}})$, we have $G=(\frac{1}{-\lambda})^n\frac{m^n}{(1-|z|^2)^{m}}.$ To keep notation simple, we will write $\mu=\mu(n)=(\frac{1}{-\lambda})^n m^n$, and thus $\mu^{-1} G=(1-|z|^2)^{-m}, H=(\mu^{-1} G)^{-\frac{1}{\lambda}}.$  Note the disk bundle $\Omega$ as defined in (\ref{eqnodl}) is now reduced to
$$\Omega=\Omega_{\lambda}:=\{W=(z, \xi)\in \mathbb{C}^{n} \times \mathbb{C}: z \in \mathbb{B}^n, |\xi|^2 (1-|z|^2)^{\frac{m}{\lambda}}<1 \}.$$
Set $p^*=\frac{1}{pm}.$ If we pick $\lambda=-\frac{1}{p^*}$, then we have $\Omega_{-\frac{1}{p^*}}=E_p$. Therefore by Theorem \ref{T1},
writing $u(W)$ for the Cheng--Yau solution of $E_p$, we have
\begin{equation}\label{eqnv}
u(W)=\mu^{\frac{p^*}{m+1}} (G(z))^{-\frac{p^*+1}{m+1}}  \phi\big(\mu^{-\frac{p^*}{2}}|\xi| (G(z))^{\frac{p^*}{2}}\big).
\end{equation}
We also recall the Bergman kernel $K$ of $E_p$ was computed by D'Angelo \cite{DA94}:
$$K(W, \Ol{W})=\sum_{i=0}^{m} c_i\frac{(1-|z|^2)^{-m+\frac{i}{p}}}{\big((1-|z|^2)^{\frac{1}{p}}-|\xi|^2 \big)^{1+i}}.$$
Here $c_i$ are constants depending on $i, m$ and $p$. Note $(1-|z|^2)^{\frac{1}{p}}=(\mu^{-1}G)^{-p^*}.$ Consequently,
\begin{equation}\label{eqnk}
\begin{split}
K(W, \Ol{W})=&\sum_{i=0}^{m} c_i \frac{(\mu^{-1}G) (\mu^{-1}G)^{-ip^*}}{\big((\mu^{-1}G)^{-p^*}-|\xi|^2 \big)^{1+i}}\\
=&(\mu^{-1}G)^{1+p^*} \sum_{i=0}^{m} c_i \frac{1}{\big(1-(\mu^{-1}G)^{p^*}|\xi|^2\big)^{1+i}}.
\end{split}
\end{equation}
	
To establish Proposition \ref{prpn1d10}, we only need to show that if the Bergman metric of $E_p$ is K\"ahler-Einstein, then $p=1.$ For that, we assume the Bergman metric $g_{\mathrm{B}}$ of $E_p$ is K\"ahler-Einstein and first follow the work of Fu-Wong \cite{FuWo} to compute the volume form of $g_{\mathrm{B}}$.  Note a generic boundary point of $E_p$ is smooth and  strictly pseudoconvex (indeed spherical). Fix any  strictly pseudoconvex  boundary point $W^*$. By using Fefferman's expansion for the Bergman kernel near $W^*$ and the argument in Cheng--Yau (\cite{CheYau}, page 510), we see the Ricci curvature of $g_{\mathrm{B}}$ at $W \in E_p$ tends to $-1$ as $W$ approaches $W^*$.     By the K\"ahler-Einstein assumption,
the Ricci curvature of $g_{\mathrm{B}}$ must equal to $-1$. Then by Proposition 1.2 in \cite{FuWo}, the determinant of $g_{\mathrm{B}}$ equals the Bergman kernel up to a constant multiple. On the other hand, the volume form of two  complete \ke metrics, $\bigl( (-\log u)_{i\bar{j}} \bigr)$ and $g_{\mathrm{B}}$,  on $E_p$ can only differ by a constant multiple.
Consequently, we have  $u^{-(m+1)}(W)=c K$ for some constant $c >0$. Combining this with (\ref{eqnv}) and (\ref{eqnk}), we obtain
	\begin{equation}
	\phi^{-(m+1)}(r)=\frac{c}{\mu} \sum_{i=0}^{m} c_i \frac{1}{(1-r^2)^{1+i}}.
	\end{equation}
	This implies $\phi^{m+1}$ is rational. By proposition \ref{prpnrn}, $\lambda=-\frac{1}{p^*}=-m.$ That is, $p=1.$	\qed

\section{\k surface with constant scalar curvature and Proof of Theorem \ref{thm 3}}\label{Sec 3}

\subsection{Preliminaries from CR and pseudohermitian geometry}\label{Sec preliminary from CR}

In this section, we briefly review some background materials on the CR and pseudohermitian geometry, which will be used in the proof later. For more details, we refer readers to \cite{Web77b}, \cite{Lee86} and \cite{Wang}.

3.1.1 \textbf{Pseudohermitian geometry.} Let $X$ be a smooth orientable manifold of real dimension $2n+1$. An almost CR structure on $X$ is a pair $(H(M), J)$, where $H(M)$ is a subbundle of $TX$ with rank $2n$, and $J$ is an almost complex structure on $H(M)$. We can then decompose $H(M)\otimes \mathbb{C}=T^{1,0}X\oplus T^{0,1}X$, where $J$ acts by $i$ on $T^{1,0}X$ and by $-i$ on $T^{0,1}X=\oo{T^{1,0}X}$. $(X, H(M), J)$ is called a \emph{CR manifold} if $T^{1,0}X$ satisfies the integrability condition $[T^{1,0}X, T^{1,0}X]\subset T^{1,0}X$.

Since $X$ is orientable and the complex structure $J$ defines an orientation on $H(M)$, the annihilator subbundle $H(M)^{\perp}:=\mathrm{Ann}(H(M))\subset T^*X$ is also orientable. A section $\theta$ of $H(M)^{\perp}$ is a real 1-form. If $\theta$ is nowhere vanishing, then we shall refer to such $\theta$ as a \emph{contact form} for $H(M)$. (The assumption of strict pseudoconvexity below justifies this terminology as $\theta$ will then be a contact form in the classical sense.) A CR structure $(X, H(M), J)$ together with a choice of contact form $\theta$ is referred to as a \emph{pseudohermitian structure}, a terminology that is explained and motivated below. Let $\omega=d\theta$. Then $\omega(\cdot, J\cdot)$ defines a symmetric bilinear form on $H(M)$. If this symmetric bilinear form is positive definite, then the CR manifold $X$ is said to be \emph{strictly (or strongly) pseudoconvex}. The \emph{Reeb vector field} associated with the contact form $\theta$ is the unique vector field $T$ on $X$ such that $\theta(T)=1$ and $T \intprod  d\theta=0$.

From now on, $(X, H(M), J)$ is always assumed to be a strictly pseudoconvex CR manifold and $\theta$ is a given choice of  a contact form. Let $\{Z_1, \cdots Z_{n}\}$ be a local frame of $T^{1,0}X$ and $Z_{\bar{\alpha}}=\oo{Z_{\alpha}}$ for $1\leq \alpha\leq n$. Then $\{T, Z_1, \cdots, Z_n, Z_{\bar{1}}, \cdots, Z_{\bar{n}}\}$ is a local frame of $TX\otimes\mathbb{C}$. We denote the dual frame by $\{\theta, \theta^1, \cdots, \theta^n, \theta^{\bar{1}}, \cdots, \theta^{\bar{n}}\}$. Note this is an admissible coframe in the terminology of Webster \cite{Web78}, and $\{\theta^1, \cdots, \theta^n\}$ can be identified with a coframe for $(T^{1,0})^* X$. By the integrability condition and the fact that $T$ is the Reeb vector field, we have
\begin{equation*}
d\theta=i g_{\alpha\bar{\beta}}\theta^{\alpha}\wedge\theta^{\bar{\beta}},
\end{equation*}
where $g_{\alpha\bar{\beta}}$ is a positive definite Hermitian matrix. The induced Hermitian form on $T^{1,0} X$ is called the {\em Levi form}. Some authors refer to the symmetric form in the previous paragraph as the Levi form, but one is the complexification (and restriction to $T^{1,0}X$) of the other so both contain the same information.

Webster introduced the notion of a pseudohermitian structure in \cite{Web78}. By fixing a contact form $\theta$, he showed that there are uniquely determined $1$-forms $\omega_{\alpha}{}^{\beta}, \tau^{\beta}$ on $X$ satisfying
\begin{align}\label{structure equation of CR manifold}
\begin{split}
&d\theta^{\beta}=\theta^{\alpha}\wedge \omega_{\alpha}{}^{\beta}+\theta\wedge\tau^{\beta},\\
\omega_{\alpha\bar{\beta}}&+\omega_{\bar{\beta}\alpha}=dg_{\alpha\bar{\beta}}, \qquad \tau_{\alpha}\wedge\theta^{\alpha}=0,
\end{split}
\end{align}
where the complex conjugate is reflected in the index (e.g., $\omega_{\bar{\beta}\alpha}=\oo{\omega_{\beta\bar{\alpha}}}$), and we have used the Levi form $g_{\alpha\bar{\beta}}$ to raise and lower the index. These conventions will be also adapted in the following. We can write $\tau_{\alpha}=A_{\alpha\gamma}\theta^{\gamma}$ with $A_{\alpha\gamma}=A_{\gamma\alpha}$, and $(A_{\alpha\gamma})$ is called the \emph{pseudohermitian torsion}.
The {\em Tanaka-Webster connection} is defined by
\begin{equation*}
\nab Z_{\alpha}=\omega_{\alpha}{}^{\beta}\otimes Z_{\beta}, \quad \nab Z_{\bar{\alpha}}=\omega_{\bar{\alpha}}{}^{\bar{\beta}}\otimes Z_{\bar{\beta}}, \quad \nab T=0.
\end{equation*}
In the case $(A_{\alpha\gamma})$ all vanish, we say that the Tanaka-Webster connection $\nab$ is (pseudohermitian) torsion free. By the work of Webster \cite{Web78} (cf. \cite{Wang}), the curvature form is given by
\begin{align*}
\Omega_{\alpha}{}^{\beta}:=&d\omega_{\alpha}{}^{\beta}-\omega_{\alpha}{}^{\gamma}\wedge \omega_{\gamma}{}^{\beta}
\\
=&-R_{\mu\bar{\nu}\alpha}{}^{\beta}\theta^{\mu}\wedge\theta^{\bar{\nu}}+\nab^{\beta}A_{\alpha\gamma}\theta^{\gamma}\wedge\theta-\nab_{\alpha}A_{\bar{\gamma}}{}^{\beta} \theta^{\bar{\gamma}}\wedge\theta+i\theta_{\alpha}\wedge\tau^{\beta}-i\tau_{\alpha}\wedge\theta^{\beta},
\end{align*}
where $R_{\mu\bar{\nu}\alpha}{}^{\beta}$ (equivalently, $R_{\mu\bar{\nu}\alpha\bar{\beta}}$) is the \emph{pseudohermitian curvature}. Taking trace, we have the \emph{pseudohermitian Ricci curvature} $R_{\mu\bar{\nu}}=-R_{\mu\bar{\nu}\alpha}{}^{\alpha}$ and the \emph{pseudohermitian scalar curvature} $R=R_{\mu\bar{\nu}}g^{\mu\bar{\nu}}$.

3.1.2 \textbf{The Fefferman space.} Let $(X, H(M), J)$ be a CR manifold and $\theta$ a choice of contact form, giving the CR manifold a pseudohermitian structure as defined in $\S$ 3.1.1. We introduce the \emph{canonical line bundle} of $X$ as follows:
\begin{equation*}
K:=\{\eta\in \Lambda^{n+1}(X): V\intprod \eta=0 \mbox{ for any } V\in T^{0,1}X \}.
\end{equation*}
If $X$ obtains its CR structure from an embedding in a complex manifold $N$, then $K$ is naturally isomorphic to the restriction of the canonical line bundle $K_N$ of $(n+1, 0)$ forms on $N$.

Let $K^*$ be the canonical line bundle with the zero section deleted. We define an intrinsic circle bundle $\mathcal{F}=K^*/R^+$ as the quotient of $K^*$ by the $R^+$ action $\eta\rightarrow \lambda\eta$ for any $\lambda\in R^+$. A point in $\mathcal{F}$ is an equivalent class $[\eta]$ of $(n+1,0)$-forms under multiplication by any $\lambda\in R^+$. We can fix a unique representative $\eta_0$ by imposing the volume normalization condition
\begin{equation}\label{normalization condition}
i^{n^2}n! \theta\wedge (T\intprod \eta_0)\wedge (T\intprod \oo{\eta_0})=\theta\wedge d\theta^n.
\end{equation}

In an admissible coframe $\{\theta, \theta^1, \cdots, \theta^n, \theta^{\bar{1}}, \cdots, \theta^{\bar{n}}\}$ over an open set $U\subset X$, define $\zeta$ to be 
\begin{equation*}
\zeta=(\det g_{\alpha\bar{\beta}})^{1/2} \theta\wedge\theta^1\wedge\cdots\wedge\theta^n.
\end{equation*}
One can verify $\zeta$ is a nonvanishing $(n+1,0)$ form on $U$ and satisfies \eqref{normalization condition}. Therefore, $\zeta$ serves as a local frame of the circle bundle $\mathcal{F}$ over $U$: any point $\eta_0$ of $\mathcal{F}$ over $U$ can be expressed as $e^{i\gamma}\zeta$. We will use $\gamma$ as a local fiber coordinate.

The \emph{Fefferman metric} on $\mathcal{F}$ is given in terms of the pseudohermitian invariants by
\begin{equation}\label{Fefferman metric}
\mathsf{g}=g_{\alpha\bar{\beta}}\,\theta^{\alpha}\cdot\theta^{\bar{\beta}}+2\theta\cdot\sigma, \qquad\mbox{ with }
\end{equation}
\begin{equation}\label{sigma in Fefferman metric}
\sigma=\frac{1}{n+2}\bigl(d\gamma+i\omega_{\alpha}{}^{\alpha}-\frac{i}{2}g^{\alpha\bar{\beta}}dg_{\alpha\bar{\beta}}-\frac{1}{2(n+1)}R\theta ).
\end{equation}

The Fefferman metric $\mathsf{g}$ is a Lorentz metric with signature $(2n+1, 1)$, and $(\mathcal{F}, \mathsf{g})$ is called the \emph{Fefferman space}. The conformal class of the Fefferman metric is a CR invariant of $(X, H(M), J)$. The reader is referred to, e.g., \cite{Lee86} for more details.

3.1.3 \textbf{The circle bundle and \k geometry.}
Let $(L,h)$ be a Hermitian line bundle over a complex manifold $M$ of dimension $n$. We assume $(L,h)$ is a \emph{negative line bundle}, that is, $\omega=-c_1(L,h)$ induces a \k form on $M$. Consider the circle bundle
\begin{equation*}
S(L)=S(L,h):=\{v\in L: |v|_h=1 \},
\end{equation*}
which is a $(2n+1)$-dimensional CR manifold. By an observation of Grauert, it is strictly pseudoconvex. Set $\theta=i\dbar\log|v|_h^2$, which defines a contact form on $S(L)$. In this way, we obtain a pseudohermitian structure on the strictly pseudoconvex CR manifold $S(L)$ with contact form $\theta$; we shall use the notation $(S(L), \theta)$ for this pseudohermitian structure.

Let $(U, (z^1, \cdots, z^n))$ be a local chart of $M$ on which we have a local trivialization $e_L$ of $L$. If we write $|e_L|_h^2=H(z,\bar{z})$, then locally $\omega=i\partial\dbar\log H(z,\bar{z}):=ig_{\alpha\bar{\beta}}dz^{\alpha}\wedge d\oo{z^{\beta}}$ with $g_{\alpha\bar{\beta}}=\frac{\partial^2\log H(z,\bar{z})}{\partial z^{\alpha}\partial \oo{z^{\beta}}}$.
Over $U$, $S(L)$ is given by $\{(z,\xi)\in U\times \mathbb{C}: |\xi|^2 H(z,\bar{z})=1\}$
with the contact form $\theta=i\dbar\log(|\xi|^2H)$ and $d\theta=\omega$. We have a local frame of the complexified tangent bundle $TS(L)\otimes\mathbb{C}$:
\begin{align*}
T=i\bigl(\xi\frac{\partial}{\partial\xi}-\oo{\xi}\frac{\partial}{\partial \oo{\xi}} \bigr), \qquad Z_{\alpha}=\frac{\partial}{\partial z^{\alpha}}-\xi\frac{\partial\log H}{\partial z^{\alpha}} \frac{\partial}{\partial\xi}, \qquad Z_{\oo{\alpha}}=\oo{Z_{\alpha}}.
\end{align*}
The dual frame is given by
\begin{align*}
\theta, \qquad \theta^{\alpha}=dz^{\alpha},\qquad \theta^{\oo{\alpha}}=d\oo{z^{\alpha}}.
\end{align*}
Let $\nab$ be the Tanaka-Webster connection of the pseudohermitian structure induced by $\theta$. As is well-known in this case, the Tanaka-Webster connection is (pseudohermitian) torsion free (i.e., the pseudohermitian torsion defined in $\S$ 3.1.1 vanishes) and the connection (covariant differentiation) is given by
\begin{equation*}
\nab_T Z_{\alpha}=0, \qquad \nab_{Z_{\oo{\alpha}}} Z_{\beta}=0, \qquad \nab_{Z_{\alpha}} Z_{\beta}=\widetilde{\Gamma}_{\alpha\beta}^{\gamma}Z_{\gamma},
\end{equation*}
where $\widetilde{\Gamma}_{\alpha\beta}^{\gamma}=g^{\gamma\bar{\nu}}\frac{\partial g_{\beta\bar{\nu}}}{\partial z_{\alpha}}$ (see more details in \cite{Wang}). Note that the $\widetilde{\Gamma}_{\alpha\beta}^{\gamma}$ are $S^1$-invariant and can by identified with the Christoffel symbol of the Levi-Civita connection of $(M,g)$ (where $g$ denotes the \k metric induced by the \k form $\omega$).

Let $\pi: S(L)\rightarrow M$ be the natural projection. Then $\pi_*T=0$, $\pi_*Z_{\alpha}=\frac{\partial}{\partial z^{\alpha}}$ and $\pi_*Z_{\oo{\alpha}}=\frac{\partial}{\partial \oo{z^{\alpha}}}$. Under this map, we can actually identify the Tanaka-Webster connection on the CR tangent bundle $H(S(L))\otimes\mathbb{C}$ with the Levi-Civita connection on $M$. In particular, the Tanaka-Webster connection $1$-forms $\omega_{\alpha}{}^{\beta}$ are identified with the Levi-Civita connection $1$-form of $(M, g)$. By the work of Webster \cite{Web78} (cf. \cite{Wang}), the pseudohermitian curvatures are given by
\begin{align*}
R_{\mu\bar{\nu}\alpha\bar{\beta}}=\frac{\partial^2 g_{\alpha\bar{\beta}}}{\partial z_{\mu}\partial \oo{z_\nu}}-g^{\gamma\bar{\delta}} \frac{\partial g_{\gamma\bar{\beta}}}{\partial \oo{z_{\nu}}} \frac{\partial g_{\alpha\bar{\delta}}}{\partial z_{\mu}}=R^M_{\mu\bar{\nu}\alpha\bar{\beta}}, \qquad
R_{\mu\bar{\nu}}=R^M_{\mu\bar{\nu}}, \qquad R=R^M,
\end{align*}
where $R^M_{\mu\bar{\nu}\alpha\bar{\beta}}, R^M_{\mu\bar{\nu}}$ and $R^M$ are respectively the (K\"ahler) Riemannian, Ricci and scalar curvature of $(M,\omega)$. Due to these identities, we can drop the superscript `$M$' in the curvature notations of $(M,\omega)$ without causing any ambiguity.

In computations, we shall often use normal coordinates for $g$ near a given point $p\in M$, under which we have all connection $1$-forms $\omega_{\alpha}{}^{\beta}$ vanish at $p$ (i.e., all Christoffel symbols $\widetilde{\Gamma}_{\alpha\gamma}^{\beta}$ vanish) and
\begin{equation}\label{normal coordinates}
	(g_{\alpha\bar{\beta}}(p))=\begin{pmatrix}
	1 & 0\\
	0 & 1\\
	\end{pmatrix}, \qquad
	(R_{\alpha\bar{\beta}}(p))=\begin{pmatrix}
	\lambda_1 & 0\\
	0 & \lambda_2\\
	\end{pmatrix},
\end{equation}
where $\lambda_1$ and $\lambda_2$ are are the Ricci eigenvalues at $p$.

\subsection{Proof of Theorem \ref{thm 3}}
In this section, we focus on the case of \k surfaces ($n=2$) with constant scalar curvature. In this case, we can give an explicit formula of the obstruction function in terms of curvature tensors of the \k surface (Theorem \ref{obstruction tensor for csck thm} below), from which Theorem \ref{thm 3} will follow immediately.

Let $K_{\alpha\bar{\beta}}$ be the trace-modified Ricci curvature tensor defined by $K_{\alpha\bar{\beta}}=\frac{1}{4}\bigl(R_{\alpha\bar{\beta}}-\frac{1}{6}Rg_{\alpha\bar{\beta}} \bigr)$. We denote $\Lambda=K_{\alpha\bar{\beta}}K^{\alpha\bar{\beta}}$, where the indices are raised up by using $g^{\alpha\bar{\beta}}$.

\begin{thm}\label{obstruction tensor for csck thm}
	If $(M,g)$ is a \k surface with constant scalar curvature, then the obstruction function $\mathcal{O}$ of the circle bundle $S(L)$ up to a nonzero constant multiple is given by
	\begin{equation}\label{obstruction tensor for csck}
		\Delta^2 \Lambda+R^{\alpha\bar{\beta}} \nab_{\alpha}\nab_{\bar{\beta}} \Lambda,
	\end{equation}
	where $\Delta=g^{\alpha\bar{\beta}}\nab_{\alpha}\nab_{\bar{\beta}}$ is the (complex) Laplacian on $M$.
\end{thm}

We first sketch the idea to prove Theorem \ref{obstruction tensor for csck thm} here. In the seminal work of Fefferman and Graham \cite{FeGr12}, they introduced the obstruction tensor $\mathcal{O}_{ij}$ on an even dimensional conformal manifold. The obstruction tensor is a conformally invariant symmetric $2$-tensor. Consider a strictly pseudoconvex CR manifold $X$ with pseudohermitian structure $(X, \theta)$ and let $(\mathcal{F}, \mathsf{g})$ be its Fefferman space. In \cite{GrHi08}, Graham and Hirachi pointed out the relation between the obstruction function $\mathcal{O}$ of $X$ and the obstruction tensor $\mathcal{O}_{ij}$ of its Fefferman space $(\mathcal{F}, \mathsf{g})$.
\begin{thm}[Proposition 3.6 in \cite{GrHi08}] \label{GrHi thm}
	Let $\pi: \mathcal{F}\rightarrow X$ be the natural projection. Then the pullback of $\mathcal{O}\, \theta^2$ is the obstruction tensor of $(\mathcal{F}, \mathsf{g})$ up to a nonzero constant multiple.
\end{thm}

In our situation, the CR manifold $X$ is $S(L)$ and $(\mathcal{F}, \mathsf{g})$ is its Fefferman space. Let the local frame $\{T, Z_{\alpha}, Z_{\oo{\alpha}}\}$ on an open neighborhood $U\subset X$ and its dual frame $\{\theta, \theta^{\alpha}, \theta^{\oo{\alpha}} \}$ be as introduced in $\S$ \ref{Sec preliminary from CR}. We shall identify $\theta, \theta^{\alpha}, \oo{\theta^{\alpha}}$ with their pullback $\pi^*\theta, \pi^*\theta^{\alpha}, \pi^*\oo{\theta^{\alpha}}$ on $\mathcal{F}$ respectively. Putting them together with $\sigma$ introduced in \eqref{sigma in Fefferman metric}, we obtain a local coframe on $\mathcal{F}$: 
\begin{equation}\label{Fefferman space dual frame}
	(\theta^0, \cdots, \theta^{2n+1}):=(\theta, \theta^1, \cdots, \theta^n, \theta^{\oo{1}}, \cdots, \theta^{\oo{n}}, \sigma).
\end{equation}
By a straightforward computation, this is the dual frame of
\begin{equation}\label{Fefferman space frame}
	(X_0, \cdots, X_{2n+1}):=(N, Y_1, \cdots, Y_n, Y_{\oo{1}}, \cdots, Y_{\oo{n}}, (n+2)S),
\end{equation}
where $S$ is the vector field induced by the $S^1$ action on $\mathcal{F}$, and
\begin{align*}
		N=T+\frac{R}{2(n+1)} S, \quad Y_{\alpha}=Z_{\alpha}-\frac{i}{2} g^{\gamma\bar{\delta}} (Z_{\alpha} g_{\gamma\bar{\delta}}) S, \quad Y_{\bar{\alpha}}=\oo{Y_{\alpha}}, \quad 1\leq \alpha \leq n.
\end{align*}

\begin{rmk}\label{identification of tensors on F and X rmk} 
	For any function $f$ on $S(L)$, we can lift it to an $S^1$-invariant function $\pi^*f$ on $\mathcal{F}$. Then $X_i, 0\leq i\leq 2n$, (i.e., $N, Y_{\alpha}, Y_{\bar{\alpha}}$, $1\leq \alpha\leq n$,) acting on $\pi^*f$, respectively, is the same as the lift of $Tf, Z_{\alpha}f, Z_{\bar{\alpha}}f$ to $\mathcal{F}$. On the other hand, $X_{2n+1}$ always annihilates $\pi^*f$. These facts will be often used in the following computations without being explicitly pointed out. We will also identify $\pi^*f$ with $f$ when there is no ambiguity.
\end{rmk}



In this local coframe \eqref{Fefferman space dual frame}, 
Theorem \ref{GrHi thm} yields that the only possibly nonzero component of the obstruction tensor $\mathcal{O}_{ij}$ is the $\theta^2$ component, i.e., $\mathcal{O}_{00}$. Furthermore,  under the map $\pi: S(L)\rightarrow M$, the obstruction function $\mathcal{O}$ equals $\mathcal{O}_{00}$ up to a nonzero constant multiple. Therefore, to find $\mathcal{O}$  it suffices to compute $\mathcal{O}_{00}$.

We will use the letters $i, j, k$ to denote indices ranging between $0$ and $2n+1$, and use Greek letters $\alpha, \beta, \gamma$ to denote indices ranging between $1$ and $n$. Let $D$ be the Levi-Civita connection of $(\mathcal{F}, \mathsf{g})$. Let $\rho_{ij}$ and $\rho$ be respectively the Ricci and scalar curvature of $(\mathcal{F}, \mathsf{g})$.
When $M$ is a \k surface ($n=2$), we have $\dim_{\mathbb{R}}\mathcal{F}=6$. In this case, Graham and Hirachi \cite{GrHi05} gave an explicit formula for the obstruction tensor in terms of certain curvatures and their covariant derivatives.
\begin{thm}[\cite{GrHi05}]
	When $\dim_{\mathbb{R}}\mathcal{F}=6$, one has
	\begin{align}\label{obstruction tensor for n=6}
	\begin{split}
		\mathcal{O}_{ij}=B_{ij},_k{}^k - 2W_{kijl}&B^{kl} -4P_k{}^kB_{ij} +8P^{kl}C_{(ij)k},_l
		-4C^k{}_i{}^lC_{ljk} \\
		&+2C_i{}^{kl}C_{jkl} +4P^k{}_k,_lC_{(ij)}{}^l
		-4W_{kijl}P^k{}_mP^{ml},
	\end{split}
	\end{align}
	where $W_{ijkl}$ is the Weyl curvature tensor of the Fefferman space $(\mathcal{F}, \mathsf{g})$; $P_{ij}, C_{ijk}$ and $B_{ij}$ are respectively the Schouten, Cotton and Bach tensor of $(\mathcal{F}, \mathsf{g})$ defined by
	\begin{equation*}
		P_{ij}=\frac{1}{4}\bigl(\rho_{ij}-\frac{\rho}{10}\mathsf{g}_{ij}\bigr), \qquad C_{ijk} = P_{ij},_k - P_{ik},_j,\qquad B_{ij} = P_{ij},_k{}^k
		- P_{ik},_j{}^k - P^{kl}W_{kijl}.
	\end{equation*}
\end{thm}

The parentheses in $C_{(ij)k},_l $ and $C_{(ij)}{}^l$ above indicate symmetrization over the index pair $i, j$: $C_{(ij)k},_l=\frac{1}{2}C_{ijk},_l+\frac{1}{2}C_{jik},_l$, and likewise for $C_{(ij)}{}^l$.

By Lee's work \cite{Lee86}, the Fefferman metric $\mathsf{g}$ on $\mathcal{F}$ can be characterized by Webster's pseudohermitian invariants on $X$. As a result, we can translate the curvature tensors in \eqref{obstruction tensor for n=6}, such as the Schouten, Cotton, Weyl and Bach tensors, into the pseudohermitian curvatures on $X$. This process consists of standard, but technical calculations;  we will therefore leave it to Appendix $\S$ \ref{Sec Appendix}. It is worth to mention that every $S^1$-invariant tensor on $\mathcal{F}$ can be identified with a tensor on $X$. We shall then express the component $\mathcal{O}_{00}$ in terms of the pseudohermitian curvatures. As pointed out in $\S$ \ref{Sec preliminary from CR}, the pseudohermitian curvatures on $X$ can be identified with the corresponding curvatures of the \k manifold $(M,\omega)$. This allows us to write $\mathcal{O}_{00}$ in terms of the curvatures on $(M,\omega)$  and obtain the desired result in Theorem \ref{obstruction tensor for csck thm}. However, a complete computation would {\em a priori} be quite laborious. We make a key observation that Theorem \ref{thm 1} can be applied to much simplify the process.


Before proving Theorem \ref{obstruction tensor for csck thm}, let us set up some notations. Denote by  $\mathcal{T}=\mathcal{T}(R_{\alpha\bar{\beta}}, g_{\alpha\bar{\beta}})$  the vector space of all pseudohermitian tensors which are linear combinations of the products of $R_{\alpha\bar{\beta}}, g_{\alpha\bar{\beta}}$ and their contractions. For example, all the following terms
\begin{equation*}
R=R_{\alpha\bar{\beta}} g^{\alpha\bar{\beta}},\quad K_{\alpha\bar{\beta}}=\frac{1}{4} \bigl( R_{\alpha\bar{\beta}}-\frac{1}{6}R g_{\alpha\bar{\beta} }\bigr),\quad K_{\alpha}{}^{\gamma}K_{\gamma\bar{\beta}}, \quad \Lambda=K_{\alpha\bar{\beta}}K^{\alpha\bar{\beta}}
\end{equation*}
are in $\mathcal{T}$. Note that neither the full Riemannian curvature $R_{\alpha\bar{\beta}\gamma\bar{\delta}}$ nor any covariant derivative of $R_{\alpha\bar{\beta}\gamma\bar{\delta}}$ is allowed in $\mathcal{T}$. For two tensors $A$ and $B$, we write $A=B+\mathcal{T}$ if $A-B\in \mathcal{T}$.

Let $\lambda_1 \leq \lambda_2$ be the Ricci eigenvalue functions on $M$. That is,  for any $p\in M$, $\lambda_1(p)$ and $\lambda_2(p)$ denote the two eigenvalues of the endomorphism $R_{\alpha}{}^{\beta}|_p: T^{1,0}_pM\rightarrow T^{1,0}_pM$.  Denote by $\mathcal{P}=\mathcal{P}(\lambda_1,\lambda_2)$ the vector space of symmetric polynomials of $\lambda_1$ and $\lambda_2$ (which are well-defined smooth functions on $M$). For example, all of $\lambda_1+\lambda_2, \lambda_1^2+\lambda_2^2$ and $\lambda_1\lambda_2$ are elements in $\mathcal{P}$. For another example, the complete contraction of any tensor in $\mathcal{T}$ is always in $\mathcal{P}$. This is because we can choose the normal coordinates at a given point $p$ to make \eqref{normal coordinates} hold.
For two functions $\Phi$ and $\Psi$ on $M$, we write $\Phi=\Psi+\mathcal{P}$ if $\Phi-\Psi\in \mathcal{P}$.

We are in the position to prove Theorem \ref{obstruction tensor for csck thm}. In the proof, we will use some results and formulas from Appendix $\S$ \ref{Sec Appendix}.

{\em Proof of Theorem \ref{obstruction tensor for csck thm}:} We first need the following lemma on the covariant derivative of the Bach tensor on $(\mathcal{F}, \mathsf{g})$.
\begin{lemma}\label{Bach tensor second order derivative lemma}
	\begin{equation*}
	\mathsf{g}^{kl}B_{00},_{kl}=8\Delta^2\Lambda+4R^{\alpha\bar{\beta}}\nab_{\alpha}\nab_{\bar{\beta}}\Lambda+\frac{10}{3}R\Delta\Lambda+\mathcal{P}.
	\end{equation*}
\end{lemma}

\begin{proof}
	We work with the local frame $\{X_0, \cdots, X_{2n+1} \}$ with $n=2$. By \eqref{Fefferman metric in matrix} we have
	\begin{align*}
		\mathsf{g}^{kl}B_{00},_{kl}=B_{00},_{0\,2n+1}+B_{00},_{2n+1\,0}+2g^{\alpha\bar{\beta}}B_{00},_{\alpha\bar{\beta}}+2g^{\bar{\beta}\alpha}B_{00},_{\bar{\beta}\alpha}.
	\end{align*}
	Let us compute these four terms one by one. Let $\Gamma_{ij}^{k}$ be the Christoffel symbol of the Levi-Civita connection $D$ on $(\mathcal{F}, \mathsf{g})$. Since $\Gamma_{00}^p=\Gamma_{2n+1\,0}^p=\Gamma_{0\,2n+1}^p=0$ for any $p$ by \eqref{christoffel symbols}, we have
	\begin{align*}
		B_{00,0\,2n+1}=X_{2n+1}B_{00,0} \qquad B_{00,2n+1\,0}=X_0 B_{00,2n+1}.
	\end{align*}
	For the same reason, we also have $B_{00},_0=X_0 B_{00}$ and $B_{00},_{2n+1}=X_{2n+1} B_{00}$. By Corollary \ref{Bach tensor n=2} and Remark \ref{identification of tensors on F and X rmk}, it follows that $B_{00},_0=B_{00},_{2n+1}=0$ and thus $B_{00},_{00}=B_{00},_{2n+1\,0}=0$.
	
	We note $g^{\alpha\bar{\beta}}B_{00,\alpha\bar{\beta}}$ and $g^{\bar{\beta}\alpha}B_{00,\bar{\beta}\alpha}$ are the conjugate of each other. Thus it is sufficient to compute either one of them. Since the Bach tensor is symmetric,
	\begin{align*}
		B_{00},_{\alpha\bar{\beta}}=X_{\bar{\beta}} B_{00},_{\alpha}-\Gamma_{\bar{\beta}0}^pB_{p0},_{\alpha}-\Gamma_{\bar{\beta}0}^pB_{0p},_{\alpha}-\Gamma_{\bar{\beta}\alpha}^pB_{00},_{p}
		=X_{\bar{\beta}} B_{00},_{\alpha}-2\Gamma_{\bar{\beta}0}^pB_{0p},_{\alpha}-\Gamma_{\bar{\beta}\alpha}^pB_{00},_{p}.
	\end{align*}
	By \eqref{christoffel symbols} and the fact $B_{00},_0=B_{00},_{2n+1}=0$, we get
	\begin{align}\label{B 2nd order derivative 2}
	\begin{split}
	B_{00},_{\alpha\bar{\beta}}
	=X_{\bar{\beta}} B_{00},_{\alpha}-2\Gamma_{\bar{\beta}0}^{\bar{\gamma}}B_{0\bar{\gamma}},_{\alpha}-\Gamma_{\bar{\beta}\alpha}^0B_{00},_{0}-\Gamma_{\bar{\beta}\alpha}^{2n+1}B_{00},_{2n+1}
	=X_{\bar{\beta}} B_{00},_{\alpha}+2iK_{\bar{\beta}}{}^{\bar{\gamma}}B_{0\bar{\gamma}},_{\alpha}.
	\end{split}
	\end{align}  	
	
	To proceed, we need to compute the first order covariant derivatives $B_{00},_{\alpha}$ and $B_{0\bar{\gamma}},_{\alpha}$. By \eqref{christoffel symbols},
	\begin{align*}
		&B_{00},_{\alpha}=X_{\alpha}B_{00}-2\Gamma_{\alpha 0}^{\mu}B_{0\mu}=X_{\alpha}B_{00}-2iK_{\alpha}{}^{\mu}B_{0\mu},
		\\
		B_{0\bar{\gamma}},_{\alpha}=X_{\alpha}B_{0\bar{\gamma}}-&\Gamma_{\alpha 0}^{\mu}B_{\mu\bar{\gamma}}-\Gamma_{\alpha\bar{\gamma}}^0B_{00}-\Gamma_{\alpha\bar{\gamma}}^{2n+1}B_{0\,2n+1}=X_{\alpha}B_{0\bar{\gamma}}-iK_{\alpha}{}^{\mu}B_{\mu\bar{\gamma}}+\frac{i}{2}g_{\alpha\bar{\gamma}}B_{00}.
	\end{align*}

We put these identities back into \eqref{B 2nd order derivative 2} and use Corollary \ref{Bach tensor n=2} to obtain
	\begin{align*}
		B_{00},_{\alpha\bar{\beta}}=&X_{\bar{\beta}}\bigl(X_{\alpha}B_{00}-2iK_{\alpha}{}^{\mu}B_{0\mu} \bigr)+2iK_{\bar{\beta}}{}^{\bar{\gamma}}\bigl(X_{\alpha}B_{0\bar{\gamma}}-iK_{\alpha}{}^{\mu}B_{\mu\bar{\gamma}}+\frac{i}{2}g_{\alpha\bar{\gamma}}B_{00}\bigr)\\
		=&X_{\bar{\beta}}X_{\alpha}B_{00}+2\nab_{\bar{\beta}}K_{\alpha}{}^{\mu}\nab_{\mu}\Lambda+2K_{\alpha}{}^{\mu}\nab_{\bar{\beta}}\nab_{\mu}\Lambda+2K_{\bar{\beta}}{}^{\bar{\gamma}}\nab_{\alpha}\nab_{\bar{\gamma}}\Lambda-K_{\alpha\bar{\beta}}B_{00}+\mathcal{T}
	\end{align*}
	We shall trace out the index by the Fefferman metric $\mathsf{g}$. Note that $g^{\alpha\bar{\beta}}\nab_{\bar{\beta}}K_{\alpha}{}^{\mu}=0$ and $g^{\alpha\bar{\beta}}K_{\alpha\bar{\beta}}=\frac{1}{6}R$ (see Lemma \ref{K some properties lemma}). Thus, by \eqref{Fefferman metric in matrix} and Corollary \ref{Bach tensor n=2} again we have
	\begin{align*}
		\mathsf{g}^{\alpha\bar{\beta}}B_{00},_{\alpha\bar{\beta}}=&2g^{\alpha\bar{\beta}}X_{\bar{\beta}}X_{\alpha}B_{00}+8K^{\alpha\bar{\beta}}\nab_{\alpha}\nab_{\bar{\beta}}\Lambda-\frac{1}{3}RB_{00}+\mathcal{P}\\
		=&2\Delta\bigl(2\Delta\Lambda+4(\frac{1}{3}\Lambda R-\frac{1}{6^3} R^3 ) \bigr)+2\bigl(R^{\alpha\bar{\beta}}-\frac{1}{6}Rg^{\alpha\bar{\beta}}\bigr)\nab_{\alpha}\nab_{\bar{\beta}}\Lambda-\frac{2}{3}R\Delta\Lambda+\mathcal{P}
		\\
		=&4\Delta^2\Lambda+2R^{\alpha\bar{\beta}}\nab_{\alpha}\nab_{\bar{\beta}}\Lambda+\frac{5}{3}R\Delta\Lambda+\mathcal{P}.
	\end{align*}
	So the result follows from
	\begin{align*}
		\mathsf{g}^{kl}B_{00},_{kl}=2g^{\alpha\bar{\beta}}B_{00},_{\alpha\bar{\beta}}+2\oo{g^{\alpha\bar{\beta}}B_{00},_{\alpha\bar{\beta}}}=8\Delta^2\Lambda+4R^{\alpha\bar{\beta}}\nab_{\alpha}\nab_{\bar{\beta}}\Lambda+\frac{10}{3}R\Delta\Lambda+\mathcal{P}
	\end{align*}
\end{proof}

We continue to prove Theorem \ref{obstruction tensor for csck thm}. The proof divides into two steps. In the first step, we prove \eqref{obstruction tensor for csck} holds up to some zero order term, that is,
\begin{equation}\label{obstruction tensor for csck weaker}
	\mathcal{O}_{00}=8\Delta^2\Lambda+8R^{\alpha\bar{\beta}}\nab_{\alpha}\nab_{\bar{\beta}}\Lambda+\mathcal{P}.
\end{equation}
Then in the second step, by comparing with some model case which has constant Ricci eigenvalues and applying Theorem \ref{thm 1}, we conclude the zero order term represented by $\mathcal{P}$ in the above equation must be identically zero.

\textbf{Step 1.} We first prove \eqref{obstruction tensor for csck weaker}.

We will compute each term of $\mathcal{O}_{00}$ in \eqref{obstruction tensor for n=6} up to some term in $\mathcal{P}$. Since the first term is computed in Lemma \ref{Bach tensor second order derivative lemma}, we shall begin from the second term in \eqref{obstruction tensor for n=6}.

$\bullet$ \textbf{Computation of $W_{k00l}B^{kl}$.}

By \eqref{Fefferman metric in matrix} and \eqref{weyl curvature 1} in the appendix and the symmetry of the Bach tensor, we have
\begin{align*}
	W_{k00l}B^{kl}=\mathsf{g}^{\alpha\bar{\delta}}\mathsf{g}^{\gamma\bar{\beta}}W_{\alpha00\bar{\beta}}B_{\bar{\delta}\gamma}+\mathsf{g}^{\alpha\bar{\delta}}\mathsf{g}^{\gamma\bar{\beta}}W_{\bar{\beta}00\alpha}B_{\gamma\bar{\delta}}=8g^{\alpha\bar{\delta}}g^{\gamma\bar{\beta}}W_{\alpha00\bar{\beta}}B_{\gamma\bar{\delta}}.
\end{align*}
Since $B_{\gamma\bar{\delta}}\in\mathcal{T}$ by Corollary \ref{Bach tensor n=2} and $W_{\alpha00\bar{\beta}}\in\mathcal{T}$ by \eqref{weyl curvature 1}, it follows that $W_{k00l}B^{kl}\in\mathcal{P}$.

$\bullet$ \textbf{Computation of $P_k{}^kB_{00}$.}

By \eqref{schouten tensor} and \eqref{K trace}, we have
\begin{align}\label{schouten trace}
	P_k{}^k=2\mathsf{g}^{\alpha\bar{\beta}}P_{\alpha\bar{\beta}}=2g^{\alpha\bar{\beta}}K_{\alpha\bar{\beta}}=\frac{1}{3}R.
\end{align}
By Corollary \ref{Bach tensor n=2}, $B_{00}=2\Delta\Lambda+\mathcal{P}$. Therefore, $P_k{}^kB_{00}=\frac{2}{3}R\Delta\Lambda+\mathcal{P}$.

$\bullet$ \textbf{Computation of $P^{kl}C_{00k},_l$.}

By \eqref{Fefferman metric in matrix}, we have
\begin{align}\label{O00 4th term}
\begin{split}
	P^{kl}C_{00k},_l=&P_{ij}C_{00k},_l \mathsf{g}^{ik}\mathsf{g}^{jl}\\
	=&P_{00}C_{002n+1},_{2n+1}+P_{2n+1\,2n+1}C_{000},_0+4P_{\alpha\bar{\beta}}C_{00\bar{\delta}},_{\gamma}g^{\alpha\bar{\delta}}g^{\gamma\bar{\beta}}+4P_{\bar{\beta}\alpha}C_{00\gamma},_{\bar{\delta}}g^{\alpha\bar{\delta}}g^{\gamma\bar{\beta}}.
\end{split}
\end{align}
We need to compute the covariant derivatives of the Cotton tensor.
Since $\Gamma_{2n+1\,2n+1}^p=\Gamma_{2n+1\,0}^p=\Gamma_{00}^p=0$ for any $p$ by \eqref{christoffel symbols}, we have
\begin{align*}
	C_{002n+1},_{2n+1}=X_{2n+1}C_{002n+1}, \qquad C_{000},_{0}=X_{0}C_{000}.
\end{align*}
By Proposition \ref{cotton tensor prop} and Remark \ref{identification of tensors on F and X rmk}, it follows that
\begin{align*}
	C_{002n+1},_{2n+1}=C_{000},_{0}=0.
\end{align*}
Note that the last two terms of \eqref{O00 4th term} are the conjugate of each other. Thus we only need to compute the last one. By \eqref{christoffel symbols}, we have
\begin{align*}
	C_{00\gamma},_{\bar{\delta}}=&X_{\bar{\delta}}C_{00\gamma}-\Gamma_{\bar{\delta}0}^pC_{p0\gamma}-\Gamma_{\bar{\delta}0}^pC_{0p\gamma}-\Gamma_{\bar{\delta}\gamma}^pC_{00p}\\
	=&X_{\bar{\delta}}C_{00\gamma}+iK_{\bar{\delta}}{}^{\bar{\nu}}C_{\bar{\nu}0\gamma}+iK_{\bar{\delta}}{}^{\bar{\nu}}C_{0\bar{\nu}\gamma}-\frac{i}{2}g_{\gamma\bar{\delta}}C_{000}-\frac{i}{2}K_{\gamma\bar{\delta}}C_{002n+1}.
\end{align*}
We further apply Proposition \ref{cotton tensor prop} and obtain
\begin{align*}
	C_{00\gamma},_{\bar{\delta}}=\frac{1}{2}X_{\bar{\delta}}X_{\gamma}\Lambda+\mathcal{T}.
\end{align*}
Therefore, by \eqref{schouten tensor} it follows that
\begin{align*}
	4P_{\bar{\beta}\alpha}C_{00\gamma},_{\bar{\delta}}g^{\alpha\bar{\delta}}g^{\gamma\bar{\beta}}=K^{\gamma\bar{\delta}}X_{\bar{\delta}}X_{\gamma}\Lambda+\mathcal{P}=\frac{1}{4}R^{\gamma\bar{\delta}}\nab_{\gamma}\nab_{\bar{\delta}}\Lambda-\frac{1}{24}R\Delta\Lambda+\mathcal{P}.
\end{align*}
We put it back into \eqref{O00 4th term} and get
\begin{align*}
	P^{kl}C_{00k},_l=\frac{1}{2}R^{\alpha\bar{\beta}}\nab_{\alpha}\nab_{\bar{\beta}}\Lambda-\frac{1}{12}R\Delta\Lambda+\mathcal{P}.
\end{align*}

$\bullet$ \textbf{Computation of $C^k{}_0{}^lC_{l0k}$ and $C_0{}^{kl}C_{0kl}$.}

Since $C_{ijk}=0$ if at least one of the index is $2n+1$ by Proposition \ref{cotton tensor prop}, we have
\begin{align*}
	C^k{}_0{}^lC_{l0k}=&\mathsf{g}^{ik}\mathsf{g}^{jl}C_{i0j}C_{l0k}
	\\
	=&\mathsf{g}^{\alpha\bar{\beta}}\mathsf{g}^{\gamma\bar{\delta}}\bigl(C_{\alpha0\gamma}C_{\bar{\delta}0\bar{\beta}}+C_{\bar{\beta}0\gamma}C_{\bar{\delta}0\alpha}+C_{\alpha0\bar{\delta}}C_{\gamma0\bar{\beta}}+C_{\bar{\beta}0\bar{\delta}}C_{\gamma0\alpha} \bigr).
\end{align*}
By Proposition \ref{cotton tensor prop} again, $C_{\alpha0\beta}=0$ and $C_{\alpha0\bar{\beta}}\in\mathcal{T}$. It follows that $C^k{}_0{}^lC_{l0k}\in\mathcal{P}$.
Similarly, we also have
\begin{align*}
	C_0{}^{kl}C_{0kl}=\mathsf{g}^{\alpha\bar{\beta}}\mathsf{g}^{\gamma\bar{\delta}}\bigl(C_{0\alpha\gamma}C_{0\bar{\beta}\bar{\delta}}+C_{0\bar{\beta}\gamma}C_{0\alpha\bar{\delta}}+C_{0\alpha\bar{\delta}}C_{0\bar{\beta}\gamma}+C_{0\bar{\beta}\bar{\delta}}C_{0\alpha\gamma} \bigr)\in\mathcal{P}.
\end{align*}

$\bullet$ \textbf{Computation of $P^k{}_k,_lC_{00}{}^l$.}

Since $P^k{}_k=\frac{1}{3}R$ by \eqref{schouten trace} and $R$ is constant, we have $P^k{}_k,_l=0$. Therefore, $P^k{}_k,_lC_{00}{}^l=0$.

$\bullet$ \textbf{Computation of $W_{k00l}P^k{}_mP^{ml}$.}

By \eqref{weyl curvature 1} and \eqref{schouten tensor upper index}, we have
\begin{align*}
	W_{k00l}P^k{}_mP^{ml}=W_{\alpha00\bar{\beta}}P^{\alpha}{}_mP^{m\bar{\beta}}+W_{\bar{\beta}00\alpha}P^{\bar{\beta}}{}_mP^{m\alpha}=W_{\alpha00\bar{\beta}}P^{\alpha}{}_{\gamma}P^{\gamma\bar{\beta}}+W_{\bar{\beta}00\alpha}P^{\bar{\beta}}{}_{\bar{\gamma}}P^{\bar{\gamma}\alpha}.
\end{align*}
By \eqref{weyl curvature 1} again and \eqref{schouten tensor}, it follows that $W_{\alpha00\bar{\beta}}, P_{\alpha\bar{\beta}}\in\mathcal{T}$. Therefore, $W_{k00l}P^k{}_mP^{ml}\in\mathcal{P}$.

Now we have done the computations for all terms in $\mathcal{O}_{00}$. We put them back into \eqref{obstruction tensor for n=6} to obtain
\begin{align*}
	\mathcal{O}_{00}=&8\Delta^2\Lambda+4R^{\alpha\bar{\beta}}\nab_{\alpha}\nab_{\bar{\beta}}\Lambda+\frac{10}{3}R\Delta\Lambda-\frac{8}{3}R\Delta\Lambda+4R^{\alpha\bar{\beta}}\nab_{\alpha}\nab_{\bar{\beta}}\Lambda-\frac{2}{3}R\Delta\Lambda+\mathcal{P}\\
	=&8\Delta^2\Lambda+8R^{\alpha\bar{\beta}}\nab_{\alpha}\nab_{\bar{\beta}}\Lambda+\mathcal{P}.
\end{align*}
Thus \eqref{obstruction tensor for csck weaker} is proved.
Consequently, $\mathcal{O}_{00}-8\Delta^2\Lambda-8R^{\alpha\bar{\beta}}\nab_{\alpha}\nab_{\bar{\beta}}\Lambda=P$ for some $P\in \mathcal{P}$.

\textbf{Step 2.} We then show $P=0$.

Given $\lambda\in \mathbb{R}$, let $\Delta_\lambda$ be the unit disk $D(0,1)\subset \mathbb{C}$ endowed with \k metric
\begin{equation*}
	\omega_{\lambda}=\begin{dcases}
	\frac{2}{\lambda}i\partial\dbar\log(1+|z|^2) & \mbox{when } \lambda>0\\
	\frac{i}{2}\partial\dbar|z|^2 & \mbox{when } \lambda=0\\
	\frac{2}{\lambda}i\partial\dbar\log(1-|z|^2) & \mbox{when } \lambda<0.
	\end{dcases}
\end{equation*}
Then $\Delta_\lambda$ has constant Gauss curvature $\lambda$.
Given $a, b\in \mathbb{R}$, $\Delta_a\times \Delta_b$ is a \k surface with constant Ricci eigenvalues $a$ and $b$. By Theorem \ref{thm 1}, its circle bundle is obstruction flat. On the other hand, by Step 1 of the proof, we know the obstruction tensor is in the form of \eqref{obstruction tensor for csck weaker}. Since $\Lambda$ is a constant in this case,
\begin{align*}
	0=\mathcal{O}_{00}=P(a,b) \quad \mbox{ for any } a, b \in \mathbb{R}.
\end{align*}
Thus we must have $P\equiv 0$ and $\mathcal{O}_{00}=8\Delta^2\Lambda+8R^{\alpha\bar{\beta}}\nab_{\alpha}\nab_{\bar{\beta}}\Lambda$.
Combining this with Theorem \ref{GrHi thm} (see the discussion right after Remark \ref{identification of tensors on F and X rmk}),  the desired conclusion follows immediately.
\qed

We are now ready to prove Theorem \ref{thm 3}.

{\em Proof of Theorem \ref{thm 3}:}
By Remark \ref{Lichnerowicz operator rmk}, when $(M,g)$ has constant scalar curvature, the differential operator $\Delta^2 +R^{\alpha\bar{\beta}} \nab_{\alpha}\nab_{\bar{\beta}}$ appearing in \eqref{obstruction tensor for csck} is the Lichnerowicz operator $L$. Then Theorem \ref{obstruction tensor for csck thm} yields that, up to a nonzero constant multiple, $\mathcal{O}$ is given by
$L\Lambda=\nab^{\bar{\alpha}}\nab^{\bar{\beta}}\nab_{\bar{\beta}}\nab_{\bar{\alpha}}\Lambda.$
By using the normal coordinates at a given point in $M$, we have at this point
	\begin{align*}
	(g_{\alpha\bar{\beta}})=\begin{pmatrix}
	1 & 0\\
	0 & 1\\
	\end{pmatrix} \qquad
	(R_{\alpha\bar{\beta}})=\begin{pmatrix}
	\lambda_1 & 0\\
	0 & \lambda_2\\
	\end{pmatrix}.
	\end{align*}
	Thus, the trace-modified Ricci curvature is given by
	\begin{align*}
		(K_{\alpha\bar{\beta}})=\frac{1}{4}\begin{pmatrix}
		\frac{5}{6}\lambda_1-\frac{1}{6}\lambda_2 & 0\\
		0 & \frac{5}{6}\lambda_2-\frac{1}{6}\lambda_1
		\end{pmatrix}.
	\end{align*}
	It follows that
	\begin{equation*}
		\Lambda=\frac{1}{16}\bigl( (\frac{5}{6}\lambda_1-\frac{1}{6}\lambda_2)^2+(\frac{5}{6}\lambda_2-\frac{1}{6}\lambda_1)^2 \bigr)=\frac{1}{288}\bigl(13\lambda_1^2+13\lambda_2^2-10\lambda_1\lambda_2 \bigr)=\frac{13}{288}R^2-\frac{1}{8}C.
	\end{equation*}
	Since the scalar curvature $R$ is a constant, we have $L\Lambda=-8LC$. Therefore the \k metric is obstruction flat if and only if $LC=0$. If $M$ is compact, by integration by parts, we have $LC=0$ if and only if $\nab_{\bar{\beta}}\nab_{\bar{\alpha}}C=0$, which is equivalent to say that $\grad^{1,0}C=g^{\alpha\bar{\beta}}\nab_{\bar{\beta}}C\frac{\partial}{\partial z^{\alpha}}$ is a holomorphic vector field on $M$.
\qed

\subsection{Proof of Proposition \ref{prop two examples}}

Since these two examples in Proposition \ref{prop two examples} are not related, we shall prove them separately. Let us first prove part (a) in Proposition \ref{prop two examples}.

Recall the Burns-Simanca metric is defined by
\begin{equation*}
\omega=i\partial\dbar \bigl(|z|^2+\log |z|^2\bigr) \quad \mbox{ for any } z=(z_1, z_2) \in \mathbb{C}^2\setminus \{0\}.
\end{equation*}
Therefore, the metric tensor and its inverse are
\begin{equation*}
(g_{\alpha\bar{\beta}})=\begin{pmatrix}
1+\frac{|z_2|^2}{|z|^4} & -\frac{\oo{z_1}z_2}{|z|^4}\\
-\frac{\oo{z_2}z_1}{|z|^4} & 1+\frac{|z_1|^2}{|z|^4}
\end{pmatrix}
\qquad
(g^{\alpha\bar{\beta}})=\frac{|z|^2}{1+|z|^2} \begin{pmatrix}
1+\frac{|z_1|^2}{|z|^4} & \frac{\oo{z_1}z_2}{|z|^4}\\
\frac{\oo{z_2}z_1}{|z|^4} & 1+\frac{|z_2|^2}{|z|^4}
\end{pmatrix}.
\end{equation*}
A straightforward computation shows the Ricci tensor is given by
\begin{equation*}
(R_{\alpha\bar{\beta}})=\frac{1}{|z|^4(1+|z|^2)^2}\begin{pmatrix}
|z_2|^2+|z|^2(|z_2|^2-|z_1|^2) & -(2|z|^2+1)\oo{z_1}z_2\\
-(2|z|^2+1)\oo{z_2}z_1 & |z_1|^2+|z|^2(|z_1|^2-|z_2|^2)
\end{pmatrix}.
\end{equation*}
By raising up the index we obtain
\begin{equation*}
(R_{\alpha}{}^{\beta})=\frac{1}{|z|^2(1+|z|^2)^2}\begin{pmatrix}
|z_2|^2-|z_1|^2 & -2\oo{z_1}z_2\\
-2z_1\oo{z_2} & |z_1|^2-|z_2|^2\\
\end{pmatrix}.
\end{equation*}

The scalar curvature is $R=\tr (R_{\alpha}{}^{\beta})=0$ and the central curvature is $C=\det(R_{\alpha}{}^{\beta})=-\frac{1}{(1+|z|^2)^4}$. 
By a straightforward computation, we have
\begin{equation*}
\Delta C=-\frac{12}{(1+|z|^2)^5}+\frac{16}{(1+|z|^2)^6} \qquad \Delta^2C=\frac{-240|z|^2+60}{(1+|z|^2)^7}+\frac{480|z|^2-96}{(1+|z|^2)^8}.
\end{equation*}
Moreover, we also have
\begin{equation*}
R^{\alpha\bar{\beta}}\nab_{\alpha}\nab_{\bar{\beta}}C=\frac{20|z|^2-4}{(1+|z|^2)^8}.
\end{equation*}
Therefore,
\begin{equation*}
LC=\Delta^2C+R^{\alpha\bar{\beta}}\nab_{\alpha}\nab_{\bar{\beta}}C=\frac{-240|z|^2+60}{(1+|z|^2)^7}+\frac{500|z|^2-100}{(1+|z|^2)^8}.
\end{equation*}
The above equals $0$ if and only if $|z|^2=\frac{4\pm\sqrt{10}}{6}$. This yields part (a) of Proposition \ref{prop two examples}.

To show part (b) of Proposition \ref{prop two examples}, we first prove the following theorem.
\begin{thm}
	Let $(M,g)$ be a compact \k surface with constant nonnegative scalar curvature and with discrete automorphism group. If $(M,g)$ is not \ke and not locally symmetric (i.e., is not  locally isometric to any Hermitian symmetric space), then $(M,g)$ is not obstruction flat.
\end{thm}

\begin{proof}
Suppose $(M,g)$ is obstruction flat. Then by Theorem \ref{thm 3}, $\grad^{1,0}C$ is a holomorphic vector field. Since $(M,g)$ has discrete automorphism group, $\grad^{1,0}C$ must be trivial and therefore $C$ is constant on $M$. Now that both the central curvature  $C$ 
and the scalar curvature 
are constant, we conclude that the Ricci eigenvalues are constant. Since $(M,g)$ is not K\"ahler-Einstein, the two Ricci eigenvalues must be distinct. But this is impossible by Theorem 2 in \cite{ApoDraMo01} (note in \cite{ApoDraMo01}, by ``Ricci tensor has constant eigenvalues", it means the Ricci eigenvalues are constant in our terminology), as $(M,g)$ is not locally symmetric and has nonnegative scalar curvature.
\end{proof}

Now to establish part (b) of Proposition \ref{prop two examples}, it suffices to prove
\begin{lemma}
	Let $k\geq 14$. Then there is a complex surface $M$ obtained by the blow up of $\mathbb{CP}^2$ at $k$ suitably chosen points, that admits a scalar flat \k metric $g$ satisfying all the following conditions:
	\begin{itemize}
		\item[(1)] the automorphism group of $M$ is discrete;
		\item[(2)] $(M,g)$ is not K\"ahler-Einstein;
		\item[(3)] $(M,g)$ is not locally symmetric.
	\end{itemize}
\end{lemma}
\begin{proof}
First by Theorem A in \cite{RoSin05}, the complex projective plane $\mathbb{CP}^2$, blown up at $10$ suitably chosen points, $p_1, \cdots, p_{10}$ (not necessarily distinct), admits a scalar flat \k metric, and any further blow-up of the resulting complex surface $\widetilde{M}$ still admits a scalar flat \k metric. Note there exists a union $T$ of subvarieties of $\widetilde{M}$  such that $\widetilde{M}\setminus T$ is biholomorphic to $\mathbb{CP}^2\setminus \cup_{i=1}^{10}\{p_i\}$. Next we pick distinct points $q_1, \cdots, q_{k-10}$ in $\widetilde{M}\setminus T\cong \mathbb{CP}^2\setminus \cup_{i=1}^{10}\{p_i\}$ such that the first $4$ points are in general position, which means any choice of $3$ of them are linearly independent in $\mathbb{C}^3$. Then we blow up $\widetilde{M}$ at $q_1, \cdots, q_{k-10}$ to obtain a complex surface, which will be denoted by $M$. By the aforementioned result in \cite{RoSin05}, we see that $M$ admits a scalar flat \k metric $g$. Since $q_1, \cdots, q_4$ are in general position, it is well known that the automorphism group of $M$ is discrete. For the convenience of the readers, we sketch a proof here.
	
Let $\pi: M\rightarrow \mathbb{CP}^2$ be the natural projection, induced by the blowups. Let $L_i=\pi^{-1}(p_i)$ for $1\leq i\leq 10$ and $H_i=\pi^{-1}(q_j)$ for $1\leq j\leq k-10$. They are all compact subvarieties of $M$. If $f$ is an automorphism of $M$ that is sufficiently close to the identity, then $\pi\circ f(L_i)$ is in an affine neighborhood of $p_i$, and $\pi\circ f(H_j)$ is in an affine neighborhood of $q_j$. On the other hand, since $\pi\circ f$ is a proper map, $\pi\circ f(L_i)$ and $\pi\circ f(H_j)$ are complex subvarieties of $\mathbb{CP}^2$. This implies $\pi\circ f$ is constant on each $L_i$ and $H_j$. But $f$ is an automorphism of $M$ and $\pi$ is isomorphic away from $L_i$ and $H_j$. Hence we must have $f(L_i)=L_i$ and $f(H_j)=H_j$ for all $i$ and $j$.	Consequently, denoting by $W$ the union of $L_i$ and $H_j$, $f$ is a biholomorphism of $M\setminus W$. We conclude there exists an automorphism $\phi$ of $\mathbb{CP}^2$  such that $\pi\circ f=\phi\circ \pi$. Note it must hold that $\phi(p_i)=p_i$ and $\phi(q_j)=q_j$ for all $i$ and $j$. Since $q_1, \cdots, q_4$ are in general position, $\phi$ must be the identity map of $\mathbb{CP}^2$. Consequently, $f$ is the identity map of $M$. Therefore, the automorphism group of $M$ is discrete.
	
It also follows from a standard argument that $(M,g)$ is not K\"ahler-Einstein. Indeed note the Chern number of $M$ satisfies $c_1^2(M)=9-k<0$. Assume $\Ric(g)=\lambda\omega_g$. Then $\int_M \Ric(g)\wedge \Ric(g)=\lambda^2\int_M \omega_g^2\geq 0$, a plain contradiction. For more details, we refer the readers to page 236 of \cite{RoSin05}. Finally, since the blowup of a simply connected compact complex manifold is also simply-connected, we see $M$ must be simply connected. Then $M$ cannot be locally symmetric. Otherwise $M$ is (globally) biholomorphic to some Hermitian symmetric space, which contradicts (1).
\end{proof}

\section{Appendix: Computation of several tensors}\label{Sec Appendix}
In this appendix, we will compute several tensors that were used in $\S$ 3. We will first compute them in general dimension, and then restrict to the $2$-dimensional case.

Let $(L, h)$ be a negative line bundle over a complex manifold $M$ of dimension $n$. We assume $(L,h)$ is negative and let $\omega=-c_1(L,h)$ be the induced \k form on $M$. Thus, $(M, \omega)$ is a \k manifold and we denote the induced \k metric by $g$. Let $X=S(L)$ be the circle bundle, which is a strictly pseudoconvex CR manifold of dimension $2n+1$. Let the local frame $\{T, Z_{\alpha}, Z_{\oo{\alpha}}\}$ on an open neighborhood $U\subset X$ and its dual frame $\{\theta, \theta^{\alpha}, \theta^{\oo{\alpha}} \}$ be as in $\S$ \ref{Sec preliminary from CR}. Let $(\mathcal{F}, \mathsf{g})$ be the Fefferman space over $X$. We will work with the local frame $(X_0, \cdots, X_{2n+1})$ of $\mathcal{F}$, and its dual frame $(\theta^0, \cdots, \theta^{2n+1})$ defined in \eqref{Fefferman space frame} and \eqref{Fefferman space dual frame} respectively.

We will use the Roman letters $i, j, k, \cdots$ to denote indices ranging between $0$ and $2n+1$ and Greek letters $\alpha, \beta, \gamma, \cdots$ to denote indices ranging between $1$ and $n$. With respect to the above frame, by \eqref{Fefferman metric} the Fefferman metric $\mathsf{g}$ on $\mathcal{F}$ and the \k metric $g$ on $M$ are related as follows:
\begin{equation}\label{Fefferman metric in matrix}
	(\mathsf{g}_{ij})=\begin{pmatrix}
	0 & 0 & 0 &1\\
	0 & 0 & \frac{1}{2}g_{\alpha\bar{\beta}} & 0\\
	0 &\frac{1}{2} g_{\bar{\alpha}\beta} & 0 & 0\\
	1 & 0 & 0 & 0
	\end{pmatrix}.
\end{equation}

For the rest of this appendix, we will always assume the \k manifold $(M, g)$ has constant scalar curvature. We will denote by $D$ the Levi-Civita connection of the Fefferman space $(\mathcal{F}, \mathsf{g})$, and denote by $\nab$ the Tanaka-Webster connection on $X$. As discussed in $\S$ \ref{Sec preliminary from CR}, the Tanaka-Webster connection $\nab$ is (pseudohermitian) torsion free and we can also identify it with the Levi-Civita connection on $(M,g)$. In the following context, we denote by $\rho_{ijkl}, \rho_{ij}$ and $\rho$ respectively the Riemannian (Lorentzian) , Ricci and scalar curvature of $(\mathcal{F}, \mathsf{g})$.
We denote by $R_{\alpha\bar{\beta}\gamma\bar{\delta}}, R_{\alpha\bar{\beta}}$ and $R$ respectively the pseudohermitian curvature, Ricci curvature and scalar curvature of $X$; recall that these curvatures of the pseudohermitian manifold $X$ can be identified with those of the \k manifold $(M, g)$.

We now compute the Schouten, Cotton, Weyl and Bach tensors of the Fefferman metric. Before carrying out the calculations, we will need the formulas for the Christoffel symbols.
\subsection{The Christoffel symbol}
Lee computed the Levi-Civita connection on the Fefferman space $(\mathcal{F}, \mathsf{g})$. Since in our setting the Tanaka-Webster connection $\nab$ is (pseudohermitian) torsion free and the pseudohermitian scalar curvature is constant, Lee's result simplifies into the following.

\begin{prop}[Proposition 6.5 in \cite{Lee86}]
	The Levi-Civita connection $1$-form $(\phi_j{}^k)$ of the Fefferman metric $\mathsf{g}$ is given by
\begin{equation}\label{connection 1 form}
	(\phi_{j}{}^k)=\begin{pmatrix}
	0 & i\sigma^{\beta} & -i\sigma^{\bar{\beta}} & 0\\
	\frac{i}{2}\theta_{\alpha} & \phi_{\alpha}{}^{\beta} & 0 & \frac{i}{2}\sigma_{\alpha}\\
	-\frac{i}{2}\theta_{\bar{\alpha}} & 0 & \phi_{\bar{\alpha}}{}^{\bar{\beta}} & -\frac{i}{2}\sigma_{\bar{\alpha}}\\
	0 & i\theta^{\beta} & -i\theta^{\bar{\beta}} & 0
	\end{pmatrix}
\end{equation}	
where
\begin{align*}
	\phi_{\alpha}{}^{\beta}=\omega_{\alpha}{}^{\beta}+iK_{\alpha}{}^{\beta}\theta+i\delta_{\alpha}^{\beta}\sigma, \quad \sigma_{\alpha}=K_{\alpha\bar{\beta}} \theta^{\bar{\beta}}, \quad K_{\alpha\bar{\beta}}=\frac{1}{n+2} \bigl( R_{\alpha\bar{\beta}}-\frac{1}{2(n+1)}R g_{\alpha\bar{\beta} }\bigr).
\end{align*}
Here $\delta_{\alpha}^{\beta}$ is the Kronecker symbol and $\omega_{\alpha}{}^{\beta}$'s are defined in \eqref{structure equation of CR manifold}.
\end{prop}

The Christoffel symbol is defined by $D_{X_i} X_j=\Gamma_{ij}^k X_k$. It is related to the connection $1$-form as $\Gamma_{ij}^k = \phi_j{}^k(X_i)$. By this relation, we can easily find out the Christoffel symbol.
\begin{align}\label{christoffel symbols}
	\begin{split}
	(\Gamma_{j0}^k)
	=&\begin{pmatrix}
	0 & 0 & 0 & 0\\
	0 & iK_{\alpha}{}^{\beta} & 0 & 0\\
	0 & 0 & -iK_{\bar{\alpha}}{}^{\bar{\beta}} & 0\\
	0 & 0 & 0 & 0\\
	\end{pmatrix}
	\qquad \qquad
	(\Gamma_{j\gamma}^k)=\begin{pmatrix}
	0 & iK_{\gamma}{}^{\beta} & 0 & 0\\
	0 & \widetilde{\Gamma}_{\alpha\gamma}^{\beta} & 0 & 0\\
	\frac{i}{2} g_{\gamma\bar{\alpha}} & 0 & 0 & \frac{i}{2} K_{\gamma\bar{\alpha}}\\
	0 & i\delta_{\gamma}^{\beta} & 0 & 0
	\end{pmatrix}
	\\
	(\Gamma_{j\bar{\gamma}}^k)=&\begin{pmatrix}
	0 & 0 & -iK_{\bar{\gamma}}{}^{\bar{\beta}} & 0\\
	-\frac{i}{2} g_{\alpha\bar{\gamma}} & 0 & 0 & -\frac{i}{2} K_{\alpha\bar{\gamma}}\\
	0 & 0 & \widetilde{\Gamma}_{\bar{\alpha}\bar{\gamma}}^{\bar{\beta}} & 0\\
	0 & 0 & -i\delta_{\bar{\gamma}}^{\bar{\beta}} & 0
	\end{pmatrix}
	\qquad
	(\Gamma_{j\, 2n+1}^k)
	=\begin{pmatrix}
	0 & 0 & 0 & 0\\
	0 & i\delta_{\alpha}^{\beta} & 0 & 0\\
	0 & 0 & -i\delta_{\bar{\alpha}}^{\bar{\beta}} & 0\\
	0 & 0 & 0 & 0\\
	\end{pmatrix}.
	\end{split}
\end{align}

In the above, $j$ is the row index and $k$ is the column index, and $\widetilde{\Gamma}_{\alpha\gamma}^{\beta}=g^{\beta\bar{\nu}}\frac{\partial g_{\gamma\bar{\nu}}}{\partial z_{\alpha}}$ is the Christoffel symbol on \k manifold $(M, g)$. Note $\Gamma_{ij}^k-\Gamma_{ji}^k=\theta^k ([X_i, X_j])$, which does not vanish in general.

The following lemma on the tensor $K_{\alpha\bar{\beta}}$ will be used in later computations.
\begin{lemma}\label{K some properties lemma}
	\begin{equation}\label{K trace}
	K_{\alpha}{}^{\alpha}=\frac{1}{2(n+1)} R.
	\end{equation}
	\begin{equation}\label{K square}
	K_{\alpha}{}^{\mu}K_{\mu\bar{\beta}}=\frac{1}{(n+2)^2}R_{\alpha}{}^{\mu}R_{\mu\bar{\beta}}-\frac{1}{(n+1)(n+2)^2}RR_{\alpha\bar{\beta}}+\frac{1}{4(n+1)^2(n+2)^2}R^2g_{\alpha\bar{\beta}}.
	\end{equation}
	\begin{equation}\label{K first derivative}
		\nab_{\alpha}K_{\beta\bar{\gamma}}=\nab_{\beta}K_{\alpha\bar{\gamma}}, \qquad \nab_{\bar{\beta}}K_{\alpha\bar{\delta}}=\nab_{\bar{\delta}}K_{\alpha\bar{\beta}}, \qquad \nab_{\alpha}K_{\beta}{}^{\alpha}=0, \qquad \nab^{\beta} K_{\beta}{}^{\alpha}=0.
	\end{equation}
	\begin{equation}\label{K derivative and product}
		\bigl(\nab_{\alpha}K_{\gamma}{}^{\beta}\bigr) K_{\beta}{}^{\gamma}=K^{\gamma\bar{\delta}}\nab_{\alpha}K_{\gamma\bar{\delta}}=\frac{1}{2}\nab_{\alpha}\Lambda.
	\end{equation}
	\begin{equation}\label{K laplace}
	\nab^{\gamma}\nab_{\gamma}K_{\alpha\bar{\beta}}=\frac{1}{n+2}R^{\gamma\bar{\delta}}R_{\alpha\bar{\beta}\gamma\bar{\delta}}+\frac{1}{n+2}R_{\alpha}{}^{\mu}R_{\mu\bar{\beta}}.
	\end{equation}
\end{lemma}

\begin{proof}
	The first two identities follow from straightforward computations:
	\begin{align*}
	K_{\alpha}{}^{\alpha}=&K_{\alpha\bar{\beta}}g^{\alpha\bar{\beta}}=\frac{1}{n+2} \bigl( R_{\alpha\bar{\beta}}-\frac{1}{2(n+1)}R g_{\alpha\bar{\beta} }\bigr) g^{\alpha\bar{\beta}}=\frac{1}{2(n+1)}R,
	\\
	K_{\alpha}{}^{\mu}K_{\mu\bar{\beta}}=&\frac{1}{(n+2)^2}\bigl(R_{\alpha}{}^{\mu}-\frac{1}{2(n+1)}R\delta_{\alpha}^{\mu}\bigr)\bigl(R_{\mu\bar{\beta}}-\frac{1}{2(n+1)}Rg_{\mu\bar{\beta}}\bigr)\\
	=&\frac{1}{(n+2)^2}R_{\alpha}{}^{\mu}R_{\mu\bar{\beta}}-\frac{1}{(n+1)(n+2)^2}RR_{\alpha\bar{\beta}}+\frac{1}{4(n+1)^2(n+2)^2}R^2g_{\alpha\bar{\beta}}.
	\end{align*}
	
	For \eqref{K first derivative}, since the scalar curvature $R$ is assumed to be constant, $\nab_{\alpha}K_{\beta\bar{\gamma}}=\frac{1}{n+2}\nab_{\alpha}R_{\beta\bar{\gamma}}$. By the second Bianchi identity, we have $	\nab_{\alpha}K_{\beta\bar{\gamma}}=\nab_{\beta}K_{\alpha\bar{\gamma}}$. If we trace out the index $\alpha$ and $\bar{\gamma}$, then
	\begin{equation*}
		\nab_{\alpha}K_{\beta}{}^{\alpha}=\nab_{\beta}K_{\alpha}{}^{\alpha}=\frac{1}{2(n+1)}\nab_{\beta}R=0.
	\end{equation*} The other two identities in \eqref{K first derivative} can be proved similarly. To prove \eqref{K derivative and product}, we have
	\begin{equation*}
		\bigl(\nab_{\alpha}K_{\gamma}{}^{\beta}\bigr) K_{\beta}{}^{\gamma}=\frac{1}{2}\nab_{\alpha} 	\bigl(K_{\gamma}{}^{\beta} K_{\beta}{}^{\gamma} \big)=\frac{1}{2}\nab_{\alpha}\Lambda.
	\end{equation*}	
	Similarly, we can also prove $K^{\gamma\bar{\delta}}\nab_{\alpha}K_{\gamma\bar{\delta}}=\frac{1}{2}\nab_{\alpha}\Lambda$.
	
	We now prove the last identity. Since $R$ is constant, by the second Bianchi identity we have
	\begin{align*}
	\nab^{\gamma}\nab_{\gamma}K_{\alpha\bar{\beta}}=\frac{1}{n+2}\nab^{\gamma}\nab_{\gamma}R_{\alpha\bar{\beta}}=\frac{1}{n+2}\nab^{\gamma}\nab_{\alpha}R_{\gamma\bar{\beta}}.
	\end{align*}
	We use the Ricci identity to commute the covariant derivatives:
	\begin{align*}
	\nab^{\gamma}\nab_{\alpha}R_{\gamma\bar{\beta}}=\nab_{\alpha}\nab^{\gamma}R_{\gamma\bar{\beta}}+R_{\mu\bar{\beta}}R_{\alpha}{}^{\gamma\mu}{}_{\gamma}+R_{\gamma\bar{\nu}}R_{\alpha}{}^{\gamma\bar{\nu}}{}_{\bar{\beta}}=\nab_{\alpha}\nab^{\gamma}R_{\gamma\bar{\beta}}+R_{\mu\bar{\beta}}R_{\alpha}{}^{\mu}+R^{\gamma\bar{\delta}}R_{\alpha\bar{\delta}\gamma\bar{\beta}}.
	\end{align*}
	Note that by the second Bianchi identity again we have
	\begin{equation*}
	\nab_{\alpha}\nab^{\gamma}R_{\gamma\bar{\beta}}=\nab_{\alpha}\nab_{\bar{\beta}}R_{\gamma}{}^{\gamma}=\nab_{\alpha}\nab_{\bar{\beta}}R=0.
	\end{equation*}
	Finally, \eqref{K laplace} follows from the above three equations.
\end{proof}

\subsection{The Schouten tensor}
In \cite{Lee86}, Lee expressed the curvature of the Fefferman metric on $\mathcal{F}$ in terms of the pseudohermitian curvature on $X$. In our setting the Tanaka-Webster connection $\nab$ is (pseudohermitian) torsion free and the pseudohermitian scalar curvature is constant, and Lee's result simplifies into the following.
\begin{thm}[Theorem 6.6 and Theorem 6.2 in \cite{Lee86}]\label{thm on curvatures by Lee}
	The Ricci curvature tensor of $(\mathcal{F}, \mathsf{g})$ is given by
	\begin{align*}
		\rho_{ij}\theta^i\otimes \theta^j=\frac{1}{n+1} R\, \mathsf{g}_{ij}\theta^i\otimes\theta^j+nK_{\alpha\bar{\beta}}\,\theta^{\alpha}\otimes\oo{\theta^{\beta}}+nK_{\alpha\bar{\beta}}\,\oo{\theta^{\beta}}\otimes \theta^{\alpha}+2n\sigma\otimes\sigma+2K_{\alpha\bar{\beta}}K^{\alpha\bar{\beta}}\theta\otimes\theta.
	\end{align*}
	In particular, the scalar curvature of $(\mathcal{F}, \mathsf{g})$ is $\rho=\frac{2(2n+1)}{n+1}R$.
\end{thm}

The Schouten tensor is defined as
\begin{equation*}
	P_{ij}=\frac{1}{2n}\bigl(\rho_{ij}-\frac{\rho}{2(2n+1)}\mathsf{g}_{ij}\bigr).
\end{equation*}
By Theorem \ref{thm on curvatures by Lee}, we can express the Schouten tensor in terms of pseudohermitian curvatures.
\begin{align}\label{schouten tensor}
	(P_{ij})=\begin{pmatrix}
	\frac{1}{n}\Lambda & 0 & 0 & 0\\
	0 & 0 & \frac{1}{2} K_{\alpha\bar{\beta}} & 0\\
	0 & \frac{1}{2} K_{\bar{\alpha}\beta} & 0 & 0\\
	0 & 0 & 0 & 1\\
	\end{pmatrix}, \quad \mbox{where } \Lambda=K_{\alpha\bar{\beta}} K^{\alpha\bar{\beta}}.
\end{align}

We emphasize that  $P_{ij}$ is a tensor on the Fefferman space $\mathcal{F}$, and its indices are raised up or lowered down   by the Fefferman metric $\mathsf{g}$. On the other hand, $K_{\alpha\bar{\beta}}$ is a pseudohermitian curvature on $X$, whose indices are raised up or lowered down by the pseudohermitian metric $g$.
For example,
\begin{align*}
	P^{\alpha\bar{\beta}}=\frac{1}{2}\mathsf{g}^{\alpha\bar{\delta}}K_{\bar{\delta}\gamma}\mathsf{g}^{\gamma\bar{\beta}}=2g^{\alpha\bar{\delta}}K_{\bar{\delta}\gamma}g^{\gamma\bar{\beta}}=2K^{\alpha\bar{\beta}}.
\end{align*}
A straightforward computation gives
\begin{align}\label{schouten tensor upper index}
	(P^{ij})=\begin{pmatrix}
	1 & 0 & 0 & 0\\
	0 & 0 & 2 K^{\alpha\bar{\beta}} & 0\\
	0 & 2 K^{\bar{\alpha}\beta} & 0 & 0\\
	0 & 0 & 0 & \frac{1}{n}\Lambda\\
	\end{pmatrix}
\end{align}

\subsection{The Cotton tensor}

Since the Cotton tensor is defined as $C_{ijk}=P_{ij},_k-P_{ik},_j$, we first compute the covariant derivatives of the Schouten tensor $P_{ij}$.

\begin{prop}\label{Schouten tensor derivatives}
	The covariant derivatives of the Schouten tensor are given as follows.
	\begin{itemize}
		\item[(1)] $P_{00},_k=\frac{1}{n} X_k\Lambda \quad \mbox{for any $k$}$.
		
		\item[(2)] $P_{0\alpha},_k=P_{\alpha 0},_k=\begin{cases}
			\frac{i}{2}\bigl(K_{\alpha\bar{\gamma}}K^{\bar{\gamma}}{}_{\bar{\beta}}-\frac{\Lambda}{n}g_{\alpha\bar{\beta}} \bigr) & \mbox{when } k=\bar{\beta}\\
			0 & \mbox{otherwise}
		\end{cases}$.
		
		\item[(3)] $P_{0\bar{\alpha}},_k=P_{\bar{\alpha}0},_k=\begin{cases}
			-\frac{i}{2}\bigl(K_{\beta}{}^{\gamma}K_{\gamma\bar{\alpha}}-\frac{\Lambda}{n}g_{\beta\bar{\alpha}} \bigr) & \mbox{when } k=\beta\\
			0 & \mbox{otherwise}
		\end{cases}$.
		
		\item[(4)] $P_{\alpha\bar{\beta}},_k=P_{\bar{\beta}\alpha},_k=\begin{cases}
		0 & \mbox{when } k=0, 2n+1\\
		\frac{1}{2}\nabla_{\gamma}K_{\alpha\bar{\beta}} & \mbox{when } k=\gamma\\
		\frac{1}{2}\nabla_{\bar{\gamma}}K_{\alpha\bar{\beta}} & \mbox{when } k=\bar{\gamma}
		\end{cases}$.
		
		\item[(5)] $P_{\alpha\beta},_k=P_{\bar{\alpha}\bar{\beta}},_k=0 \quad \mbox{for any $k$}$.
		
		\item[(6)] $P_{ij},_k=0 \quad \mbox{if at least one of $i ,j $ and $k$ is $2n+1$}$.	
	\end{itemize}
\end{prop}

\begin{proof} We start with the covariant derivative formula
\begin{equation}\label{Pij covariant derivative}
	P_{ij},_k=X_k P_{ij}-\Gamma_{ki}^lP_{lj}-\Gamma_{kj}^lP_{il}.
\end{equation}

To prove (1), we let $i=j=0$ in the above equation and apply \eqref{schouten tensor} to obtain
\begin{equation*}
	P_{00},_k=X_kP_{00}-\Gamma_{k0}^lP_{l0}-\Gamma_{k0}^lP_{0l}=X_kP_{00}-2\Gamma_{k0}^0P_{00}.
\end{equation*}	
Since $\Gamma_{k0}^0=0$ for any $k$ and $P_{00}=\frac{1}{n}\Lambda$, it follows that $P_{00},_k=\frac{1}{n}X_k \Lambda$.

For (2), we first note $P_{\alpha 0},_k=P_{0 \alpha},_k$ as the Schouten tensor is symmetric. Therefore it is sufficient to prove the formula for $P_{0\alpha},_k$. For that, we let $i=0$ and $j=\alpha$ in \eqref{Pij covariant derivative} and get
\begin{equation*}
	P_{0\alpha},_k=X_kP_{0\alpha}-\Gamma_{k0}^lP_{l\alpha}-\Gamma_{k\alpha}^lP_{0l}=-\Gamma_{k0}^{\bar{\gamma}}P_{\bar{\gamma}\alpha}-\Gamma_{k\alpha}^0P_{00}.
\end{equation*}
By \eqref{christoffel symbols} and \eqref{schouten tensor}, we get
\begin{equation*}
	P_{0\alpha},_k=\begin{dcases}
	\frac{i}{2}K_{\bar{\beta}}{}^{\bar{\gamma}} K_{\alpha\bar{\gamma}}-\frac{i}{2n}g_{\alpha\bar{\beta}}\,  \Lambda & \mbox{when } k=\bar{\beta},\\
	0 & \mbox{otherwise}.
	\end{dcases}
\end{equation*}
Then the result in (2) follows from the symmetry $K_{\bar{\gamma}}{}^{\bar{\beta}}=K^{\bar{\beta}}{}_{\bar{\gamma}}$. We take the conjugate of the equation (2) to obtain (3).


To prove (4) we let $i=\alpha$ and $j=\bar{\beta}$ in \eqref{Pij covariant derivative} and apply \eqref{schouten tensor} to obtain
\begin{equation*}
	P_{\alpha\bar{\beta}},_k=X_kP_{\alpha\bar{\beta}}-\Gamma_{k\alpha}^lP_{l\bar{\beta}}-\Gamma_{k\bar{\beta}}^lP_{\alpha l}=\frac{1}{2}X_k K_{\alpha\bar{\beta}}-\Gamma_{k\alpha}^{\mu}P_{\mu\bar{\beta}}-\Gamma_{k\bar{\beta}}^{\bar{\nu}}P_{\alpha \bar{\nu}}.
\end{equation*}
It follows that when $k=\gamma$, by Remark \eqref{identification of tensors on F and X rmk} and \eqref{christoffel symbols},
\begin{equation*}
	P_{\alpha\bar{\beta}},_{\gamma}=\frac{1}{2}X_{\gamma}K_{\alpha\bar{\beta}}-\frac{1}{2}\Gamma_{\gamma\alpha}^{\mu}K_{\mu\bar{\beta}}-\frac{1}{2}\Gamma_{\gamma\bar{\beta}}^{\bar{\nu}}K_{\alpha \bar{\nu}}=\frac{1}{2}Z_{\gamma}K_{\alpha\bar{\beta}}-\frac{1}{2}\widetilde{\Gamma}_{\gamma\alpha}^{\mu}K_{\mu\bar{\beta}}=\frac{1}{2}\nab_{\gamma}K_{\alpha\bar{\beta}}.
\end{equation*}
When $k=\bar{\gamma}$, $P_{\alpha\bar{\beta}},_{\bar{\gamma}}=\oo{P_{\beta\bar{\alpha}},_{\gamma}}=\frac{1}{2}\nab_{\bar{\gamma}} K_{\alpha\bar{\beta}}$.
Likewise when $k=0, 2n+1$, one can show $P_{\alpha\bar{\beta}},_k=0$. The equations (5) and (6) can be proved in the same way as above, and we omit the proof.
\end{proof}

Then the explicit formulas of the Cotton tensor are given as follows.
\begin{prop}\label{cotton tensor prop}
	The following equations hold.
	\begin{itemize}
		\item[(1)] $\displaystyle  C_{ijk}=0 \quad \mbox{if at least one of $i, j$ and $k$ is $2n+1$} $.
		
		\item[(2)] $\displaystyle  C_{00k}=\frac{1}{n}X_k \Lambda \quad \mbox{for any $k$}$.
		
		\item[(3)] $C_{0\alpha k}$ equals $-\frac{1}{n} \nabla_{\alpha} \Lambda$ when $k=0$, equals $i\bigl(K_{\alpha}{}^{\gamma}K_{\gamma\bar{\beta}}-\frac{\Lambda}{n} g_{\alpha\bar{\beta}} \bigr)$ when $k=\bar{\beta}$, and equals $0$ for any other $k$.
		
		\item[(4)] $C_{0\bar{\alpha}k}$ equals $-\frac{1}{n}\nab_{\bar{\alpha}}\Lambda$ when $k=0$, equals $-i\bigl(K_{\beta}{}^{\gamma} K_{\gamma\bar{\alpha}}-\frac{\Lambda}{n} g_{\beta\bar{\alpha}} \bigr)$ when $k=\beta$, and equals $0$ for any other $k$.
				
		\item[(5)] $C_{\alpha 0 k}$ equals $\frac{i}{2} \bigl( K_{\alpha}{}^{\gamma}K_{\gamma\bar{\beta}}-\frac{\Lambda}{n} g_{\alpha\bar{\beta}} \bigr)$ for $k=\bar{\beta}$, and equals $0$ for any other $k$.
		
		\item[(6)] $C_{\bar{\alpha} 0k}$ equals $-\frac{i}{2}\bigl(K_{\beta}{}^{\gamma}K_{\gamma\bar{\alpha}}-\frac{\Lambda}{n}g_{\beta\bar{\alpha}}\bigr)$ for $k=\beta$, and equals $0$ for any other $k$.
				
		\item[(7)] $C_{\alpha\beta k}$ equals $-\frac{1}{2} \nab_{\beta} K_{\alpha\bar{\gamma}}$ for $k=\bar{\gamma}$, and equals $0$ for any other $k$.
			
		\item[(8)] $C_{\alpha\bar{\beta}k}$ equals $-\frac{i}{2}\bigl(K_{\alpha}{}^{\gamma}K_{\gamma\bar{\beta}}-\frac{\Lambda}{n}g_{\alpha\bar{\beta}} \bigr)$ for $k=0$, equals $\frac{1}{2}\nab_{\gamma} K_{\alpha\bar{\beta}}$ for $k=\gamma$, and equals $0$ for any other $k$.
	\end{itemize}
\end{prop}
\begin{proof}
	The above formulas follow directly from Proposition \ref{Schouten tensor derivatives} as $C_{ijk}=P_{ij},_k-P_{ik},_j$. To show $C_{\alpha\bar{\beta}\bar{\gamma}}=0$, we need to apply \eqref{K first derivative} to see $\nab_{\bar{\gamma}} K_{\alpha\bar{\beta}}-\nab_{\bar{\beta}} K_{\alpha\bar{\gamma}}=0$.
\end{proof}

\subsection{The Weyl tensor}
We shall compute the Weyl curvature tensor, which is defined by
\begin{equation}\label{weyl curvature def}
W_{ijkl}=\rho_{ijkl}-\bigl(P_{ik}\mathsf{g}_{jl}+P_{jl}\mathsf{g}_{ik}-P_{il}\mathsf{g}_{jk}-P_{jk}\mathsf{g}_{il} \bigr).
\end{equation}

Here $\rho_{ijkl}$ is the Riemannian curvature of the Fefferman metric,  given by the structure equation
\begin{equation}\label{Fefferman space curvature}
	\Omega_i{}^j:=d\phi_i{}^j-\phi_i{}^k\wedge \phi_k{}^j=-\frac{1}{2} \rho_{pqi}{}^j\theta^p\wedge\theta^q.
\end{equation}

\begin{prop}
	The Weyl curvature tensor satisfies
	\begin{align}\label{weyl curvature 1}
	\begin{split}
	&W_{p 00\,2n+1}=0 \quad \mbox{ for any } p, \qquad W_{\alpha 2n+1\,0\,2n+1}=0,\\
	&W_{\alpha 00\beta}=W_{\bar{\alpha} 00 \bar{\beta}}=0, \qquad\qquad
	W_{\alpha 00\bar{\beta}}=W_{\bar{\beta}00\alpha}=-\frac{1}{2} K_{\alpha}{}^{\gamma}K_{\gamma\bar{\beta}}+\frac{\Lambda}{2n} g_{\alpha\bar{\beta}},
	\\
	&W_{\alpha\beta 0\bar{\delta}}=W_{\alpha\,2n+1\,0\bar{\delta}}=0, \qquad\qquad W_{\alpha 0\gamma\bar{\delta}}=-\frac{i}{2}\nab_{\alpha}K_{\gamma\bar{\delta}}.	
	\end{split}
	\end{align}
\end{prop}

\begin{proof}
	We begin with the computation of the curvature form $\Omega_0{}^0$. By \eqref{connection 1 form}, we have
	\begin{align*}
		\Omega_0{}^0=d\phi_0{}^0-\phi_0{}^k\wedge\phi_k{}^0=-\phi_0{}^{\beta}\wedge\phi_{\beta}{}^0-\phi_0{}^{\bar{\beta}}\wedge\phi_{\bar{\beta}}{}^0
		=-i\sigma^{\beta}\wedge \frac{i}{2}\theta_{\beta}-i\sigma^{\bar{\beta}}\wedge \frac{i}{2}\theta_{\bar{\beta}}.
	\end{align*}
	Since $\sigma^{\beta}=K_{\gamma}{}^{\beta}\theta^{\gamma}$ and $\theta_{\beta}=g_{\beta\bar{\delta}}\theta^{\bar{\delta}}$, it follows that $\sigma^{\beta}\wedge\theta_{\beta}=K_{\gamma\bar{\delta}} \theta^{\gamma}\wedge\theta^{\bar{\delta}}$. Therefore,
	\begin{align*}
		\Omega_0{}^0=\frac{1}{2}K_{\gamma\bar{\delta}} \theta^{\gamma}\wedge\theta^{\bar{\delta}}+\frac{1}{2}\oo{K_{\gamma\bar{\delta}} \theta^{\gamma}\wedge\theta^{\bar{\delta}}}=0.
	\end{align*}
	Thus, by \eqref{Fefferman space curvature} the Riemannian curvature satisfies $\rho_{pq0}{}^0=0$ for any $p, q$. By lowering down the index, we get $\rho_{pq0\,2n+1}=0$ for any $p, q$. Then by \eqref{weyl curvature def},
	\begin{equation*}
		W_{pq0\,2n+1}=-\bigl(P_{p0}\mathsf{g}_{q\,2n+1}+P_{q\,2n+1}\mathsf{g}_{p0}-P_{p\,2n+1}\mathsf{g}_{q0}-P_{q0}\mathsf{g}_{p\,2n+1} \bigr) \quad \mbox{ for any } p, q.
	\end{equation*}
	In particular, we let $q=0$ and apply \eqref{Fefferman metric in matrix} to get
	\begin{align*}
		W_{p00\,2n+1}=-\bigl(P_{p0}\mathsf{g}_{0\,2n+1}+P_{0\,2n+1}\mathsf{g}_{p0}-P_{p\,2n+1}\mathsf{g}_{00}-P_{00}\mathsf{g}_{p\,2n+1} \bigr)
		=-P_{p0}+P_{00}\mathsf{g}_{p\,2n+1}.
	\end{align*}
	When $p=0$, it follows immediately $W_{000\,2n+1}=0$. When $p\neq 0$, $P_{p0}$ and $\mathsf{g}_{p\,2n+1}$ both vanish, and thus $W_{p00\,2n+1}=0$. Therefore, $W_{p00\,2n+1}=0$ for any $p$.
	
	Moreover, by \eqref{weyl curvature def}, \eqref{Fefferman metric in matrix} and \eqref{schouten tensor}, we also have
	\begin{equation*}
		W_{\alpha 2n+1\,0\,2n+1}=-\bigl(P_{\alpha 0}\mathsf{g}_{2n+1\, 2n+1}+P_{2n+1\, 2n+1}\mathsf{g}_{\alpha 0}-P_{\alpha\,2n+1}\mathsf{g}_{2n+1\,0}-P_{2n+1\,0}\mathsf{g}_{\alpha\,2n+1} \bigr)=0.
	\end{equation*}
	
We then compute $\Omega_0{}^{\alpha}$:
\begin{align*}
	\Omega_0{}^{\alpha}=d\phi_0{}^{\alpha}-\phi_0{}^k\wedge\phi_k{}^{\alpha}=&d(i\sigma^{\alpha})-\phi_0{}^{\beta}\wedge\phi_{\beta}{}^{\alpha}\\
	=&id(K_{\gamma}{}^{\alpha}\theta^{\gamma})-iK_{\gamma}{}^{\beta}\theta^{\gamma}\wedge \bigl(\omega_{\beta}{}^{\alpha}+iK_{\beta}{}^{\alpha}\theta+i\delta_{\beta}^{\alpha}\sigma \bigr).
\end{align*}	
We use the normal coordinates at a given point $p\in M$. At $p$ all $\omega_{\alpha}{}^{\beta}$ vanish and furthermore
\begin{align*}
	\Omega_0{}^{\alpha}=i\nab_{\delta}K_{\gamma}{}^{\alpha}\theta^{\delta}\wedge\theta^{\gamma}+i\nab_{\bar{\delta}}K_{\gamma}{}^{\alpha}\theta^{\bar{\delta}}\wedge\theta^{\gamma}+K_{\gamma}{}^{\beta}K_{\beta}{}^{\alpha}\theta^{\gamma}\wedge \theta+K_{\gamma}{}^{\alpha}\theta^{\gamma}\wedge\sigma.
\end{align*}

By \eqref{K first derivative}, $\nab_{\delta}K_{\gamma}{}^{\alpha}$ is symmetric in $\gamma$ and $\delta$. Thus the first term on the right hand side vanishes. Therefore,
\begin{equation*}
	\Omega_0{}^{\alpha}=-i\nab_{\bar{\delta}}K_{\gamma}{}^{\alpha}\theta^{\gamma}\wedge\theta^{\bar{\delta}}-K_{\gamma}{}^{\beta}K_{\beta}{}^{\alpha}\theta\wedge\theta^{\gamma}+K_{\gamma}{}^{\alpha}\theta^{\gamma}\wedge\sigma.
\end{equation*}

It follows that the Riemannian curvatures satisfy	
\begin{align*}
	&\rho_{\gamma 0 0}{}^{\alpha}=-K_{\gamma}{}^{\beta}K_{\beta}{}^{\alpha}, \qquad \rho_{\bar{\gamma}\, 00}{}^{\alpha}=0,  \\
	& \rho_{\gamma \delta 0}{}^{\alpha}=0, \qquad \rho_{\gamma \bar{\delta} 0}{}^{\alpha}=i\nab_{\bar{\delta}}K_{\gamma}{}^{\alpha}, \qquad  \rho_{\gamma\, 2n+1\,0}{}^{\alpha}=-K_{\gamma}{}^{\alpha}.
\end{align*}

We shall lower down the index $\alpha$. Note that the left hand sides are curvatures of the Fefferman metric, whose index is raised up or lowered down by the Fefferman metric $\mathsf{g}_{ij}$. But the right hand sides are pseudohermitian curvatures, whose index is raised up or lowered down by the pseudohermitian metric $g_{\alpha\bar{\beta}}$. By relation $\mathsf{g}_{\alpha\bar{\beta}}=\frac{1}{2}g_{\alpha\bar{\beta}}$, we obtain
\begin{align*}
	&\rho_{\gamma 0 0\bar{\beta}}=-\frac{1}{2}K_{\gamma}{}^{\mu}K_{\mu\bar{\beta}}, \qquad \rho_{\bar{\gamma} 00\bar{\beta}}=0, \\
	&\rho_{\gamma \delta 0\bar{\beta}}=0, \qquad \rho_{\gamma \bar{\delta} 0\bar{\beta}}=\frac{i}{2}\nab_{\bar{\delta}}K_{\gamma\bar{\beta}}, \qquad\rho_{\gamma\, 2n+1\,0\bar{\beta}}=-\frac{1}{2}K_{\gamma\bar{\beta}}.
\end{align*}

We can further compute the Weyl curvatures by \eqref{weyl curvature def}, \eqref{Fefferman metric in matrix} and \eqref{schouten tensor}. We first compute $W_{\alpha 00 \bar{\beta}}$.
\begin{align*}
	W_{\alpha 00\bar{\beta}}=\rho_{\alpha 00\bar{\beta}}-\bigl( P_{\alpha 0}\mathsf{g}_{0\bar{\beta}}+P_{0\bar{\beta}}\mathsf{g}_{\alpha 0}-P_{\alpha\bar{\beta}} \mathsf{g}_{00}-P_{00}\mathsf{g}_{\alpha\bar{\beta}} \bigr)=-\frac{1}{2}K_{\alpha}{}^{\gamma}K_{\gamma\bar{\beta}}+\frac{\Lambda}{2n} g_{\alpha\bar{\beta}}.
\end{align*}
By the symmetry of the Weyl curvature, we also have
\begin{equation*}
	W_{\bar{\beta} 00\alpha}=W_{\alpha 00\bar{\beta}}=-\frac{1}{2}K_{\alpha}{}^{\gamma}K_{\gamma\bar{\beta}}+\frac{\Lambda}{2n} g_{\alpha\bar{\beta}}.
\end{equation*}

Next we compute $W_{\alpha 00 \beta}$ and $W_{\bar{\alpha} 00 \bar{\beta}}$. Similarly as above, we get
\begin{align*}
	W_{\bar{\alpha} 00\bar{\beta}}=\rho_{\bar{\alpha} 00\bar{\beta}}-\bigl( P_{\bar{\alpha} 0}\mathsf{g}_{0\bar{\beta}}+P_{0\bar{\beta}}\mathsf{g}_{\bar{\alpha} 0}-P_{\bar{\alpha}\bar{\beta}} \mathsf{g}_{00}-P_{00}\mathsf{g}_{\bar{\alpha}\bar{\beta}} \bigr)=0.
\end{align*}
We take the conjugate and obtain $W_{\alpha 00 \beta}=0$. We can similarly compute $W_{\alpha\beta 0 \bar{\delta}}$ and $W_{\alpha\bar{\beta} 0 \bar{\delta}}$:
\begin{align*}
	&W_{\alpha\beta 0 \bar{\delta}}=\rho_{\alpha\beta 0 \bar{\delta}}-\bigl( P_{\alpha 0}\mathsf{g}_{\beta\bar{\delta}}+P_{\beta\bar{\delta}}\mathsf{g}_{\alpha 0}-P_{\alpha\bar{\delta}}\mathsf{g}_{\beta 0}-P_{\beta 0}\mathsf{g}_{\alpha\bar{\delta}} \bigr)=0,
	\\
	&W_{\alpha\bar{\beta} 0 \bar{\delta}}=\rho_{\alpha\bar{\beta} 0 \bar{\delta}}-\bigl( P_{\alpha 0}\mathsf{g}_{\bar{\beta}\bar{\delta}}+P_{\bar{\beta}\bar{\delta}}\mathsf{g}_{\alpha 0}-P_{\alpha\bar{\delta}}\mathsf{g}_{\bar{\beta} 0}-P_{\bar{\beta} 0}\mathsf{g}_{\alpha\bar{\delta}} \bigr)=\frac{i}{2}\nab_{\bar{\delta}} K_{\alpha\bar{\beta}}.
\end{align*}
For the last equality we have used \eqref{K first derivative}.
By taking the conjugate and using the symmetry of Weyl tensor, we also have
\begin{align*}
	W_{\alpha 0 \gamma\bar{\delta}}=\oo{W_{\bar{\alpha}0\bar{\gamma}\delta}}=\oo{W_{\delta\bar{\gamma}0\bar{\alpha}}}=-\frac{i}{2}\nab_{\alpha}K_{\gamma\bar{\delta}}.
\end{align*}
Finally we compute $W_{\alpha\, 2n+1\,0 \bar{\delta}}$ in a similar way:
\begin{align*}
	W_{\alpha\, 2n+1\,0 \bar{\delta}}=&\rho_{\alpha\, 2n+1\,0 \bar{\delta}}-\bigl( P_{\alpha 0} \mathsf{g}_{2n+1\bar{\delta}}+P_{2n+1\bar{\delta}}\mathsf{g}_{\alpha 0}-P_{\alpha\bar{\delta}}\mathsf{g}_{2n+1\,0}-P_{2n+1\,0}\mathsf{g}_{\alpha\bar{\delta}} \bigr)
	\\=&-\frac{1}{2}K_{\alpha\bar{\delta}}+\frac{1}{2}K_{\alpha\bar{\delta}}=0.
\end{align*}
\end{proof}

\begin{prop}
	The Weyl curvature tensor satisfies
	\begin{align}\label{weyl curvature 2}
	\begin{split}
	W_{\alpha\bar{\beta}\gamma\bar{\delta}}=\frac{1}{2}&R_{\alpha\bar{\beta}\gamma\bar{\delta}}+\frac{1}{2}K_{\alpha\bar{\beta}}g_{\gamma\bar{\delta}}+\frac{1}{2}K_{\gamma\bar{\delta}}g_{\alpha\bar{\beta}}+\frac{1}{2}K_{\gamma\bar{\beta}}g_{\alpha\bar{\delta}}+\frac{1}{2}K_{\alpha\bar{\delta}}g_{\gamma\bar{\beta}},
	\\ &W_{\alpha\beta\bar{\gamma}\bar{\delta}}=0,
	\qquad W_{2n+1\,\alpha\bar{\beta}\,2n+1}=W_{2n+1\,\bar{\beta}\alpha\,2n+1}=0.
	\end{split}
	\end{align}
\end{prop}

\begin{proof}
	We first compute the Riemannian curvature $\rho_{\alpha\bar{\beta}\gamma\bar{\delta}}$ of the Fefferman metric. Note
\begin{align*}
		\Omega_{\alpha}{}^{\beta}=d\phi_{\alpha}{}^{\beta}-\phi_{\alpha}{}^j\wedge\phi_j{}^{\beta}=d\phi_{\alpha}{}^{\beta}-\phi_{\alpha}{}^0\wedge\phi_0{}^{\beta}-\phi_{\alpha}{}^{\gamma}\wedge\phi_{\gamma}{}^{\beta}-\phi_{\alpha}{}^{2n+1}\wedge\phi_{2n+1}{}^{\beta}.
	\end{align*}
	By \eqref{connection 1 form}, we get
	\begin{align*}
		\Omega_{\alpha}{}^{\beta}=d(\omega_{\alpha}{}^{\beta}+iK_{\alpha}{}^{\beta}\theta+&i\delta_{\alpha}^{\beta}\sigma)+\frac{1}{2}\theta_{\alpha}\wedge K_{\gamma}{}^{\beta}\theta^{\gamma}+\frac{1}{2}K_{\alpha\bar{\delta}}\theta^{\bar{\delta}}\wedge \theta^{\beta}
		\\
		&-(\omega_{\alpha}{}^{\gamma}+iK_{\alpha}{}^{\gamma}\theta+i\delta_{\alpha}^{\gamma}\sigma)\wedge (\omega_{\gamma}{}^{\beta}+iK_{\gamma}{}^{\beta}\theta+i\delta_{\gamma}^{\beta}\sigma).
	\end{align*}
	We next compute the term $d\sigma$. For that, we use the normal coordinates at a given point $p\in M$. Then by \eqref{sigma in Fefferman metric} the following holds at $p$:
	\begin{align*}
		d\sigma=\frac{1}{n+2}\bigl( id\omega_{\alpha}{}^{\alpha}-\frac{1}{2(n+1)}Rd\theta\bigr).
	\end{align*}
	Note that at $p$ we have $d\omega_{\alpha}{}^{\beta}=-R_{\gamma\bar{\delta}\alpha}{}^{\beta}\theta^{\gamma}\wedge \theta^{\bar{\delta}}$. In particular, $d\omega_{\alpha}{}^{\alpha}=R_{\gamma\bar{\delta}}\theta^{\gamma}\wedge \theta^{\bar{\delta}}$. We can thus simplify $d\sigma$ into
	\begin{align*}
	d\sigma=\frac{1}{n+2}\bigl( iR_{\alpha\bar{\beta}} \theta^{\alpha}\wedge\theta^{\bar{\beta}}-\frac{1}{2(n+1)}iR g_{\alpha\bar{\beta}}\theta^{\alpha}\wedge \theta^{\bar{\beta}}\bigr)=iK_{\alpha\bar{\beta}}\theta^{\alpha}\wedge \theta^{\bar{\beta}}.
	\end{align*}


	By this expression of $d\sigma$, under the normal coordinates at $p$ we write the curvature form into
	\begin{align*}
		\Omega_{\alpha}{}^{\beta}=&d\omega_{\alpha}{}^{\beta}+iK_{\alpha}{}^{\beta} d\theta-\delta_{\alpha}^{\beta}K_{\gamma\bar{\delta}}\theta^{\gamma}\wedge\theta^{\bar{\delta}}+\frac{1}{2}\theta_{\alpha}\wedge K_{\gamma}{}^{\beta}\theta^{\gamma}+\frac{1}{2}K_{\alpha\bar{\delta}}\theta^{\bar{\delta}}\wedge \theta^{\beta} \quad \mod \theta
		\\
		=&\bigl(-R_{\gamma\bar{\delta}\alpha}{}^{\beta}-K_{\alpha}{}^{\beta}g_{\gamma\bar{\delta}}-\delta_{\alpha}^{\beta}K_{\gamma\bar{\delta}}-\frac{1}{2}g_{\alpha\bar{\delta}}K_{\gamma}{}^{\beta}-\frac{1}{2}K_{\alpha\bar{\delta}}\delta_{\gamma}^{\beta}\bigr) \theta^{\gamma}\wedge\theta^{\bar{\delta}} \quad \mod \theta.
	\end{align*}
	Thus, by \eqref{Fefferman space curvature}
	\begin{align*}
		\rho_{\gamma\bar{\delta}\alpha}{}^{\beta}=R_{\gamma\bar{\delta}\alpha}{}^{\beta}+K_{\alpha}{}^{\beta}g_{\gamma\bar{\delta}}+\delta_{\alpha}^{\beta}K_{\gamma\bar{\delta}}+\frac{1}{2}g_{\alpha\bar{\delta}}K_{\gamma}{}^{\beta}+\frac{1}{2}K_{\alpha\bar{\delta}}\delta_{\gamma}^{\beta}.
	\end{align*}
	We lower down the index $\beta$ to get the following equation. Here again recall on the two sides of the above equation the indices are raised up or lowered down by $\mathsf{g}_{ij}$ and $g_{\alpha\bar{\beta}}$ respectively.
	\begin{align*}
		\rho_{\gamma\bar{\delta}\alpha\bar{\beta}}=\frac{1}{2}R_{\gamma\bar{\delta}\alpha\bar{\beta}}+\frac{1}{2}K_{\alpha\bar{\beta}}g_{\gamma\bar{\delta}}+\frac{1}{2}g_{\alpha\bar{\beta}}K_{\gamma\bar{\delta}}+\frac{1}{4}g_{\alpha\bar{\delta}}K_{\gamma\bar{\beta}}+\frac{1}{4}K_{\alpha\bar{\delta}}g_{\gamma\bar{\beta}}.
	\end{align*}
	
	The corresponding Weyl curvature is given by
	\begin{align*}
		W_{\alpha\bar{\beta}\gamma\bar{\delta}}=&\rho_{\alpha\bar{\beta}\gamma\bar{\delta}}-\bigl(P_{\alpha\gamma}\mathsf{g}_{\bar{\beta}\bar{\delta}}+P_{\bar{\beta}\bar{\delta}}\mathsf{g}_{\alpha\gamma}-P_{\alpha\bar{\delta}}\mathsf{g}_{\bar{\beta}\gamma}-P_{\bar{\beta}\gamma}\mathsf{g}_{\alpha\bar{\delta}} \bigr)\\
		=&\frac{1}{2}R_{\alpha\bar{\beta}\gamma\bar{\delta}}+\frac{1}{2}K_{\alpha\bar{\beta}}g_{\gamma\bar{\delta}}+\frac{1}{2}g_{\alpha\bar{\beta}}K_{\gamma\bar{\delta}}+\frac{1}{2}g_{\alpha\bar{\delta}}K_{\gamma\bar{\beta}}+\frac{1}{2}K_{\alpha\bar{\delta}}g_{\gamma\bar{\beta}}.
	\end{align*}
	
	Next we compute the Riemannian curvature $\rho_{\alpha\beta\bar{\gamma}\bar{\delta}}$ and the Weyl curvature $W_{\alpha\beta\bar{\gamma}\bar{\delta}}$.	
	\begin{align*}
		\Omega_{\alpha}{}^{\bar{\beta}}=d\phi_{\alpha}{}^{\bar{\beta}}-\phi_{\alpha}{}^j\wedge\phi_j{}^{\bar{\beta}}=-\phi_{\alpha}{}^0\wedge\phi_0{}^{\bar{\beta}}-\phi_{\alpha}{}^{2n+1}\wedge\phi_{2n+1}{}^{\bar{\beta}}.
	\end{align*}
	Then by \eqref{connection 1 form}
	\begin{align*}
		\Omega_{\alpha}{}^{\bar{\beta}}=-\frac{1}{2}\theta_{\alpha}\wedge \sigma^{\bar{\beta}}-\frac{1}{2}\sigma_{\alpha}\wedge \theta^{\bar{\beta}}=-\frac{1}{2}g_{\alpha\bar{\gamma}}\theta^{\bar{\gamma}}\wedge K_{\bar{\delta}}{}^{\bar{\beta}} \theta^{\bar{\delta}}-\frac{1}{2}K_{\alpha\bar{\gamma}}\theta^{\bar{\gamma}}\wedge \theta^{\bar{\beta}}.
	\end{align*}
	By anti-symmetrizing the index $\bar{\gamma}$ and $\bar{\delta}$ on the right hand side, we get
	\begin{align*}
		\rho_{\bar{\gamma}\bar{\delta}\alpha}{}^{\bar{\beta}}=\frac{1}{2}g_{\alpha\bar{\gamma}}K_{\bar{\delta}}{}^{\bar{\beta}}+\frac{1}{2}K_{\alpha\bar{\gamma}}\delta_{\bar{\delta}}^{\bar{\beta}}-\frac{1}{2}g_{\alpha\bar{\delta}}K_{\bar{\gamma}}{}^{\bar{\beta}}-\frac{1}{2}K_{\alpha\bar{\delta}}\delta_{\bar{\gamma}}^{\bar{\beta}}.
	\end{align*}
	We lower down the index $\bar{\beta}$ to obtain
	\begin{align*}
		\rho_{\bar{\gamma}\bar{\delta}\alpha\beta}=\frac{1}{4}g_{\alpha\bar{\gamma}}K_{\beta\bar{\delta}}+\frac{1}{4}g_{\beta\bar{\delta}}K_{\alpha\bar{\gamma}}-\frac{1}{4}g_{\alpha\bar{\delta}}K_{\beta\bar{\gamma}}-\frac{1}{4}g_{\beta\bar{\gamma}}K_{\alpha\bar{\delta}}.
	\end{align*}
	Therefore, the corresponding Weyl curvature is
	\begin{align*}
		W_{\alpha\beta\bar{\gamma}\bar{\delta}}=\rho_{\alpha\beta\bar{\gamma}\bar{\delta}}-\bigl(P_{\alpha\bar{\gamma}}\mathsf{g}_{\beta\bar{\delta}}+P_{\beta\bar{\delta}}\mathsf{g}_{\alpha\bar{\gamma}}-P_{\alpha\bar{\delta}}\mathsf{g}_{\beta\bar{\gamma}}-P_{\beta\bar{\gamma}}\mathsf{g}_{\alpha\bar{\delta}} \bigr)=0.
	\end{align*}
	Finally, we compute $\rho_{2n+1\,\bar{\beta}\alpha\,2n+1}$ and $W_{2n+1\,\bar{\beta}\alpha\,2n+1}$. By \eqref{connection 1 form}, we have
	\begin{equation*}
		\Omega_{\alpha}{}^0=d\phi_{\alpha}{}^0-\phi_{\alpha}{}^j\wedge \phi_j{}^0=\frac{i}{2}d(g_{\alpha\bar{\beta}}\theta^{\bar{\beta}})-\bigl(\omega_{\alpha}{}^{\beta}+iK_{\alpha}{}^{\beta}\theta+i\delta_{\alpha}^{\beta}\sigma\bigr)\wedge \frac{i}{2}\theta_{\beta}.
	\end{equation*}	
	Since $d\theta^{\beta}=\theta^{\alpha}\wedge \omega_{\alpha}{}^{\beta}$ by \eqref{structure equation of CR manifold}, if we use normal coordinates at a given point, then it simplifies into
	\begin{equation*}
		\Omega_{\alpha}{}^0=-\bigl(iK_{\alpha}{}^{\beta}\theta+i\delta_{\alpha}^{\beta}\sigma\bigr)\wedge \frac{i}{2}g_{\beta\bar{\gamma}}\theta^{\bar{\gamma}}=\frac{1}{2}K_{\alpha\bar{\gamma}}\theta\wedge\theta^{\bar{\gamma}}+\frac{1}{2}g_{\alpha\bar{\gamma}}\sigma\wedge\theta^{\bar{\gamma}}.
	\end{equation*}
	Thus, $\rho_{2n+1\bar{\gamma}\alpha}{}^0=-\frac{1}{2} g_{\alpha\bar{\gamma}}$. If we lower down the index $0$, then $\rho_{2n+1\,\bar{\gamma}\alpha \,2n+1}=-\frac{1}{2} g_{\alpha\bar{\gamma}}$. Therefore, 
	\begin{align*}
		W_{2n+1\,\bar{\beta}\alpha \,2n+1}=&\rho_{2n+1\,\bar{\beta}\alpha \,2n+1}-\bigl(P_{2n+1\,\alpha}\mathsf{g}_{\bar{\beta}\,2n+1}+P_{\bar{\beta}\,2n+1}\mathsf{g}_{2n+1\,\alpha}-P_{2n+1\,2n+1}\mathsf{g}_{\bar{\beta}\alpha}-P_{\bar{\beta}\alpha}\mathsf{g}_{2n+1\,2n+1} \bigr)\\
		=&-\frac{1}{2}g_{\alpha\bar{\beta}}+\frac{1}{2}g_{\alpha\bar{\beta}}=0.
	\end{align*}
	By taking the conjugate, we also have $W_{2n+1\,\alpha\bar{\beta}\,2n+1}=0$.
\end{proof}
\subsection{The Bach tensor}
We next compute the Bach tensor $B_{ij}=\mathsf{g}^{kl}D_lC_{ijk}-P^{kl}W_{kijl}.$

\begin{prop}\label{Bach tensor prop}
	The Bach tensor of the Fefferman metric satisfies
	\begin{align*}
		B_{00}=&\frac{4}{n}\Delta\Lambda+8K_{\alpha}{}^{\beta}K_{\beta}{}^{\gamma}K_{\gamma}{}^{\alpha}-\frac{4}{n(n+1)}\Lambda R,
		\\
		B_{\alpha\bar{\beta}}=&-\frac{n-1}{(n+1)(n+2)^2}RR_{\alpha\bar{\beta}}+\frac{n-1}{(n+2)^2}R_{\alpha}{}^{\mu}R_{\mu\bar{\beta}}+\frac{n-1}{4(n+1)^2(n+2)^2}R^2 g_{\alpha\bar{\beta}}-\frac{n-1}{n}\Lambda g_{\alpha\bar{\beta}},
		\\
		B_{0\alpha}=&\frac{2n-2}{n}i\nab_{\alpha}\Lambda, \qquad B_{0\,2n+1}=0,
	\end{align*}
	where $\Delta=g^{\alpha\bar{\beta}}\nab_{\alpha}\nab_{\bar{\beta}}$ is the complex Laplacian on the base \k manifold $M$.
\end{prop}

\begin{proof}
	We begin with the computation of $B_{00}$. Note
	\begin{equation}\label{B00}
		B_{00}=\mathsf{g}^{kl}D_lC_{00k}-P^{kl}W_{k00l}.
	\end{equation}
	
	By \eqref{schouten tensor upper index} and the symmetry of Weyl tensor, the second term on the right hand side writes into
	\begin{align*}
	P^{kl}W_{k00l}=&P^{00}W_{0000}+P^{\alpha\bar{\beta}}W_{\alpha 00\bar{\beta}}+P^{\bar{\beta}\alpha}W_{\bar{\beta}00\alpha}+P^{2n+1\,2n+1}W_{2n+1\,00\,2n+1}\\
	=&2P^{\alpha\bar{\beta}} W_{\alpha 00\bar{\beta}}+P^{2n+1\,2n+1}W_{2n+1\,00\,2n+1}.
	\end{align*}
	
	We further use \eqref{weyl curvature 1}, \eqref{schouten tensor upper index} and \eqref{K trace} to write it into
	\begin{align*}
		P^{kl}W_{k00l}=&4K^{\alpha\bar{\beta}} \bigl(-\frac{1}{2} K_{\alpha}{}^{\gamma}K_{\gamma\bar{\beta}}+\frac{\Lambda}{2n} g_{\alpha\bar{\beta}}\bigr)
		=-2K_{\alpha}{}^{\gamma} K_{\gamma}{}^{\mu} K_{\mu}{}^{\alpha}+\frac{\Lambda}{n(n+1)} R.
	\end{align*}
	
	Let us also compute the first term on the right hand side of \eqref{B00}.
	\begin{equation}\label{B00 first term}
		\mathsf{g}^{kl}D_l C_{00k}= D_{2n+1}C_{000}+\mathsf{g}^{\alpha\bar{\beta}}D_{\bar{\beta}}C_{00\alpha}+\mathsf{g}^{\bar{\beta}\alpha}D_{\alpha}C_{00\bar{\beta}}+D_{0}C_{00\,2n+1}.
	\end{equation}
	We need to compute the covariant derivatives of the Cotton tensor on the right hand side of the above equality. For the first one,	
	\begin{align*}
		D_{2n+1}C_{000}=X_{2n+1}C_{000}-\Gamma_{2n+1\,0}^pC_{p00}-\Gamma_{2n+1\,0}^pC_{0p0}-\Gamma_{2n+1\,0}^pC_{00p}
	\end{align*}
	
	Since $C_{000}=0$ by Proposition \ref{cotton tensor prop} and $\Gamma_{2n+1\,0}^p=0$ for any $p$ by \eqref{christoffel symbols}, it follows that $D_{2n+1}C_{000}=0$.
	
For the second and third terms in \eqref{B00 first term}, we apply Proposition \ref{cotton tensor prop} and \eqref{christoffel symbols} to obtain
\begin{align*}
		D_{\bar{\beta}} C_{00\alpha}=&X_{\bar{\beta}} C_{00\alpha}-\Gamma_{\bar{\beta}0}^p C_{p0\alpha}-\Gamma_{\bar{\beta}0}^p C_{0p\alpha}-\Gamma_{\bar{\beta}\alpha}^p C_{00p}\\
		=&\frac{1}{n}X_{\bar{\beta}}X_{\alpha} \Lambda+\frac{3}{2}K_{\bar{\beta}}{}^{\bar{\gamma}} \bigl(K_{\alpha}{}^{\mu}K_{\mu\bar{\gamma}}-\frac{\Lambda}{n}g_{\alpha\bar{\gamma}} \bigr);
\end{align*}
\begin{align*}
		\mathsf{g}^{\alpha\bar{\beta}}D_{\bar{\beta}} C_{00\alpha}=&\frac{2}{n}g^{\alpha\bar{\beta}}X_{\bar{\beta}}X_{\alpha} \Lambda+3K^{\alpha\bar{\gamma}} \bigl(K_{\alpha}{}^{\mu}K_{\mu\bar{\gamma}}-\frac{\Lambda}{n}g_{\alpha\bar{\gamma}} \bigr)
		\\
		=&\frac{2}{n}\Delta\Lambda+3K_{\alpha}{}^{\beta}K_{\beta}{}^{\gamma}K_{\gamma}{}^{\alpha}-\frac{3}{2n(n+1)}\Lambda R.
\end{align*}
	
	Note that $\mathsf{g}^{\alpha\bar{\beta}}D_{\bar{\beta}} C_{00\alpha}$ is real. By taking the conjugate, we see $\mathsf{g}^{\bar{\beta}\alpha}D_{\alpha}C_{00\bar{\beta}}=\mathsf{g}^{\alpha\bar{\beta}}D_{\bar{\beta}} C_{00\alpha}$.
	
	For the last term in \eqref{B00 first term}, we have
	\begin{align*}
		D_0C_{00\,2n+1}=X_0C_{00\,2n+1}-\Gamma_{00}^{p}C_{p0\,2n+1}-\Gamma_{00}^{p}C_{0p\,2n+1}-\Gamma_{0\,2n+1}^{p}C_{00 p}.
	\end{align*}
	Since $C_{00\,2n+1}=0$ by Proposition \ref{cotton tensor prop} and $\Gamma_{00}^p=\Gamma_{0\,2n+1}^p=0$ for any $p$ by \eqref{christoffel symbols}, it follows that $D_0C_{00\,2n+1}=0$.
	
Bringing these results on Cotton tensor's derivatives into \eqref{B00 first term}, we get
	\begin{align*}
		\mathsf{g}^{kl}D_lC_{00k}=\frac{4}{n}\Delta\Lambda+6K_{\alpha}{}^{\beta}K_{\beta}{}^{\gamma}K_{\gamma}{}^{\alpha}-\frac{3}{n(n+1)}\Lambda R.
	\end{align*}
	Therefore,
	\begin{equation*}
		B_{00}=\frac{4}{n}\Delta\Lambda+8K_{\alpha}{}^{\beta}K_{\beta}{}^{\gamma}K_{\gamma}{}^{\alpha}-\frac{4}{n(n+1)}\Lambda R.
	\end{equation*}
	
	Next we will compute $B_{\alpha\bar{\beta}}$:
	\begin{align}\label{Bab}
		B_{\alpha\bar{\beta}}=\mathsf{g}^{kl}D_lC_{\alpha\bar{\beta}k}-P^{kl}W_{k\alpha\bar{\beta}l}.
	\end{align}
	The second term on the right hand side writes into
	\begin{align*}
		P^{kl}W_{k\alpha\bar{\beta}l}=P^{00}W_{0\alpha\bar{\beta}0}+P^{2n+1\,2n+1}W_{2n+1\alpha\bar{\beta}2n+1}+P^{\gamma\bar{\delta}}W_{\gamma\alpha\bar{\beta}\bar{\delta}}+P^{\bar{\delta}\gamma}W_{\bar{\delta}\alpha\bar{\beta}\gamma}.
	\end{align*}
	By \eqref{schouten tensor upper index}, \eqref{weyl curvature 1} and \eqref{weyl curvature 2}, we get
	\begin{align*}
	P^{kl}W_{k\alpha\bar{\beta}l}=&W_{\alpha 00\bar{\beta}}+P^{\bar{\delta}\gamma}W_{\bar{\delta}\alpha\bar{\beta}\gamma}
	\\
	=&-\frac{1}{2} K_{\alpha}{}^{\gamma}K_{\gamma\bar{\beta}}+\frac{\Lambda}{2n} g_{\alpha\bar{\beta}}+2K^{\gamma\bar{\delta}}\bigl(\frac{1}{2}R_{\alpha\bar{\delta}\gamma\bar{\beta}}+\frac{1}{2}K_{\alpha\bar{\beta}}g_{\gamma\bar{\delta}}+\frac{1}{2}K_{\gamma\bar{\delta}}g_{\alpha\bar{\beta}}+\frac{1}{2}K_{\gamma\bar{\beta}}g_{\alpha\bar{\delta}}+\frac{1}{2}K_{\alpha\bar{\delta}}g_{\gamma\bar{\beta}} \bigr).
	\end{align*}	
	
	We use \eqref{K trace} and the fact that $\Lambda=K_{\alpha\bar{\beta}}K^{\alpha\bar{\beta}}$ to simplify it into
	\begin{align}\label{Bab second term}
		P^{kl}W_{k\alpha\bar{\beta}l}	=K^{\gamma\bar{\delta}}R_{\alpha\bar{\beta}\gamma\bar{\delta}}+\frac{1}{2(n+1)}R K_{\alpha\bar{\beta}}+\frac{2n+1}{2n}\Lambda g_{\alpha\bar{\beta}}+\frac{3}{2}K_{\alpha}{}^{\gamma}K_{\gamma\bar{\beta}}.
	\end{align}
	We next compute the first term on the right hand side of \eqref{Bab}.
	\begin{align}\label{Bab first term}
		\mathsf{g}^{kl}D_lC_{\alpha\bar{\beta}k}=D_0C_{\alpha\bar{\beta}\,2n+1}+D_{2n+1}C_{\alpha\bar{\beta}0}+2g^{\gamma\bar{\delta}}D_{\bar{\delta}}C_{\alpha\bar{\beta}\gamma}+2g^{\bar{\delta}\gamma}D_{\gamma}C_{\alpha\bar{\beta}\bar{\delta}}.
	\end{align}
	Let us compute these covariant derivatives of the Cotton tensor. Since $C_{ijk}=0$ if at least one of the index is $2n+1$ by Proposition \ref{cotton tensor prop} and $\Gamma_{0\,2n+1}^p=0$ for any $p$ by \eqref{christoffel symbols},
	\begin{align*}
		D_0C_{\alpha\bar{\beta}\,2n+1}=-\Gamma_{0\,2n+1}^p C_{\alpha\bar{\beta}p}=0.
	\end{align*}
	
	For the second term in \eqref{Bab first term}, we have
	\begin{align*}
		D_{2n+1}C_{\alpha\bar{\beta}0}=&X_{2n+1}C_{\alpha\bar{\beta}0}-\Gamma_{2n+1\alpha}^pC_{p\bar{\beta}0}-\Gamma_{2n+1\bar{\beta}}^pC_{\alpha p0}-\Gamma_{2n+1\,0}^pC_{\alpha\bar{\beta}p}\\
		=&-\Gamma_{2n+1\alpha}^pC_{p\bar{\beta}0}-\Gamma_{2n+1\bar{\beta}}^pC_{\alpha p0}.
	\end{align*}
	By \eqref{christoffel symbols}, we get
	\begin{align*}
		D_{2n+1}C_{\alpha\bar{\beta}0}=-\Gamma_{2n+1\alpha}^{\mu}C_{\mu\bar{\beta}0}-\Gamma_{2n+1\bar{\beta}}^{\bar{\nu}}C_{\alpha \bar{\nu}0}=-i\delta_{\alpha}^{\mu}C_{\mu\bar{\beta}0}+i\delta_{\bar{\beta}}^{\bar{\nu}}C_{\alpha \bar{\nu}0}=0.	
	\end{align*}
	
	For the third term in \eqref{Bab first term}, we apply \eqref{christoffel symbols}, Remark \ref{identification of tensors on F and X rmk} and Proposition \ref{cotton tensor prop} to obtain
	\begin{align*}
		D_{\bar{\delta}}C_{\alpha\bar{\beta}\gamma}=&X_{\bar{\delta}}C_{\alpha\bar{\beta}\gamma}-\Gamma_{\bar{\delta}\bar{\beta}}^pC_{\alpha p\gamma}-\Gamma_{\bar{\delta}\alpha}^pC_{p\bar{\beta}\gamma}-\Gamma_{\bar{\delta}\gamma}^pC_{\alpha\bar{\beta}p}\\
		=&\bigl(X_{\bar{\delta}}C_{\alpha\bar{\beta}\gamma}-\widetilde{\Gamma}_{\bar{\delta}\bar{\beta}}^{\bar{\nu}}C_{\alpha \bar{\nu}\gamma}\bigr)-\Gamma_{\bar{\delta}\alpha}^0C_{0\bar{\beta}\gamma}-\Gamma_{\bar{\delta}\gamma}^0C_{\alpha\bar{\beta}0}\\
		=&\frac{1}{2}\nab_{\bar{\delta}}\nab_{\gamma}K_{\alpha\bar{\beta}}-\frac{1}{2}g_{\alpha\bar{\delta}} 	\bigl(K_{\gamma}{}^{\mu} K_{\mu\bar{\beta}}-\frac{\Lambda}{n} g_{\gamma\bar{\beta}} \bigr)  -\frac{1}{4}g_{\gamma\bar{\delta}}\bigl(K_{\alpha}{}^{\mu} K_{\mu\bar{\beta}}-\frac{\Lambda}{n} g_{\alpha\bar{\beta}} \bigr).
	\end{align*}
	Therefore,
	\begin{align*}
		2g^{\gamma\bar{\delta}}D_{\bar{\delta}}C_{\alpha\bar{\beta}\gamma}=&\nab^{\gamma}\nab_{\gamma}K_{\alpha\bar{\beta}}-\delta_{\alpha}^{\gamma}\bigl(K_{\gamma}{}^{\mu} K_{\mu\bar{\beta}}-\frac{\Lambda}{n} g_{\gamma\bar{\beta}} \bigr)  -\frac{n}{2}\bigl(K_{\alpha}{}^{\mu} K_{\mu\bar{\beta}}-\frac{\Lambda}{n} g_{\alpha\bar{\beta}} \bigr)
		\\
		=&\nab^{\gamma}\nab_{\gamma}K_{\alpha\bar{\beta}} -\frac{n+2}{2}\bigl(K_{\alpha}{}^{\mu} K_{\mu\bar{\beta}}-\frac{\Lambda}{n} g_{\alpha\bar{\beta}} \bigr).
	\end{align*}
	
	For the last term in \eqref{Bab first term}, we have
	\begin{align*}
		D_{\gamma}C_{\alpha\bar{\beta}\bar{\delta}}=&X_{\gamma}C_{\alpha\bar{\beta}\bar{\delta}}-\Gamma_{\gamma\alpha}^pC_{p\bar{\beta}\bar{\delta}}-\Gamma_{\gamma\bar{\beta}}^pC_{\alpha p\bar{\delta}}-\Gamma_{\gamma\bar{\delta}}^pC_{\alpha\bar{\beta}p}\\
		=&X_{\gamma}C_{\alpha\bar{\beta}\bar{\delta}}-\Gamma_{\gamma\alpha}^{\mu}C_{\mu\bar{\beta}\bar{\delta}}-\Gamma_{\gamma\bar{\beta}}^0C_{\alpha 0\bar{\delta}}-\Gamma_{\gamma\bar{\delta}}^0C_{\alpha\bar{\beta}0}.
	\end{align*}
	By Proposition \ref{cotton tensor prop} and \eqref{christoffel symbols}, it follows that
	\begin{align*}
		D_{\gamma}C_{\alpha\bar{\beta}\bar{\delta}}=-\frac{1}{4}g_{\gamma\bar{\beta}}\bigl(K_{\alpha}{}^{\mu}K_{\mu\bar{\delta}}-\frac{\Lambda}{n}g_{\alpha\bar{\delta}} \bigr)+\frac{1}{4}g_{\gamma\bar{\delta}}\bigl(K_{\alpha}{}^{\mu}K_{\mu\bar{\beta}}-\frac{\Lambda}{n}g_{\alpha\bar{\beta}} \bigr).
	\end{align*}
	Therefore,
	\begin{align*}
		2g^{\bar{\delta}\gamma}D_{\gamma}C_{\alpha\bar{\beta}\bar{\delta}}=-\frac{1}{2}\delta_{\bar{\beta}}^{\bar{\delta}}\bigl(K_{\alpha}{}^{\mu}K_{\mu\bar{\delta}}-\frac{\Lambda}{n}g_{\alpha\bar{\delta}} \bigr)+\frac{n}{2}\bigl(K_{\alpha}{}^{\mu}K_{\mu\bar{\beta}}-\frac{\Lambda}{n}g_{\alpha\bar{\beta}} \bigr)=\frac{n-1}{2}\bigl(K_{\alpha}{}^{\mu}K_{\mu\bar{\beta}}-\frac{\Lambda}{n}g_{\alpha\bar{\beta}} \bigr).
	\end{align*}
We put the formulas on covariant derivatives of the Cotton tensor into \eqref{Bab first term} and obtain
	\begin{align}\label{DlCabk}
		\mathsf{g}^{kl}D_lC_{\alpha\bar{\beta}k}=\nab^{\gamma}\nab_{\gamma}K_{\alpha\bar{\beta}}-\frac{3}{2}\bigl(K_{\alpha}{}^{\mu}K_{\mu\bar{\beta}}-\frac{\Lambda}{n}g_{\alpha\bar{\beta}} \bigr).
	\end{align}
	By putting \eqref{Bab second term} and \eqref{DlCabk} into \eqref{Bab}, we have
	\begin{align*}
		B_{\alpha\bar{\beta}}=\nab^{\gamma}\nab_{\gamma}K_{\alpha\bar{\beta}}-K^{\gamma\bar{\delta}}R_{\alpha\bar{\beta}\gamma\bar{\delta}}-\frac{1}{2(n+1)}R K_{\alpha\bar{\beta}}-\frac{n-1}{n}\Lambda g_{\alpha\bar{\beta}}-3K_{\alpha}{}^{\gamma}K_{\gamma\bar{\beta}}.
	\end{align*}
	By \eqref{K square} and \eqref{K laplace}, we finally get
	\begin{align*}
		B_{\alpha\bar{\beta}}=&\frac{R}{2(n+1)(n+2)}g^{\gamma\bar{\delta}}R_{\alpha\bar{\beta}\gamma\bar{\delta}}+\frac{1}{n+2}R_{\alpha}{}^{\mu}R_{\mu\bar{\beta}}-\frac{1}{2(n+1)}R K_{\alpha\bar{\beta}}-\frac{n-1}{n}\Lambda g_{\alpha\bar{\beta}}-3K_{\alpha}{}^{\gamma}K_{\gamma\bar{\beta}}
		\\=&-\frac{n-1}{(n+1)(n+2)^2}RR_{\alpha\bar{\beta}}+\frac{n-1}{(n+2)^2}R_{\alpha}{}^{\mu}R_{\mu\bar{\beta}}+\frac{n-1}{4(n+1)^2(n+2)^2}R^2 g_{\alpha\bar{\beta}}-\frac{n-1}{n}\Lambda g_{\alpha\bar{\beta}}.
	\end{align*}

	The last two equations of Proposition \ref{Bach tensor prop} can be proved similarly and we omit the proof.
\end{proof}	

\begin{cor}\label{Bach tensor n=2}
	When $n=2$, we have
	\begin{align*}
	&B_{00}=2\Delta\Lambda+4\bigl(\frac{1}{3}\Lambda R-\frac{1}{6^3} R^3 \bigr)
	\\
	&B_{\alpha\bar{\beta}}\in\mathcal{T}, \qquad
	B_{0\alpha}=i\nab_{\alpha}\Lambda, \qquad B_{0\,2n+1}=0.
	\end{align*}
\end{cor}

\begin{proof}
	Note that the last three identities follow immediately from Proposition \ref{Bach tensor prop}. It remains to prove the first one on $B_{00}$. Fix a point $p\in M$ and we use the normal coordinates at $p$. In particular \eqref{normal coordinates} holds at $p$, i.e., $(g_{\alpha\bar{\beta}})$ becomes the identity matrix and $(R_{\alpha\bar{\beta}})$ becomes $\mathrm{diag}(\lambda_1, \lambda_2)$ at $p$. Since $K_{\alpha\bar{\beta}}=\frac{1}{4}\bigl(R_{\alpha\bar{\beta}}-\frac{1}{6}Rg_{\alpha\bar{\beta}} \bigr)$, we have the following two equations hold at $p$.
	\begin{align*}
		(K_{\alpha\bar{\beta}})=&\begin{pmatrix}
		\frac{5\lambda_1-\lambda_2}{24} & 0\\
		0 & \frac{5\lambda_2-\lambda_1}{24}\\
		\end{pmatrix}:=\begin{pmatrix}
		\mu_1 & 0\\
		0 & \mu_2\\
		\end{pmatrix};\\
		K_{\alpha}{}^{\beta}K_{\beta}{}^{\gamma}K_{\gamma}{}^{\alpha}=&\mu_1^3+\mu_2^3
		=\frac{1}{2} (\mu_1+\mu_2) \bigl(3\mu_1^2+3\mu_2^2-(\mu_1+\mu_2)^2 \bigr).
	\end{align*}
	Note that $\mu_1+\mu_2=\frac{1}{6}R$ and $\mu_1^2+\mu_2^2=\Lambda$. Therefore,
	\begin{align*}
		K_{\alpha}{}^{\beta}K_{\beta}{}^{\gamma}K_{\gamma}{}^{\alpha}=\frac{1}{12}R \bigl(3\Lambda-\frac{1}{36}R^2 \bigr)=\frac{1}{4}\Lambda R-\frac{1}{2^4\cdot 3^3} R^3.
	\end{align*}
	
Finally bringing the above identity into the first equation of Proposition \ref{Bach tensor prop}, we get
\begin{align*}
		B_{00}=2\Delta\Lambda+8\bigl(\frac{1}{4}\Lambda R-\frac{1}{2^4\cdot 3^3} R^3\bigr)-\frac{2}{3}\Lambda R
		=2\Delta \Lambda+\frac{4}{3}R\Lambda-\frac{1}{2\cdot 3^3} R^3.
\end{align*}
\end{proof}

\bibliographystyle{plain}
\bibliography{references}

\end{document}